\documentclass[reqno, 12pt]{amsart}
\usepackage{array}
\usepackage{amsmath}
\usepackage{amsfonts}
\usepackage{amssymb}
\usepackage{mathrsfs}
\usepackage{enumerate}
\usepackage{amsthm}
\usepackage{verbatim}
\usepackage{amsmath, amscd}
\usepackage{caption}

\usepackage[usenames,dvipsnames]{color}
\usepackage{bm}
\usepackage{xy}
\xyoption{all}
\usepackage{color}
\usepackage{marvosym}
\usepackage{units}
\usepackage{amsbsy}
\usepackage{hyperref}
\usepackage{tikz}
\usepackage{tikz-cd}
\usepackage{marginnote}
\usetikzlibrary{matrix,arrows,backgrounds}%, decorations, pathmorphing, positioning, fit}

\newtheorem{theorem}{Theorem}[section]

\newtheorem{lemma}[theorem]{Lemma}
\newtheorem{proposition}[theorem]{Proposition}
\newtheorem{corollary}[theorem]{Corollary}

\newtheorem{conjecture}[theorem]{Conjecture}

\newtheorem{convention}[theorem]{Convention}
\newtheorem*{theorem*}{Theorem}
\newtheorem*{conjecture*}{Conjecture}

\theoremstyle{definition}
\newtheorem{definition}[theorem]{Definition}

\newtheorem{remark}[theorem]{Remark}

\newtheorem{example}[theorem]{Example}

\newcommand{\op}[1]{\operatorname{#1}}

\newcommand{\newterm}{\textsf}

\newcommand{\dbcoh}[1]{\operatorname{D}^{\operatorname{b}}(\operatorname{coh }#1)}

\newcommand{\dabsfact}[1]{\operatorname{D}^{\operatorname{abs}}[#1]}
\newcommand{\dabsFact}[1]{\operatorname{D}^{\operatorname{abs}}[\mathsf{Fact} \ #1]}

\newcommand{\dabs}{\op{D}^{\op{abs}}}
\newcommand{\gm}{\mathbb{G}_m}
\newcommand{\E}{\mathcal{E}}

\newcommand{\bigslant}[2]{{\raisebox{.2em}{$#1$}\left/\raisebox{-.2em}{$#2$}\right.}}

\def\N{\op{\mathbb{N}}}
\def\Z{\op{\mathbb{Z}}}
\def\C{\op{\mathbb{C}}}
\def\R{\op{\mathbb{R}}}
\def\Q{\op{\mathbb{Q}}}
\def\F{\op{\mathcal{F}}}

\def\O{\op{\mathcal{O}}}
\def\A{\op{\mathbb{A}}}
\def\P{\op{\mathbf{P}}}

\def\T{\op{\mathcal{T}}}

\def\tif{\text{if } }

\def\L{\mathop{\mathcal{L}}}

\title[Fractional CY Categories from LG Models]{Fractional Calabi-Yau Categories from Landau-Ginzburg Models}

\author[Favero]{David Favero}
\address{
  \begin{tabular}{l}
   David Favero \\
   \hspace{.1in} University of Alberta, Department of Mathematical and Statistical Sciences \\
   \hspace{.1in} Central Academic Building 632, Edmonton, AB, Canada T6G 2C7 \\
         \hspace{.1in} Korean Institute for Advanced Study \\
   \hspace{.1in} 85 Hoegiro, Dongdaemun-gu, Seoul, Republic of Korea 02455 \\
   \hspace{.1in} Email: {\bf favero@ualberta.ca} \\
  \end{tabular}
}

\author[Kelly]{Tyler L. Kelly}
\address{
  \begin{tabular}{l}
   Tyler L. Kelly \\
   \hspace{.1in} University of Cambridge, Department of Pure Mathematics and Mathematical \\ \hspace{.1in} Statistics,  Wilberforce Road, Cambridge, United Kingdom CB3 0WB \\
   \hspace{.1in} Email: {\bf tlk20@dpmms.cam.ac.uk} \\
  \end{tabular}
}

\numberwithin{equation}{section}
\begin{document}

\begin{abstract}
We give criteria for the existence of a Serre functor on the derived category of a gauged Landau-Ginzburg model.  This is used to provide a general theorem on the existence of an admissible (fractional) Calabi-Yau subcategory of a gauged Landau-Ginzburg model and a geometric context for crepant categorical resolutions.  We explicitly describe our framework in the toric setting.  As a consequence, we generalize several theorems and examples of Orlov and Kuznetsov, ending with new examples of semi-orthogonal decompositions containing (fractional) Calabi-Yau categories.
\end{abstract}

\maketitle
\setcounter{tocdepth}{1}
\tableofcontents

 \section{Introduction} 
 
In \cite{BK90}, Bondal and Kapranov generalized Serre duality to triangulated categories, providing an arbitrary $k$-linear triangulated category with a sense of a canonical bundle.
\begin{definition}
A \newterm{Serre functor} on a $k$-linear triangulated category $\mathcal T$ is an exact auto-equivalence
\[
S : \mathcal T \to \mathcal T
\]
such that there exists bifunctorial isomorphisms
\[
\op{Hom}(A,B) \cong \op{Hom}(B, S(A))^\vee.
\]
The category $\mathcal T$ is called \newterm{Calabi-Yau} (CY) of dimension $d$ if $S = [d]$. $\mathcal T$ is called \newterm{fractional Calabi-Yau} (FCY) of dimension $\frac{a}{b}$ if $S^b = [a]$.  
\end{definition}

The term Serre functor is inspired by the case where $\mathcal T$ is the bounded derived category of coherent sheaves $\dbcoh{X}$ for a smooth projective variety $X$ of dimension $n$. In this case, the Serre functor is a rephrasing of Serre duality, hence
$$
S = - \otimes \omega_X [n].
$$
In particular, the derived category of a Calabi-Yau variety of dimension $n$ is a Calabi-Yau category of dimension $n$ as the canonical bundle is trivial. Similarly, if the canonical bundle of $X$ is torsion, then $\dbcoh{X}$ is fractional Calabi-Yau.

Kuznetsov showed that fractional Calabi-Yau categories also occur as admissible subcategories of $\dbcoh{X}$ when $X \subseteq \P^n$ is a smooth hypersurface of degree $d \leq n+1$ (see, e.g., Corollary 4.4 of \cite{Kuz04}).  He first defined the admissible subcategory
\[
\mathcal A_X := \{ C \in \dbcoh{X} \ | \ \op{H}^j(C(i)) = 0 \text{ for } 0 \leq i \leq n-d \text{ and all } j \}.
\]
and then proved directly that $\mathcal A_X$ is FCY of dimension $\frac{n+1}{d}$ and CY when $d$ divides $n+1$.

In the special case of cubic fourfolds ($n=5, d=3$) we get a 2-dimensional Calabi-Yau category.  Kuznetsov went on to show that, for most of the known rational cubic fourfolds, $\mathcal A_X$ is equivalent to the derived category of a $K3$ surface. He conjectured that a smooth cubic fourfold $X$ is rational if and only if there is a $K3$ surface $Y$ and an equivalence of categories,
\[
\mathcal A_X \cong \dbcoh{Y}.
\]
\noindent This conjecture has steered the study of rational cubic fourfolds ever since.

Orlov later provided a beautiful description of $\mathcal A_X$ in terms of the categorical analogue of the Landau-Ginzburg model corresponding to the hypersurface. 
Let $X$ be a smooth projective hypersurface defined by the function $w$. There is an equivalence of categories,
\[
\mathcal A_X \cong  \op{D}^{\op{abs}}[\A^{n+1},\gm,w].
\]
The category $\op{D}^{\op{abs}}[\A^{n+1},\gm,w]$ can be loosely defined as the derived category associated to the gauged Landau-Ginzburg model $(\A^{n+1}, \gm, w)$. Here, $w$ is a section of the equivariant bundle $\O(\chi)$ for the $d^{\op{th}}$ power character $\chi$ (see Subsection~\ref{sec: fact definition} for a precise definition). This description has the advantage of being a geometric description of $\mathcal A_X$. 

Orlov's description of $\mathcal A_X$ gives rise to two leading questions.  
\begin{itemize}
\item When is the derived category of a Landau-Ginzburg model (fractional) Calabi-Yau?
\item  When do derived categories of Landau-Ginzburg models that are fractional Calabi-Yau appear as admissible subcategories of $\dbcoh{X}$?
\end{itemize}
In this paper, we give sufficient criteria for these questions.  

By studying the derived category of Landau-Ginzburg models, we give an alternate view of identifying (fractional) Calabi-Yau categories than that given by Kuznetsov in \cite{Kuz15}. There, Kuznetsov provides examples of (fractional) Calabi-Yau categories for a smooth variety $X$ by finding a spherical functor $\Phi: \dbcoh{X} \rightarrow \dbcoh{M}$ to another variety $M$ whose derived category comes equipped with a Lefschetz fibration. He provides a list of examples in Subsection 4.5 of \cite{Kuz15}. Many of his examples come from complete intersections in homogeneous varieties. 

In our viewpoint, we pass to the Landau-Ginzburg model and use geometric invariant theory to find a GIT chamber that is associated to a Calabi-Yau category instead of using a spherical functor. Due to this difference, our theories work in different contexts. For example, our framework is quite concrete for complete intersections in toric varieties, while \cite{Kuz15} naturally recovers many of the examples given in \cite{IM15} of complete intersections in homogeneous spaces.

\subsection{General Results}
First, we establish the Serre functor for the derived category of a Landau-Ginzburg model:

\begin{theorem}[=Theorem~\ref{SerreFunctorDescription}]\label{IntroSerre}
Let $X$ be a smooth algebraic variety and $G$ be a linearly reductive algebraic group acting on $X$. Let $\chi: G \to \gm$ be a character and $w \in \Gamma(X, \O_X(\chi))^G$.   Assume that $[X/ \op{ker }\chi]$ has finite diagonal.  In addition, assume that $\partial w \subseteq Z(w) $, and that $[\partial w / \op{ker } \chi]$ is proper.  Then,
\[
  \dabsfact{X,G,w}
\]
admits a Serre functor given by 
\[
S := (- \otimes \omega_X)[\op{dim } X - \op{dim } G +1].
\]
\end{theorem}

The relevance of the theorem above in the context of the literature is the presence of a $G$-action.  For example, Serre-Grothendieck duality without the presence of a $G$-action was proven by Efimov-Positselski \cite{EP15}.  In addition, existence of a Serre functor follows from a dg-enhancement which is smooth and proper \cite{Shl} (which we will use).  The existence of such a dg enhancement in the case  where $X$ is affine space, $G = 1$ and $\partial w$ is isolated was proven by Dyckerhoff \cite{Dyc}.  Lin-Pomerleano \cite{LP} subsequently exhibited a smooth and proper dg enhancement  for any smooth variety $X$ in the case $G=1$ and $\partial w$ is proper.  Furthermore when $X$ is Calabi-Yau, they demonstrated that the category is as well.  Preygel independently proved similar results for matrix factorizations on derived schemes, using a different set of tools from derived algebraic geometry \cite{Pre}.  The $G$-equivariant case was first studied by independently by Polishchuk-Vaintrob and Ballard, Katzarkov, and the first named author \cite{PV, BFK11}.  These provide a suitable dg-enhancement in the $G$-invariant case which we will rely heavily on.  

The theorem above also has the following corollary which provides a sufficient criteria for when the derived category of a Landau-Ginzburg model is (fractional) Calabi-Yau: 

\begin{theorem}[= Corollary~\ref{cor: torsion canonical fractional calabi-yau}]\label{theorem:GeneralLG}
Let $X$ be a smooth algebraic variety and $G$ be a linearly reductive algebraic group acting on $X$. Let $\chi: G \to \gm$ be a character and $w \in \Gamma(X, \O_X(\chi))^G$.   Assume that $[X/ \op{ker }\chi]$ has finite diagonal and torsion canonical bundle.  In addition, assume that the critical locus $\partial w$ is contained in $Z(w)$, and that $[\partial w / \op{ker } \chi]$ is proper.  Then,
\[
  \dabsfact{X,G,w}
\]
is fractional Calabi-Yau.
\end{theorem}

Second, we can use the birational geometry of Landau-Ginzburg models to attack the latter problem.
\begin{definition}
 Let $(Y_1, w_1)$ and $(Y_2, w_2)$ be two gauged LG-models with $Y_i$ smooth, $G$ acting on $Y_i$, and $w_i$ a section of $\O(\chi)$ for a character $\chi: G \to \gm$. We say that $(Y_1, w_1)$ \textbf{$K$-dominates} $(Y_2, w_2)$ if there exists a smooth $G$-variety, $Z$, and proper equivariant birational morphisms, $f_1: Z \to Y_1$ and $f_2: Z \to Y_2$, such that 
 \begin{itemize}
%  \item $f_1^*L_1 \cong f_2^*L_2$,
  \item $f_1^*w_1 = f_2^*w_2$ 
  \item $f_1^*K_{Y_1} - f_2^*K_{Y_2} \geq 0$.
 \end{itemize}
\end{definition}

In the context of finding FCY admissible subcategories, Kawamata's LG-model conjecture (see Conjecture 4.3.7 of \cite{BFK11}) specializes to the following. 

\begin{conjecture} \label{conj: LGKawamata}
 If $(Y_1,w_1)$ $K$-dominates $(Y_2,w_2)$ and $[Y_2 / \op{ker }\chi_2]$ has torsion canonical bundle,  then 
$\dabsfact{Y_2,G,w_2}$ is a FCY admissible subcategory of $
  \dabsfact{Y_1,G,w_1}$.
\end{conjecture}

While general birational relationships like $K$-dominance are more difficult to analyze, W\l{}odarcykz's weak factorization theorem \cite{Wlo03} shows that all birational transformations can be broken up into a sequence of simpler ones called elementary wall crossings (see Definition~\ref{def: EWC}).  These transformations were shown to yield fully-faithful functors between derived categories of gauged LG-models in \cite{BFK12}. An immediate consequence of Theorem~\ref{cor: torsion canonical fractional calabi-yau} and \cite{BFK12} is the following.

\begin{corollary} [= Corollary~\ref{cor: DHT}]\label{thm: DHT}
Conjecture~\ref{conj: LGKawamata} holds for elementary wall crossings.
\end{corollary}

\subsection{Toric Results}

In Section~\ref{sec:3}, we specialize to the toric situation and the description we find is quite pleasing.  Indeed, we obtain comparisons between FCY categories and derived categories of toric gauged Landau-Ginzburg models very similar to those found in \cite{Orl06}.

Let us begin by describing the toric backdrop.  Let $M$ and $N$ be dual lattices. Let $\nu = \{v_1, \ldots, v_n\}\subset N$ be a collection of distinct primitive lattice points. Consider the cone $\sigma : = \op{Cone}(\nu)$. We say $\sigma$ is $\Q$-Gorenstein (resp. almost Gorenstein) if there exists an element $\mathfrak{m} \in M_{\Q}$ (resp. $\mathfrak{m} \in M$) so that the cone $\sigma$ is generated over $\Q$ by finitely many lattice points $\{n \in N \ | \ \langle \mathfrak{m}, n \rangle = 1\}$. We partition the set $\nu$ as $\nu_{=1} \cup A$ where $\nu_{=1} = \{ v_i \in \nu \ | \ \langle \mathfrak{m}, n\rangle = 1\}$ and $A$ is its complement in $\nu$. 

Associate a group $S_\nu$ to the point collection $\nu$ in the following way. Consider the right exact sequence 
\begin{equation}\begin{aligned}
M &\stackrel{f_\nu}{\rightarrow} \Z^n \stackrel{\pi}{\rightarrow} \op{coker}(f_\nu) \rightarrow 0, \\
m &\mapsto \sum_{i=1}^n \langle m, v_i\rangle e_i. 
\end{aligned}\end{equation}
Set $S_\nu := \op{Hom}(\op{coker}(f_\nu), \gm).$ If we apply $\op{Hom}(-, \gm)$ to the above sequence, we obtain
$$
1 \rightarrow S_\nu \stackrel{\hat\pi}{\rightarrow} \gm^n \stackrel{\hat{f_\nu}}{\rightarrow} M \otimes \gm.
$$
This defines an action of $S_\nu$ on $\mathbb{A}^n$ by first taking the inclusion of $S_\nu$ into the maximal torus $\gm^n$ given by the map $\hat \pi$ and then extending the action naturally. 

Given a subset $R\subset \nu$, we can also define a $\gm$-action called \emph{R-charge} to act on $\A^{\nu}$ by extending the action
$$
\lambda \cdot (x_1, \ldots, x_n) = (y_1, \ldots, y_n), \text{ where } y_i : =   \begin{cases}
\lambda x_i & \text{ if } v_i \in R \\
x_i &  \text{ if } v_i \notin R.
\end{cases}
$$

Given the cone $\sigma :=\op{Cone}(\nu)$ we can consider the dual cone 
$$
\sigma^\vee : = \{ m \in M_{\R} \ | \ \langle m, n\rangle \geq 0 \text{ for all } n \in \sigma\}.
$$
Define a superpotential $w$ on $\mathbb{A}^n$ given by taking a finite set $\Xi \subset \sigma^\vee \cap M$ and defining $w$ to be
$$
w = \sum_{m \in \Xi} c_m x^m, \text{ where } x^m := \prod_{i=1}^n x_i^{\langle m, v_i\rangle}.
$$
Let $\tilde{\Sigma}$ be any simplicial fan such that $\tilde{\Sigma}(1) = \nu$. The quotient construction of a toric variety determines an open set $U_{\tilde{\Sigma}}$ of $\A^{\nu}$ (for a precise treatment see Equation~\eqref{CoxFan} for the fan associated to this open set). The triplet of data 
$$(U_{\tilde\Sigma}, S_\nu \times \gm, w)$$
constitutes a toric gauged Landau-Ginzburg model.

Define another gauged Landau-Ginzburg model associated to $\nu$. Take any simplicial fan $\Sigma$ so that $\Sigma(1) \subseteq \nu_{=1}$ and $\op{Cone}(\Sigma(1)) = \sigma$.  We can define a group $H$ which depends on $\nu$, $\Sigma(1)$, and $R$ (see Equation~\eqref{definemapforH} below) that acts on $U_{\Sigma}$. Consider the action of $H$ on the open affine set $U_{\Sigma} \subset \mathbb{A}^{\Sigma(1)}$. Finally, construct a potential $\bar w$ by just taking
$$
\bar{w} = \sum_{m \in \Xi} c_m \bar{x}^m, \text{ where } \bar{x}^m := \prod_{i, \op{Cone}(v_i)\subseteq \Sigma(1)}^n x_i^{\langle m, v_i\rangle}.
$$
There is another gauged Landau-Ginzburg model that comes from the triplet of data
$$
(U_{\Sigma}, H, \bar w).
$$

We prove a Orlov-type theorem that compares the derived categories associated to these two gauged Landau-Ginzburg models.

\begin{theorem}[=Theorem~\ref{FFLGmodels}]
Let $\tilde{\Sigma}$ be any simplicial fan such that $\tilde{\Sigma}(1) = \nu$ and $X_{\tilde{\Sigma}}$ is semiprojective.  Similarly, let $\Sigma$ be any simplicial fan such that $\Sigma(1) \subseteq \nu_{=1}$, $X_{\Sigma}$ is semiprojective, and $\op{Cone}(\Sigma(1)) = \sigma$. We have the following:
\begin{enumerate}
\item If $\langle \mathfrak{m}, a \rangle > 1$ for all $a \in \nu_{\neq 1}$, then there is a fully-faithful functor,
\[
\op{D}^{\op{abs}}[U_{\Sigma},   H, \bar w] \longrightarrow \op{D}^{\op{abs}}[U_{\tilde{\Sigma}},   S_\nu \times \gm, w].
\]
\item If $\langle \mathfrak{m}, a \rangle < 1$ for all $a \in \nu_{\neq 1}$, then there is a fully-faithful functor,
\[
 \op{D}^{\op{abs}}[U_{\tilde{\Sigma}},   S_\nu \times \gm, w]  \longrightarrow \op{D}^{\op{abs}}[U_{\Sigma},   H, \bar w].
 \]
\item If $\nu_{\neq 1} = \emptyset$, then there is an equivalence,
\[
\op{D}^{\op{abs}}[U_{\Sigma},   H, \bar w] \cong \op{D}^{\op{abs}}[U_{\tilde{\Sigma}},   S_\nu \times \gm, w].
\]
\end{enumerate}
Furthermore, if $\partial \bar w \subseteq Z(\bar w)$ and $[\partial \bar w / S_{\Sigma(1)}]$ is proper, then $\op{D}^{\op{abs}}[U_{\Sigma},   H, \bar w]$ is fractional Calabi-Yau.  If, in addition, $\sigma$ is almost Gorenstein, then $\op{D}^{\op{abs}}[U_{\Sigma},   H, \bar w]$ is Calabi-Yau.
\label{thm: main thm}
\end{theorem}

These types of relationships are intimately related to the Aspinwall and Plesser's formulation of mirror pairs \cite{AP15}. In particular, a corollary of this theorem is that if one considers a gauged linear $\sigma$-model in their setting that is nonsingular, then it has an associated Calabi-Yau category. 

For certain simplicial fans $\Sigma$, $\tilde \Sigma$, the categories $ \op{D}^{\op{abs}}[U_{\tilde{\Sigma}},   S_\nu \times \gm, w] $ and $\op{D}^{\op{abs}}[U_{\Sigma},   H, \bar w]$ may be geometric, i.e., equivalent to the derived category of some stack. One way to realize these equivalences is via the following setup. Suppose that:
\begin{itemize}
\item $\tilde \Sigma$ is a fan where the toric stack $\mathcal{X}_{\tilde \Sigma}$ is the total space of a vector bundle
$$
\mathcal{X}_\Sigma = \op{tot}\left(\bigoplus_{i=1}^t \O_{\mathcal{X}_{\Psi}}(\chi_{D_i})\right),
$$
where $\Psi$ is some fan corresponding to a semiprojective toric stack $\mathcal{X}_{\Psi}$ and $D_i$ are $\Q$-Cartier anti-nef divisors on $\mathcal{X}_{\Psi}$;
\item the $\gm$ action acts by fiberwise dilation on the total space of the vector bundle; and
\item the potential $w$ is of the form
$$
w = u_1 f_1 + \ldots + u_tf_t,
$$
where $f_i \in \Gamma(\mathcal{X}_{\Psi}, \O_{\mathcal{X}_{\Psi}}(D_i))$ and $u_i$ is the coordinate corresponding to the ray in $\Psi$ associated to the construction of the line bundle $\O_{\mathcal{X}_{\Psi}}(\chi_{D_i})$.
\end{itemize}

In this case, one can consider the complete intersection 
$$
\mathcal{Z} = Z(f_1, \ldots, f_t) \subset \mathcal{X}_{\Psi}.
$$
A result of Hirano (Proposition 4.8 of \cite{Hirano}, repeated here as Theorem~\ref{Hirano} for convenience) provides an equivalence amongst the derived category of coherent sheaves of $\mathcal{Z}$ and the factorization category $ \op{D}^{\op{abs}}[U_{\tilde{\Sigma}},   S_\nu \times \gm, w] $. By requiring the data $\Sigma, H$ and $\bar w$ to satisfy the analogous criteria, one has a different complete intersection $\mathcal{Z'} \subseteq \mathcal{X}_{\Upsilon}$ in some other toric stack $\mathcal{X}_{\Upsilon}$ associated to a fan $\Upsilon$. Thus under appropriate conditions, one or both of the relevant categories can be made geometric. In this case we get Corollary~\ref{thm:comparegeometric} which relates the derived categories of the stacks $\mathcal{Z}$ and $\mathcal{Z'}$. For a precise explanation of these conditions and results, refer to Subsection~\ref{subsec: toric CI}.

Almost immediately, we start to recover many theorems and examples as corollaries to our framework. For example, the Batyrev-Nill conjecture, Conjecture 5.3 of \cite{BN07}, is just Case (c) of Corollary~\ref{thm:comparegeometric}. This recovers the main result of \cite{FK14}. As an instructive example, we specialize to the case of Orlov's theorem on the fan associated to the line bundle $\op{tot}(\O_{\P^n}(-d))$, which we do as an example in Subsection~\ref{subsec: Orlov}.

\subsection{Crepant Categorical Resolutions}

 In \cite{Kuz08}, Kuznetsov studies the derived category of coherent sheaves on a singular variety $Y$. He constructs a subcategory $\tilde{\mathcal{D}}$ of the derived category $\dbcoh{\tilde Y}$ of coherent sheaves on a resolution $\tilde Y$ of $Y$ that he views as a categorical resolution of $\dbcoh{Y}$.  In Section~\ref{sec: CCR}, we provide an interpretation of crepant categorical resolutions in terms of Landau-Ginzburg models. 

In sum, crepant categorical resolutions have a simple geometric interpretation as partial compactifications of LG models.  
Roughly speaking, if the singular locus of $w$ is not proper, we can make it proper by partially compactifying.  This, in turn, provides a crepant categorical resolution. 

 Specifically, one finds a $G$-equivariant variety $U$ so that $V$ is openly immersed in $U$ and the function $w$ extends to $V$ so that $\dabsfact{U,w,G}$ is smooth and proper.  In other words, $[U/G]$ is a partial compactification of $[V/G]$ which has the benefit of satisfying the criteria of  Corollary~\ref{theorem:GeneralLG}.
Hence, to obtain crepant categorical resolutions for singular complete intersections $\mathcal Z$ in $X$, first, apply Hirano's result to replace $\mathcal Z$ by an LG model $(V,w,G)$ 
\[
\dbcoh{\mathcal Z} = \dabsfact{V,w,G}.
\]
Second, find a $G$-equivariant compactification $(U,w,G)$ satisfying the conditions of Corollary~\ref{theorem:GeneralLG}.  
In the examples below, such a $G$-equivariant compactification can be found by performing birational operations on the total space of a vector bundle on a blow-up on $X$.

Furthermore, we show that this geometric interpretation of crepant categorical resolutions behaves well with respect to full subcategories coming from VGIT.   Namely, if $\mathcal A$ is a full subcategory of $\dabsfact{U,w,G}$ obtained from an elementary wall crossing, then there is a corresponding elementary wall crossing of $\dabsfact{V,w,G}$ and the corresponding subcategory $\tilde{\mathcal A}$ is a crepant categorical resolution of $\mathcal A$ (Theorem~\ref{theorem: CCRGIT}).

\subsection{Examples}
Our results on fractional CY subcategories and crepant categorical resolutions can be used to generalize Kuznetsov's work on singular cubic fourfolds \cite{Kuz10}.

 First, we generalize the example outlined by Kuznetsov in \cite{Kuz10} of singular cubic fourfolds to higher dimension. 
\begin{example}\label{SingCubic}
Let $X$ be a singular cubic hypersurface in $\P^{3n+2}$ defined by the equation
$$
f(x_1, \ldots, x_{3n+3}) = \sum_{i=1}^n x_i f_i(x_{n+1}, \ldots, x_{3n+3}) + f_0(x_{n+1}, \ldots, x_{3n+3}),
$$
where $f_0$ is a generic cubic with the given variables and $f_1, \ldots, f_n$ are generic quadrics. There is a semi-orthogonal decomposition for $X$ in the case where $n=1$:
$$
\dbcoh{X} = \langle \mathcal{A}, \O, \O(1), \O(2)\rangle.
$$
Kuznetsov proves that while the category $\mathcal{A}$ is not Calabi-Yau, it has a crepant categorical resolution $\tilde{\mathcal{A}}$ that is a Calabi-Yau category (Theorem 5.2 of \cite{Kuz10}). Moreover, $\tilde{\mathcal{A}}$ is the derived category of a K3 surface.  Here, when $n>1$, we can generalize the story. Analogously, there is a semi-orthogonal decomposition
$$
\dbcoh{X} = \langle \mathcal{A}, \O, \ldots, \O(3n-1)\rangle
$$
We find that $\dbcoh{Y}$ is a crepant categorical resolution of $\mathcal{A}$ where $Y$ is the $(n+1)$-dimensional Calabi-Yau complete intersection $Y$ in $\P^{2n+2}$ given by the zero locus of $f_0, ..., f_n$.
\end{example}

Second, we generalize the example of cubic fourfolds containing two planes in \cite{Ha00}.  
\begin{example}\label{2PlanesCubic}
Let $X$ be a generic degree $d$ hypersurface in $\P^{2d-1}$ that contains a two-dimensional plane $P_1$ and a $(2d-4)$-dimensional plane $P_2$ so that $P_1 \cap P_2 = \emptyset$. While smooth when $d =3$, this example becomes singular when $d>3$. By Orlov's theorem, there is a semi-orthogonal decomposition
$$
\dbcoh{X} = \langle \mathcal{A}, \O, \ldots, \O(d-1)\rangle.
$$
When $d=3$, then $\mathcal{A}$ is the derived category of a K3 surface (Proposition 4.7 of \cite{Kuz10}). We prove that when $d>4$, there exists a Calabi-Yau $(2d-4)$-fold $Y$ defined by the complete intersection of two hypersurfaces of bidegree $(d-1,1)$ and $(d-2,2)$ in $\P^{2d-4} \times \P^2$ so that $\dbcoh{Y}$ is a crepant categorical resolution of $\mathcal{A}$.
\end{example}

\subsection{Plan of Paper}

The plan for this paper is as follows. Section~\ref{sec: serre functors factorizations} introduces the factorization category, the main object of study for the paper. After providing its proper definition, we give criteria for showing that it admits a Serre functor and then compute it explicitly, proving Theorem~\ref{IntroSerre}. We end with the proof of Corollary~\ref{theorem:GeneralLG}, and Theorem~\ref{thm: DHT}. 

Section~\ref{sec: CCR} explains the relationship between crepant categorical resolutions, LG models, and variation of GIT. 

Section~\ref{sec:3} provides the required toric geometry to study the factorization category for toric complete intersections, setting up the next section. Here we recall the necessary definitions of cones associated to certain total spaces of invertible sheaves over toric stacks. We also recall the relevant machinery for studying variation of GIT on affine spaces and its relation to the secondary fan. 

Section~\ref{GorCones} provides sufficient criteria for when a factorization category associated to a toric Landau-Ginzburg model is FCY and explicitly computes it dimension in terms of the fan and the R-charge. We then prove a comparison theorem, Theorem~\ref{thm: main thm}, for two birational toric gauged Landau-Ginzburg. We finish the section by considering the case where one or more of the Landau-Ginzburg models have a geometric interpretation as a complete intersection in a toric variety. 

We end the paper with Section~\ref{sec: examples}, where we provide a set of examples of our theorems, including a reproving of Orlov's theorem, a semi-orthogonal decomposition with a geometric FCY category, and the generalizations of the cases of singular cubic fourfolds and a cubic fourfold containing two planes outlined by Examples~\ref{SingCubic} and~\ref{2PlanesCubic} above.

\vspace{2.5mm}
\noindent \textbf{Acknowledgments:}  The authors would like to thank Paul Aspinwall, Michael Wemyss, Yuki Hirano, Genki Ouchi, Colin Diemer, Al Kasprzyk, and Matthew Ballard for conversations related to this work.  We are also grateful to Alexander Kuznetsov for suggesting the addition of crepant categorical resolutions and other corrections/improvements to an earlier version of this manuscript.  We also thank the referee for an extensive, comprehensive report that greatly improved the paper. The first-named author is grateful to the Korean Institute for Advanced Study for their hospitality and support while this document was being prepared. The first-named author is also thankful for support provided by the Natural Sciences and Engineering Research Council of Canada and Canada Research Chair Program under NSERC RGPIN 04596 and CRC TIER2 229953.
The second-named author acknowledges that this material is based upon work supported by the National Science Foundation under Award No.\ DMS-1401446 and the Engineering and Physical Sciences Research Council under Grant EP/N004922/1.

\section{Serre Functors for Landau-Ginzburg Models} \label{sec: serre functors factorizations}

In this section, we will prove a certain class of triangulated categories associated to Landau-Ginzburg models admit an explicit Serre functor.

\subsection{Background on Factorizations}\label{sec: fact definition}

In order to keep the paper as self-contained as possible, we provide a summary of the language of factorizations, see \cite{BFK11} for more details.

Let $\kappa$ be an algebraically closed field of characteristic zero. Let $X$ be a smooth variety over $\kappa$ and $G$ an affine algebraic group that acts on it via the map $\sigma: G \times X \rightarrow X$. Take $w$ to be a $G$-invariant section of an invertible $G$-equivariant sheaf, $\mathcal{L}$, i.e., $w \in \Gamma(X, \mathcal{L})^G$. 
\begin{definition} \label{defn:factorization}
A \newterm{factorization} is the data $\E = (\E_{-1}, \E_0, \phi_{-1}^\E, \phi_0^\E)$ where $\mathcal{E}_{-1}$, $\mathcal{E}_0$ are $G$-equivariant quasi-coherent sheaves and 
$$
\mathcal{E}_{-1} \stackrel{\phi_0^{\mathcal{E}}}{\rightarrow} \mathcal{E}_0 \stackrel{\phi_{-1}^{\mathcal{E}}}{\rightarrow} \mathcal{E}_{-1} \otimes_{\O_X} \mathcal{L}
$$
are morphisms such that
\begin{equation*}\begin{aligned}
\phi_{-1}^{\mathcal{E}} \circ \phi_0^{\mathcal{E}} &= w, \\
(\phi_0^{\mathcal{E}} \otimes \mathcal{L}) \circ \phi_{-1}^{\mathcal{E}} &= w.
\end{aligned}\end{equation*}
\end{definition}

A morphism between two factorizations of even degree $f : \E \to \F[2k]$ is a pair $f = (f_0, f_{-1})$ defined by 
$$
\mathrm{Hom}^{2k}_{\mathsf{Fact}(X, G, w)} ( \E, \F) : = \mathrm{Hom}_{\mathrm{Qcoh}_GX}( \E_{-1}, \F_{-1} \otimes_{\O_X} \L{}^k) \oplus \mathrm{Hom}_{\mathrm{Qcoh}_GX}( \E_{0}, \F_{0} \otimes_{\O_X} \L{}^k)
$$
and, similarly,  a morphism of odd degree $f : \E \to \F[2k+1]$ is a pair $f = (f_0, f_{-1})$ defined by 
$$
\mathrm{Hom}^{2k+1}_{\mathsf{Fact}(X, G, w)} ( \E, \F) : = \mathrm{Hom}_{\mathrm{Qcoh}_GX}( \E_{0}, \F_{-1} \otimes_{\O_X} \L{}^{k+1}) \oplus \mathrm{Hom}_{\mathrm{Qcoh}_GX}( \E_{-1}, \F_{0} \otimes_{\O_X} \L{}^k).
$$
You can equip these Hom sets with a differential coming from the graded commutator with the morphisms defining $E$ and $F$.  This yields a dg category $\mathsf{Fact}(X, G, w)$.
Also, denote by $\mathsf{fact}(X, G, w)$  to be the full dg-subcategory of $\mathsf{Fact}(X, G, w)$ whose components are coherent.

We now take a subcategory of $\mathsf{Fact}(X, G, w)$ with the same objects but only with the closed degree zero morphisms between any two objects $\E$ and $\F$. Denote this subcategory $Z^0 \mathsf{Fact}(X, G, w)$. The category $Z^0 \mathsf{Fact}(X, G, w)$ is abelian.  Hence, the notion of a complex of objects in $Z^0 \mathsf{Fact}(X, G, w)$ makes sense.

Given a complex of objects from $Z^0 \mathsf{Fact}(X, G, w)$
$$
\ldots \rightarrow \E^b \stackrel{f^b}{\rightarrow} \E^{b+1} \stackrel{f^{b+1}}{\rightarrow} \ldots
$$
the \newterm{totalization} of the complex is the factorization $\T \in \mathsf{Fact}(X, G, w)$ given by the data:
\begin{equation}\begin{aligned}
\T_{-1} &: = \bigoplus_{i = 2k} \E^i_{-1} \otimes_{\O_X} \L{}^{-k} \oplus \bigoplus_{i = 2k-1} \E^i_0 \otimes_{\O_X} \L{}^{-k} \\
\T_0 &:= \bigoplus_{i = 2k} \E^i_{0} \otimes_{\O_X} \L{}^{-k} \oplus \bigoplus_{i = 2k+1} \E^i_{-1} \otimes_{\O_X} \L{}^{-k}  \\
\phi_0^{\T} & := \bigoplus_{i =2k} f^i_0 \otimes {\L}^{-k}  \oplus \bigoplus_{i=2k-1} f^i_{-1} \otimes {\L}^{-k}, \\
\phi_1^{\T} & := \bigoplus_{i =2k} f^i_{-1} \otimes {\L}^{-k}  \oplus \bigoplus_{i=2k-1} f^i_{0} \otimes {\L}^{-k}. 
\end{aligned}\end{equation}

Now let, $\mathsf{Acyc}(X, G, w)$ be the full subcategory of objects of $\mathsf{Fact}(X, G, w)$ consisting of totalizations of bounded exact complexes of $Z^0 \mathsf{Fact}(X,G,w)$. Similarly, let $\mathsf{acyc}(X, G, w) = \mathsf{Acyc}(X, G, w) \cap \mathsf{fact}(X, G, w)$.  Finally, by $[\mathcal{C}]$ we denote the homotopy category of a dg category $\mathcal{C}$.

We have the following general definition.

\begin{definition}
The \newterm{absolute derived category} $\op{D}^{\text{abs}}[\mathsf{Fact}(X, G, w)]$ of $[\mathsf{Fact}(X,G,w)]$ is the Verdier quotient of $[\mathsf{Fact}(X,G,w)]$ by $[\mathsf{Acyc}(X, G, w)]$.
\end{definition}

However, the category we focus on in this paper uses only coherent sheaves as objects.  For this, we use the  following slightly abbreviated notation.

\begin{definition}
The \newterm{absolute derived category} $\dabsfact{X,G,w}$ of $[\mathsf{fact}(X,G,w)]$ is the idempotent completion of the Verdier quotient of $[\mathsf{fact}(X,G,w)]$ by $[\mathsf{acyc}(X, G, w)]$. Equivalently, this is the full subcategory of $\op{D}^{\text{abs}}[\mathsf{Fact}(X, G, w)]$ split-generated by objects in $\mathsf{fact}(X,G,w)$.
\end{definition}

The category $\dabsfact{X,G,w}$ can be thought of as the derived category of the gauged Landau-Ginzburg model $(X, G, w)$.

\begin{remark}
The category $\dabsfact{X,G,w}$ is triangulated with shift functor $$\E[1] := (\E_0, \E_{-1}\otimes \L, \phi_{-1}^{\E}, \phi_0^{\E} \otimes \L)$$ where   $\E = (\E_{-1}, \E_0, \phi_{-1}^\E, \phi_0^\E)$. Note that, in particular, 
\begin{equation}\label{2isO(chi)}
[2] = - \otimes \L.
\end{equation}
\end{remark}

\subsection{Serre Functors of  $\dabsfact{X,G,w}$}

In this section, our goal is to calculate the Serre functor for $\dabsfact{X,G,w}$. We do this by first proving that, under certain assumptions, a certain dg-enhancement of $\dabsfact{X,G,w}$ is homologically smooth and proper.  This implies that it admits a Serre functor by a result of \cite{Shl}. 

 Following \cite{BFK11}, we take our enhancement to be $\mathsf{Inj}_{\op{coh}}(X,G,w)$ which is defined to be the full subcategory of $\mathsf{Fact}(X, G, w)$ consisting of objects with injective components which are isomorphic in $\op{D}^{\text{abs}}[\mathsf{Fact}(X, G, w)]$  to objects with coherent components.
 
\begin{proposition}[Proposition 5.11 of \cite{BFK11}]
The dg-category $\mathsf{Inj}_{\op{coh}}(X,G,w)$ is a dg enhancement of $\dabsfact{X,G,w}$.
\end{proposition}

We can, then, describe the Serre functor starting from the formal definition in \cite{Shl}.  This requires a sequence of lemmas.   Many of the technical aspects of which can be outsourced to \cite{BFK11}, which we cite often in this section.  Hence, we follow the notations and conventions of Ibid. 

Let us start by collecting some notation and definitions.  Suppose $G$ is an algebraic group acting on two algebraic varieties $X, Y$. 
We define the following shorthand for the global quotient stack,
\[
X \overset{G}{\times} Y  := [X \times Y / G].
\]

If $H$ is a closed subgroup of $G$, we let $H$ act on $G$ by inverse multiplication on the right
\[
g \cdot h := h g^{-1}
\]
to define $G \overset{H}{\times} X$.
Then we define an inclusion
\begin{align*}
\iota: X & \rightarrow G \overset{H}{\times} X \\
 x & \mapsto (e,x).
\end{align*}

By Lemma 1.3 of \cite{Tho97}, we have that the pullback functor $\iota^*$ induces the equivalences of equivariant categories of sheaves
\begin{equation*}
\op{Qcoh}_G(G \overset{H}{\times} X) \cong \op{Qcoh}_H(X).
\end{equation*}
\begin{definition}
Let $H$ be a closed subgroup of $G$ and assume we have an action $\sigma: G \times X \rightarrow X$. Consider the inclusion map $\iota: X \rightarrow G \overset{H}{\times} X$ and the $G$-equivariant morphism $\alpha: G \overset{H}{\times} X\rightarrow X$ descending from the action $\sigma : G \times X \rightarrow  X$. The \newterm{induction functor} is the composition 
\[
 \op{Ind}_H^{G} := \alpha_* \circ (\iota^*)^{-1}.
 \]
\end{definition}

The induction functor allows us to remind the reader of the following notation from \cite{BFK11}:
\[
\nabla(\mathcal F)  := \op{Ind}_G^{G \times_{\gm} G} \Delta_* \mathcal F,
\]
where $\Delta$ is the diagonal map.

\begin{remark}
The functors $ \op{Ind}_H^{G}$ and $\Delta_*$  are exact as $\Delta$ and $\alpha$ are both affine morphisms and $\iota^*$ is an equivalence of abelian categories.  Hence, functors appearing in the definition of $\nabla$ can be viewed both in the abelian and derived setting.
\end{remark}

\begin{lemma} \label{lem: iso of Ind}
Let 
\[
s,p : \op{ker }\chi \times X \to X
\] denote the action and the projection respectively and consider the map
\begin{align*}
(s,p) : \op{ker }\chi \times X & \to X \times X \\
(g, x) & \mapsto (gx, x).
\end{align*}

\noindent Let $\F$ be a quasi-coherent sheaf. There is an isomorphism of quasi-coherent sheaves
\[
\nabla(\mathcal F)  \cong (s,p)_* s^* \mathcal F.
\].
\end{lemma}

\begin{proof}
Let $G \times_{\gm} G$ be the fiber product using the $\chi : G\rightarrow \gm$ for both factors. We have the following commutative diagram,
 \begin{center}
 \begin{tikzpicture}[description/.style={fill=white,inner sep=2pt}]
  \matrix (m) [matrix of math nodes, row sep=2em, column sep=2em, text height=1.5ex, text depth=0.25ex]
  {X &   (G \times_{\gm} G) \overset{G}{\times} X & \op{ker }\chi \times X \\  X \times X &   (G \times_{\gm} G) \overset{G}{\times} X \times X& \op{ker }\chi \times X \times X \\ 
  &   X \times X &  \\ };
  \path[->,font=\scriptsize]
  (m-1-1) edge node[above] {$j$} (m-1-2)
  (m-2-1) edge node[above] {$\iota$} (m-2-2)
    (m-1-1) edge node[left] {$\Delta$} (m-2-1)
        (m-1-2) edge node[left] {$\tilde{\Delta}$} (m-2-2)
                (m-1-3) edge node[left] {$\hat{\Delta}$} (m-2-3)
  (m-2-2) edge node[left] {$\alpha$} (m-3-2)
  (m-2-2) edge node[above] {$\Phi$} (m-2-3)
    (m-1-2) edge node[above] {$\hat{\Phi}$} (m-1-3)
  (m-2-3) edge node[below] {$(p,p)$} (m-3-2)
  ;
 \end{tikzpicture}
 \end{center} 
where
 \begin{align*}
 \Delta(x) & = (x,x), &
  j(x) & = (e, e,x), \\
 \tilde{\Delta}(g_1, g_2, x) & = (g_1, g_2, x,x), &
 \hat{\Delta}(g, x) & = (g, gx,x), \\
 \Phi(g_1, g_2, x, y) & = (g_1g_2^{-1}, g_1x, g_2y), &
 \hat{\Phi}(g_1, g_2, x) & = (g_1g_2^{-1}, g_1x). \\
 \end{align*}

We compute
\begin{align*}
\nabla(\mathcal F) & =  \alpha_* \circ (\iota^*)^{-1} \Delta_* \mathcal F \\
& =  \alpha_* \circ (\iota^*)^{-1} \Delta_* j^* \hat{\Phi}^* s^* \mathcal F \\
& =  \alpha_* \circ \tilde{\Delta}_* \hat{\Phi}^* s^* \mathcal F \\
& = (p,p)_* \hat{\Delta}_* s^* \mathcal F \\
& = (s,p)_* s^* \mathcal F.
\end{align*}

\end{proof}

\begin{definition}
A dg category $A$ is called \newterm{homologically smooth} if $A$ is a compact object of $D(A \otimes A^{\op{op}} - \text{Mod})$ i.e. $A \in D_{\op{perf}}(A \otimes A^{\op{op}})$.
\end{definition}

\begin{definition}
Consider a group $G$ acting on a space $X$ and let $w$ be a global function defined on $X$. We say that $w$ is \newterm{semi-invariant} with respect to a character $\chi$ of $G$ if, for any $g \in G$,
$$
w(g\cdot x) = \chi(g) w(x).
$$
The global function $w$ is semi-invariant if and only if $w$ is a section of the equivariant line bundle $\O(\chi)$ on the global quotient stack $[X/G]$.  This can also be written $w \in \Gamma(X, \O_X(\chi))^G$.
\end{definition}

For the rest of the paper, we restrict our attention to the case where $w$ is a semi-invariant function.

\begin{lemma}
Let $X$ be a smooth algebraic variety and $G$ be an algebraic group acting on $X$.  Let $\chi: G \to \gm$ be a character and $w \in \Gamma(X, \O_X(\chi))^G$ be a semi-invariant function.  Denote by $\partial w $ the critical locus (with its reduced scheme structure). Assume that $[X/ \op{ker }\chi]$ has finite diagonal and that  $\partial w \subseteq Z(w)$.  Then the dg-category,
\[
\mathsf{Inj}_{\op{coh}}(X,G,w)
\]
is homologically smooth.
\label{lem: homologically smooth}
\end{lemma}

\begin{proof}
The diagonal map for $[X / \op{ker } \chi]$ is realized as 
\[
(s,p): \op{ker } \chi \times X \to X \times X.
\]
  This is finite by assumption, hence proper.  Therefore $(s,p)_*\O_{\op{ker} \chi \times X}$ is coherent.

By Lemma~\ref{lem: iso of Ind}, $\nabla(\O_X) = (s,p)_*\O_{\op{ker} \chi \times X}$.  So, renotating, $\nabla(\O_X)$ is coherent.  It follows from Proposition 3.15 of \cite{BFK11} that since $\nabla(\O_X)$ is an object of $\op{fact}(X\times X,G\times_{\gm} G,w\boxplus w)$, it is a compact object.  

By Theorem 5.15 of \cite{BFK11}, we have a dg-functor
\[
\lambda_{w \boxplus w} : \mathsf{Inj}_{\op{coh}}(X \times X,G \overset{\gm}{\times} G,w \boxplus w)  \to (\mathsf{Inj}_{\op{coh}}(X,G,w) \otimes \mathsf{Inj}_{\op{coh}}(X,G,w)^{\op{op}} )-\op{Mod},
\]
(here, we have implicitly used the assumption that $\partial w \subseteq Z(w)$ to remove the support condition in the statement of Theorem 5.15 of ibid).

This induces an equivalence
\[
\op{D}^{\text{abs}}[X \times X,G \overset{\gm}{\times} G,w \boxplus w] \rightarrow D((\mathsf{Inj}_{\op{coh}}(X,G,w) \otimes \mathsf{Inj}_{\op{coh}}(X,G,w)^{\op{op}} )-\op{Mod}).
\]
This equivalence takes $\nabla$ to the bimodule $\mathsf{Inj}_{\op{coh}}(X,G,w)$ by Lemma 3.54 of \cite{BFK11}.

In conclusion, when viewed as a bimodule, $\mathsf{Inj}_{\op{coh}}(X,G,w)$ is a compact object of $D(\mathsf{Inj}_{\op{coh}}(X,G,w) \otimes \mathsf{Inj}_{\op{coh}}(X,G,w)^{\op{op}} -\op{Mod})$, i.e.,  $\mathsf{Inj}_{\op{coh}}(X,G,w)$ is cohomologically smooth.

\end{proof}

\begin{remark}
Any separated Deligne-Mumford stack has finite diagonal.  Conversely, over $\C$, any stack with finite diagonal is separated and Deligne-Mumford.
\end{remark}

\begin{definition}
A dg category $A$ is called \newterm{proper} if there exists a strong generator $E$ of the homotopy category of $A$ such that 
 \[
 \bigoplus_r \op{H}^r(\op{Hom}_{A}(E,E))
 \]
 is finite dimensional.
\end{definition}

Recall from \cite{BFK11} that given a $G$-equivariant sheaf $\mathcal F$ supported on $Z(w)$ we can define a factorization
\[
\Upsilon \mathcal F := (0, \mathcal F, 0 ,0),
\]
using the notation given in Definition~\ref{defn:factorization}.

\begin{lemma}
Let $X$ be a smooth algebraic variety and $G$ be a linearly reductive algebraic group acting on $X$.  Let $\chi: G \to \gm$ be a non-trivial character.  Assume that $[X/ \op{ker }\chi]$ has finite diagonal  and is proper over  $\op{Spec }\kappa$.  In addition, assume that we have the containment, $\partial w \subseteq Z(w)$.
  Then,
\[
\mathsf{Inj}_{\op{coh}}(X,G,w)
\]
is a proper dg-category.
\label{lem: proper}
\end{lemma}

\begin{proof}
By Lemma~\ref{lem: homologically smooth}, $\dabsfact{X,G,w}$ is homologically smooth.  Hence by Lemma 4.23 of \cite{BFK11}, the diagonal object $\op{Ind}^{G\times_{\mathbb{G}_m} G}_{G} \Delta_* \O_X$ is generated by exterior products.  Now, if  $\op{Ind}^{G\times_{\mathbb{G}_m} G}_{G} \Delta_* \O_X$ is a summand of a finite sequence of cones of exterior products $\mathcal E_i \boxtimes \mathcal F_i$, then thinking of these exterior products as integral transforms expresses any object as a summand of a finite sequence of cones of some graded vector spaces tensored with the $\mathcal F_i$.   Therefore, $\dabsfact{X,G,w}$ admits a strong generator. 

Now, we show that the category is Ext-finite, so that, in particular, the cohomologies of the endomorphism algebra of a strong generator must be finite dimensional.
By Proposition 3.64 and Lemma 4.13 of \cite{BFK11},  $\dabsfact{X,G,w}$ is generated by objects of the form $\Upsilon E$ where $E \in \dbcoh{[\partial w / G]}$.  Since, $\dbcoh{[\partial w / G]}$ is generated by sheaves,  it suffices to show that
\[
\bigoplus_r \op{Hom}_{\op{D}^{\op{abs}}(\op{fact} w)}(\Upsilon E,\Upsilon F[r])
\]
 is finite dimensional for any $E, F \in \op{coh}[\partial w / G]$.

That is, let $E, F$ be $G$-equivariant coherent sheaves on $\partial w$.
By Lemma 3.11 of \cite{BDFIK16}, there is a spectral sequence whose $E_1$-page is 
 \begin{displaymath}
  E^{p,q}_1 = 
  \begin{cases} 
\op{Ext}^{p+q}_{[X/G]}(E,F \otimes \O_X(-s\chi)) & p = 2s \\ 
0 & p = 2s+1,
  \end{cases}
 \end{displaymath}
which strongly converges to $\bigoplus_r \op{Hom}_{\op{D}^{\op{abs}}(\op{fact} w)}(E,F[r])$.

Since $X$ is smooth, $E^{p,q}_1 = 0$ unless $0 \leq p+q \leq \op{dim }X$. Now since $G$ and $\op{ker }\chi$ are linearly reductive,
\begin{align*}
\bigoplus_{s \in \Z} \op{Ext}_{[X/G]}^{i}(E, F \otimes \O_X(-s\chi)) & = \bigoplus_{s \in \Z} \op{Ext}_{X}^{i}(E, F \otimes \O_X(-s\chi))^G \\
& \subseteq   \op{Ext}_{X}^{i}(E, F)^{\op{ker }\chi} \\
& = \op{Ext}_{[X/\op{ker \chi}]}^{i}(E, F).
\end{align*} 
The righthand side is finite dimensional by assumption. 

 Therefore, there are finitely many pairs $(s,i)$ where $ \op{Ext}_{[X/G]}^{i}(E, F \otimes \O_X(-s\chi))$ is nonzero.  It follows that $E^{2s,q}_1$ is nonzero for finitely many $q$.  Furthermore, since $E^{2s,q}_1 = 0$ unless $0 \leq 2s+q \leq \op{dim }X$, it follows that there are also finitely many values of $s$ for which $E^{2s,q}_1$ is nonzero.
 
In conclusion, the spectral sequence is bounded and its terms are finite dimensional.   Hence, $\bigoplus_r \op{Hom}_{\op{D}^{\op{abs}}(\op{fact} w)}(E,F[r])$ is finite dimensional, as desired.
\end{proof}

\begin{remark}
It is enough to assume that $[\partial w / \op{ker }\chi]$ is only cohomologically proper.  This means, essentially by definition, $\op{Ext}_{[X/\op{ker \chi}]}^{i}(E, F)$ is finite dimensional for any two coherent sheaves $E,F$ which is all that is used in the proof.  The assumption that   $[\partial w / \op{ker }\chi]$  is proper propagates to other results in this section, which could also be replaced by cohomologically proper.  For later applications in this paper, we will always have that $[\partial w / \op{ker }\chi]$  is proper.  \end{remark}

\begin{lemma}
Let $G$ be an algebraic group acting on a smooth variety $X$.  There is a $G \times_{\gm} G$-equivariant isomorphism
\[
  \mathbf{R}\mathcal Hom_{X \times X}(\op{Ind}_G^{G \times_{\gm} G} \Delta_* \O_X, \O_{X \times X}) \cong \op{Ind}_G^{G \times_{\gm} G} \Delta_* \omega^{-1}_X[\op{dim } G - 1 -\op{dim } X].
\]
\label{lem: canonical is dual}
\end{lemma}

\begin{proof}  
We formally compute:
\begin{align*}
  \mathbf{R}\mathcal Hom_{X \times X}(\op{Ind}_G^{G \times_{\gm} G} \Delta_* \O_X, \O_{X \times X}) & =   \mathbf{R}\mathcal Hom_{X \times X}((p,s)_*\O_{\op{ker }\chi \times X}, \O_{X \times X}) \\
  & = (p,s)_*(p,s)^! \O_{X \times X} \\
  & = (p,s)_* s^*\omega^{-1}_{X} [\op{dim } G - 1 -\op{dim } X] \\
  & = \op{Ind}_G^{G \times_{\gm} G} \Delta_* \omega^{-1}_X[\op{dim } G - 1 -\op{dim } X].
\end{align*}
The first line is Lemma~\ref{lem: iso of Ind}.  The second line is equivariant Grothendieck duality \cite{Has60}.  The third line is just a computation of the relative canonical bundle since $\op{ker }\chi \times X$ and $X \times X$ are smooth.  The fourth line is Lemma~\ref{lem: iso of Ind} again.

\end{proof}

\begin{lemma}
Let $\mathcal F$ be a $G$-equivariant sheaf on $Z(w)$ such that
\[
\mathbf{R}\mathcal Hom_{X}(\mathcal F, \O_X) \cong \mathcal G [-t],
\]
where $\mathcal G$ is also a $G$-equivariant sheaf on $Z(w)$.  Then there is an isomorphism in $\op{D}^{\op{abs}}(\op{fact} w)$,
\[
\mathbf{R}\mathcal Hom_{X}(\Upsilon \mathcal F, \O_X) \cong \Upsilon \mathcal G [-t].
\]
\label{lem: dual commutes}
\end{lemma}

\begin{proof}
By Proposition 3.14 of \cite{BFK11}, there exists an exact sequence of factorizations in the abelian category of factorizations
\[
0 \to \mathcal V_s \xrightarrow{d_s} ... \xrightarrow{d_1} \mathcal V_0 \to \Upsilon \mathcal F \to 0
\]
so that $\Upsilon \mathcal F$ is isomorphic to the totalization of the complex
\[
\mathcal V_s \to ... \to \mathcal V_0
 \]
 in $\dabsfact{X,G,w}$.
Therefore, 
\[
\mathbf{R}\mathcal Hom_{X}(\Upsilon \mathcal F, \O_X) 
 \]
 is isomorphic to the totalization of the complex
 \[
\mathbf{R}\mathcal Hom_{X}( \mathcal V_0, \O_X) \xrightarrow{d_1^\vee} ... \xrightarrow{d_s^\vee} \mathbf{R}\mathcal Hom_{X}(\mathcal V_s, \O_X)[-s].
 \]

Now, notice that there are also exact sequences,

\[
0 \to \mathbf{R}\mathcal Hom_{X}( \mathcal V_0, \O_X) \to ... \to \mathbf{R}\mathcal Hom_{X}(\mathcal V_{t-1}, \O_X) \to \op{im}(d_{t}^\vee) \to  0,
\]
\[
0 \to \op{ker}(d_{t+1}^\vee) \to \mathbf{R}\mathcal Hom_{X}( \mathcal V_{t}, \O_X) \to ... \to \mathbf{R}\mathcal Hom_{X}(\mathcal V_s, \O_X) \to 0,
\]
and
\[
0 \to \op{im}(d_{t}^\vee) \to \op{ker}(d_{t+1}^\vee) \to  \Upsilon \mathcal G \to 0
\]
in the abelian category of factorizations.

Hence, we have a distinguished triangle
\begin{equation}
\op{im}(d_{t}^\vee) \to \op{ker}(d_{t+1}^\vee) \to  \Upsilon \mathcal G \to \op{im}(d_{t}^\vee)[1]
\label{triangle}
\end{equation}
in $\dabsfact{X,G,w}$.
Denoting the totalization of
\[
\mathbf{R}\mathcal Hom_{X}( \mathcal V_0, \O_X) \to ... \to \mathbf{R}\mathcal Hom_{X}(\mathcal V_{t-1}, \O_X)
\]
by $A$ and the totalization of
\[
\mathbf{R}\mathcal Hom_{X}( \mathcal V_{t}, \O_X) \to ... \to \mathbf{R}\mathcal Hom_{X}(\mathcal V_s, \O_X)
\]
by $B$, we can replace the terms in the distinguished triangle \eqref{triangle}
by
\[
A \to B[t-s] \to \Upsilon \mathcal G \to A[1].
\]
Hence, $\Upsilon \mathcal G[-t]$ is the cone of $A[-t]\to B[-s]$ in $\dabsfact{X,G,w}$.  The totalization of
 \[
 \mathbf{R}\mathcal Hom_{X}( \mathcal V_s, \O_X) \to ... \to \mathbf{R}\mathcal Hom_{X}(\mathcal V_0, \O_X)[-s]
 \]
 can also be described as the cone of $A[-t] \to B[-s]$.   Hence $\Upsilon G[-t]$ agrees with the derived dual of $\Upsilon F$, as desired.
 
\end{proof}

\begin{theorem}\label{SerreFunctorDescription}
Let $X$ be a smooth algebraic variety and $G$ be a linearly reductive algebraic group acting on $X$. Let $\chi: G \to \gm$ be a character and $w \in \Gamma(X, \O_X(\chi))^G$.   Assume that $[X/ \op{ker }\chi]$ has finite diagonal.  In addition, assume that $\partial w \subseteq Z(w) $, and that $[\partial w / \op{ker } \chi]$ is proper.  Then,
\[
  \dabsfact{X,G,w}
\]
admits a Serre functor given by 
\[
S := (- \otimes \omega_X)[\op{dim } X - \op{dim } G +1].
\]
\end{theorem}

\begin{proof}
By Lemma~\ref{lem: homologically smooth}, $\dabsfact{X,G,w}$ is homologically smooth.  Hence by Lemma 4.23 of \cite{BFK11}, the diagonal object $\op{Ind}^{G\times_{\mathbb{G}_m} G}_{G} \Delta_* \O_X$ is generated by exterior products.  It follows that $\dabsfact{X,G,w}$ admits a strong generator.
Hence,
\[
\dabsfact{X,G,w} \cong D_{\op{perf}}(A)
\]
for a dg-algebra $A$.  

Since $\dabsfact{X,G,w}$ is homologically smooth (Lemma~\ref{lem: homologically smooth}) and proper (Lemma~\ref{lem: proper}), so is $A$.
Hence by Theorems 4.2 and 4.4 of \cite{Shl}, it admits a Serre functor whose inverse is given formally by
\[
A^! := \op{RHom}_{A^{\op{op}} \otimes A}(A, A^{\op{op}} \otimes A).
\]

By Lemma 3.30 of \cite{BFK11},
\[
A^{\op{op}} \cong \dabsfact{X, G, -w}.
\]

Now, by Theorem 5.15 of \cite{BFK11}, the dg-category $A^e$ is quasi-equivalent to 
\[
\mathsf{fact}[X \times X, G \times_{\gm} G, w \boxplus -w]
\]
and $A^!$ is identified with
\[
 \mathbf{R}\mathcal Hom_{X \times X}(\nabla(\O_X), \O_{X \times X}).
\]

Now,
\begin{align*}
S^{-1} & =   \mathbf{R}\mathcal Hom_{X \times X}(\nabla(\O_X), \O_{X \times X}) \\
& =   \mathbf{R}\mathcal Hom_{X \times X}(\Upsilon \op{Ind}^{G\times_{\mathbb{G}_m} G}_{G} \Delta_* \O_X, \O_{X \times X}) \\
& \cong \Upsilon    \mathbf{R}\mathcal Hom_{X \times X}(\op{Ind}^{G\times_{\mathbb{G}_m} G}_{G} \Delta_* \O_X, \O_{X \times X}) \\
& \cong   \Upsilon \op{Ind}_G^{G \times_{\gm} G} \Delta_* \omega^{-1}_X[\op{dim } G -1 - \op{dim }X].
\end{align*}

The second line is by definition.  The third line is Lemma \ref{lem: dual commutes}.  The fourth line is Lemma  \ref{lem: canonical is dual}.  

Finally, as an integral kernel,   $\Upsilon \op{Ind}_G^{G \times_{\gm} G} \Delta_* \omega^{-1}_X[\op{dim } G -1 - \op{dim }X]$ is just
\[
(- \otimes \omega^{-1}_X)[\op{dim } G -1 - \op{dim }X],
\]
by Lemma 3.54 of \cite{BFK11}.  The inverse to this functor is
\[
S = (- \otimes \omega_X)[\op{dim } X - \op{dim }G +1].
\]
\end{proof}

\begin{corollary}\label{cor: torsion canonical fractional calabi-yau}
Let $X$ be a smooth algebraic variety and $G$ be a linearly reductive algebraic group acting on $X$. Let $\chi: G \to \gm$ be a character and $w \in \Gamma(X, \O_X(\chi))^G$.   Assume that $[X/ \op{ker }\chi]$ has finite diagonal and $\omega_X$ is torsion as a $\op{ker} \chi$-equivariant line bundle.  In addition, assume that $\partial w \subseteq Z(w) $, and that $[\partial w / \op{ker } \chi]$ is proper.   Then,
\[
  \dabsfact{X,G,w}
\]
is fractional Calabi-Yau.  If the canonical bundle of $[X/ \op{ker }\chi]$ is trivial, then this category is Calabi-Yau.
\end{corollary}
\begin{proof}
By Theorem~\ref{SerreFunctorDescription}, the Serre Functor on
\[
  \dabsfact{X,G,w}
\]
is given by 
\[
S = (- \otimes \omega_X)[\op{dim } X - \op{dim }G +1].
\]
where $\omega_X$ has the natural $G$-equivariant structure.  Applying $\op{Hom}(-,\gm)$ to the exact sequence
\[
0 \to \op{ker }\chi \to G \to \gm \to 0,
\]
we get 
\begin{equation}
0 \to \Z \to\widehat{G} \to  \widehat{\op{ker }\chi} \to 0,
\label{eq: ses}
\end{equation}
where $\Z$ is generated by the character $\chi$.

Now, by assumption, 
\[
\omega_X^{\otimes l}= \O_X
\]
 with its natural $\op{ker }\chi$-equivariant structure, i.e., $\omega_X^{\otimes l} $ is in the kernel of the map $\widehat{G} \to  \widehat{\op{ker }\chi}$. Therefore, from the exact sequence \eqref{eq: ses}, we have
\[
\omega_X^{\otimes l}= \O_X(\chi)^{\otimes m}
\]
as $G$-equivariant sheaves, for some $m\in \Z$.

Using Equation~\eqref{2isO(chi)}, there is a natural isomorphism of functors,
\[
-\otimes \O_X(\chi) = [2].
\]

Hence,
\begin{align*}
S^l & = (- \otimes \omega_X^{\otimes l})[l(\op{dim } X - \op{dim }G +1)] \\
& = (- \otimes  \O_X(\chi)^{\otimes m})[l(\op{dim } X - \op{dim }G +1)] \\
& = [l(\op{dim } X - \op{dim }G +1) + 2m].
\end{align*}

The Calabi-Yau case is when $l = 1$.

\end{proof}

\begin{corollary} \label{cor: DHT}
Conjecture~\ref{conj: LGKawamata} holds for elementary wall crossings (see Definition~\ref{def: EWC}).
\end{corollary}
\begin{proof}
This follows immediately from Proposition 4.3.8 of \cite{BFK12} and Corollary~\ref{cor: torsion canonical fractional calabi-yau}.
\end{proof}

\section{Crepant Categorical Resolutions via LG Models}\label{sec: CCR}
Let $Z$ be a variety with a $G$-action and  $\mathcal D$ be an admissible subcategory of $\dbcoh{[Z/G]}$.  We denote by $\mathcal D^{\op{perf}}$ the full subcategory of $\mathcal D$ consisting of $G$-equivariant perfect complexes on $Z$.

\begin{definition}
Let $\tilde{\mathcal D}$ be the homotopy category of a homologically smooth and proper pretriangulated dg category.  A pair of exact functors
\[
F: \tilde{\mathcal D} \to \mathcal D
\]
\[
G : \mathcal D^{\op{perf}} \to \tilde{\mathcal D}
\]
is a \newterm{categorical resolution of singularities} if $G$ is left adjoint to $F$ and the natural morphism of functors $\op{Id}_{\mathcal D^{\op{perf}}} \to FG$ is an isomorphism.  We say that the categorical resolution of singularities is \newterm{crepant} if $G$ is also right adjoint to $F$.
\end{definition}

\begin{remark}\label{rmk: CCR}
The definition presented here is slightly different than that of \cite{Kuz10}.  In Ibid., $\tilde{\mathcal D}$ is required to be an admissible subcategory of $\dbcoh{X}$ where $X$ is a smooth variety.  This definition is in lieu of  requiring $\mathcal D$ be a homologically smooth and proper triangulated dg category.  All examples in this paper will be crepant categorical resolutions in both senses.
\end{remark}

Let $U$ be a variety with the action of a linearly reductive group $G$, $\chi$ be a character of $G$, and $w$ be a section of $\O_U(\chi)$.
Let
\[
i : V \to U
\]
be a $G$-equivariant open immersion.  We have a (both left and right) adjoint pair of functors between categories of factorizations with quasi-coherent components
\[
i_*: \dabsFact{V, G, w} \to  \dabsFact{U, G, w}
\]
\[
i^*: \dabsFact{U, G, w} \to  \dabsFact{V, G, w}
\]
Note that, since $i$ is an open immersion, $i_*$ is both left and right adjoint to $i^*$.

\begin{definition}
Let $\dabsfact{V, G, w}_{\op{rel }U}$ denote the full subcategory of $\dabsfact{V, G, w}$ consisting of factorizations $\mathcal F$ such that the closure of the support of $\mathcal F$ as a subset of $U$ does not intersect $U \setminus V$.  
\end{definition}
Then, the adjunction between $i_*$ and $i^*$ restricts to
\[
i_*: \dabsfact{V, G, w}_{\op{rel }U} \to \dabsfact{U, G, w}
\]
\[
i^*: \dabsfact{U, G, w} \to \dabsfact{V, G, w}.
\]
Let $Y$ be a smooth quasi-projective variety with a $G$ action.  Suppose that $s$ is a regular section of a $G$-equivariant vector bundle $\mathcal E$ on $Y$ with vanishing locus $Z:=Z(s)$.  Let $\gm$ act on $\op{tot }\mathcal E^\vee$ by fiberwise dilation and consider the pairing $w = \langle s, - \rangle$ as a section of $\O_{\op{tot }\mathcal E^\vee}(\chi)$ where $\chi$ is the projection character.

\begin{definition}
We define the gauged Landau-Ginzburg model associated to the complete intersection $Z$ to be the data
$$
(\op{tot }\mathcal E^\vee, G \times \gm, w).
$$
\end{definition}

The following theorem is originally due to Isik \cite{Isik} and Shipman \cite{Shipman} and due to Hirano \cite{Hirano} in the $G$-equivariant case, which is the case we will use.

\begin{theorem}[Proposition 4.8 of \cite{Hirano}]\label{Hirano}
Assume that $w$ is a regular section of $\mathcal E$.  There is an equivalence of categories,
$$
\Omega : \dbcoh{[Z/G]} \to \op{D}^{\op{abs}}[\op{tot}(\mathcal{E}^\vee), G \times \gm, \langle w,-\rangle].
$$
\end{theorem}

The following lemma is the $G$-equivariant case of Remark 3.7 of \cite{Shipman}.

\begin{lemma}
Assume that $Y$ admits a $G$-ample line bundle.  The equivalence of categories
\[
\Omega: \dbcoh{[Z/G]} \to \dabsfact{\op{tot }\mathcal E^\vee, G \times \gm, w},
\]
restricts to an equivalence between the full subcategory of perfect objects $ \op{Perf } [Z/G]$  and the full subcategory of $\dabsfact{\op{tot }\mathcal E^\vee, G \times \gm, w}$ with objects supported on the zero section of $\mathcal E^\vee$.
\label{lem: zerosec}
\end{lemma}  

\begin{proof}
Recall that the functor $\Omega = j_* (\pi|_Z)^*$ where $\pi|_Z : \op{tot }\mathcal E^\vee|_{Z} \to Z $ is the projection and $j : \op{tot }\mathcal E^\vee|_{Z} \to \op{tot }\mathcal E^\vee$ is the inclusion.  To clarify notation, there is also a map $\pi : \op{tot }\mathcal E^\vee \to Y$.
Let $h : Z \to Y$ be the inclusion.  Since $Y$ is quasi-projective with a $G$-ample line bundle $\mathcal L$, the category $\op{Perf } [Z/G]$ is generated by objects of the form $h^* \mathcal L^{\otimes n}$. 

Since $h^* \mathcal L^{\otimes n}$ is a generator of $\op{Perf } [Z/G]$, it is enough to check that $\Omega(h^* \mathcal L^{\otimes n})$ is supported on  the zero section of $\mathcal E^\vee$ and that the objects $\Omega(h^* \mathcal L^{\otimes n})$ generate the full subcategory of $\dabsfact{\op{tot }\mathcal E^\vee, G \times \gm, w}$ with objects supported on the zero section of $\mathcal E^\vee$.  Now, we have,
\begin{align*}
\Omega(h^* \mathcal L^{\otimes n}) & =  j_*(\pi|_Z)^* h^* \mathcal L^{\otimes n} & \\
& =  j_*j^*\pi^* \mathcal L^{\otimes n} & \\
& \cong (0,  \O_{Z(\pi^*s)}, 0, 0) \otimes  \pi^* \mathcal L^{\otimes n}\\
& \cong ( \O_{Z(\op{taut})}, 0, 0, 0) \otimes \det(\E) \otimes  \pi^* \mathcal L^{\otimes n} [-\op{rk}\E]  \\
& \cong (\det(\E)\otimes \pi^* \mathcal L^{\otimes n}|_{Z(\op{taut})}, 0, 0, 0)[-\op{rk}\E].
\end{align*}
Line four is Proposition 3.20 of \cite{BFK11}. 

First, this shows, in particular, that $\Omega(h^* \mathcal L^{\otimes n})$ is supported on $Z(\op{taut})$, the zero section of $\mathcal E^\vee$. Second, let $\F = (\F_{-1}, \F_0, \phi_{-1}^{\F}, \phi_0^{\F})$ be an object of  $\dabsfact{\op{tot }\mathcal E^\vee, G \times \gm, w}$ supported on the zero section of $\mathcal E^\vee$.  We aim to show that $\F$ is generated by objects of the form $( \det(\E)\otimes\mathcal L^{\otimes n}|_{Z(\op{taut})}, 0, 0, 0)$.   
For this notice that 
\[
Z( w ) = Z(\op{taut}) \cup Z(\pi^* s).
\]
Now, the full subcategory of $\dbcoh{[Z( w ) / G \times \gm] }$ consisting of objects supported on $Z(\op{taut})$ is generated by the essential image of the pushforward.  Since $\det(\E)\otimes\mathcal L^{\otimes n}$ generates $\dbcoh{[Y/G]}$, we may just use objects of the form $\det(\E)\otimes\mathcal L^{\otimes n}|_{Z(\op{taut})}$.  Finally, under the equivalence (see Theorem 3.6 of \cite{Hirano}),
\[
\op{D}_{sg}[Z( w ) / G \times \gm] \to \dabsfact{\op{tot }\mathcal E^\vee, G \times \gm, w}
\]
these objects go precisely to $( \det(\E)\otimes\mathcal L^{\otimes n}|_{Z(\op{taut})}, 0, 0, 0)$ and objects supported on $Z(\op{taut})$ go to objects supported on $Z(\op{taut})$ as desired.

\end{proof}

\begin{theorem}\label{theorem: CCRTotal}

With the setup as above, assume that $Y$ admits a $G$-ample line bundle.  Let $U$ be a $G \times \gm$-equivariant partial compactification of $\op{tot }\mathcal E^\vee$.  Assume that 
\begin{itemize}
\item $w$ extends to $U$ as a section of $\O(\chi)$, 
\item $[U/G]$ has finite diagonal, and 
\item $[\partial w / G] \subseteq [U/G]$ is proper over $\op{Spec }\kappa$ and that $\partial w \subseteq Z(w)$ in $U$.  
\end{itemize}
Then, the functors
\[
i_* \circ \Omega :  \op{Perf}([Z/G]) \to \dabsfact{U, G, w}
\]
\[
\Omega^{-1} \circ i^* :  \dabsfact{U, G, w} \to \dbcoh{[Z/G]}
\]
form a crepant categorical resolution.

\end{theorem}
\begin{proof}
The assumptions assure that  $\dabsfact{U, G, w}$ is the homotopy category of a homologically smooth and proper dg-category by Lemmas~\ref{lem: homologically smooth} and \ref{lem: proper} and Proposition 5.11 of \cite{BFK11}.
Since $i$ is an open immersion, the functors are both left and right adjoint.  Furthermore, the adjunction morphism factors via the following natural isomorphisms,
\begin{align*}
\Omega^{-1} \circ i^* \circ i_* \circ \Omega & \cong \Omega^{-1} \circ\op{Id}_{\dabsfact{\op{tot }\mathcal E^\vee, G \times \gm, w}} \circ \Omega \\
&  \cong \op{Id}_{  \op{Perf}([Z/G]) }.
\end{align*}
\end{proof}

\begin{remark}
An extension of a general $w$ need not exist. We will give two examples of such extensions in the toric case in Subsections~\ref{subsec:SingCubics} and~\ref{subsec:TwoPlanes}. 
\end{remark}

We now give the general framework for identifying crepant categorical resolutions for factorization categories using variations of geometric invariant theory quotients. We finish with a general theorem. In the following sections, we will specialize to the toric case. The reader more interested in toric applications can refer Section~\ref{GorCones} for a specialization to toric varieties of Theorem~\ref{thm: BFKVGIT}, namely Theorem~\ref{thm: BFK toric version}.

\begin{definition} \label{definition: contracting loci}
 Let $\lambda: \mathbb{G}_m \to G$ be a one-parameter subgroup of $G$. We shall denote a connected component of $U^{\lambda}$ by $Z_{\lambda}^0$. Associated to $Z_{\lambda}^0$, we can associate another subvariety
 \begin{align*}
  Z_{\lambda} & := \{ u \in U \mid \lim_{\alpha \to 0} \sigma(\lambda(\alpha),u) \in Z_{\lambda}^0 \}.
 \end{align*}
 We call $Z_{\lambda}$ the \textbf{contracting variety} associated to $Z^0_{\lambda}$.
 
 We will also close these varieties up under the action of $G$. We set
 \begin{align*}
  S_{\lambda}^0 & := G \cdot Z_{\lambda}^0 \\
  S_{\lambda} & := G \cdot Z_{\lambda}.
 \end{align*}
\end{definition}
Also, set
 \begin{displaymath}
  P(\lambda) := \{ g \in G \mid \lim_{\alpha \to 0} \lambda(\alpha) g \lambda(\alpha)^{-1} \text{ exists}\}.
 \end{displaymath}
 and
 \begin{align*}
  U_+ & := U \setminus S_\lambda \\
  U_- & := U \setminus S_{-\lambda}
 \end{align*}

\begin{definition} \label{definition: matched wall crossing}
 Let $U$ be a smooth, quasi-projective variety equipped with a $G$ action and $\lambda : \gm \to G$ be a one parameter subgroup.  Fix a connected component of the fixed locus, $Z_{\lambda}^0$.  Assume that 
   \begin{itemize}
  \item The morphisms,
  \begin{align*}
   \tau_{\pm \lambda}: G \overset{P(\pm \lambda)}{\times} Z_{\pm \lambda} \to S_{\pm \lambda},
  \end{align*}
   are isomorphisms.
  \item The subsets $S_{\pm \lambda}$ are closed.
\end{itemize}
Under these assumptions, the pair of stratifications
 \begin{align*}
  U & = U_+ \sqcup S_{\lambda} \\
  U & = U_- \sqcup S_{-\lambda},
 \end{align*}
is called  an \textbf{elementary wall crossing}. 
\label{def: EWC}
 \end{definition}

\begin{definition}
Let $V$ be a smooth variety with $G$-action.  
Fix a one-parameter subgroup $\lambda : \gm \to G$ and a connected component of the fixed locus $Z_\lambda^0$ of the $G$ action on $V$ such that  $V = V_{\pm} \sqcup S_{\pm \lambda}$ is an elementary wall-crossing.    Let $U$ be a smooth variety with a $G$-action.  We say that a $G$-equivariant open immersion $V \to U$ is  \newterm{compatible} with an elementary wall crossing if $Z_{\pm \lambda}$ remains closed in $U$.
\end{definition}

Given an elementary wall crossing,
 \begin{align*}
  U & = U_+ \sqcup S_{\lambda} \\
  U & = U_- \sqcup S_{-\lambda},
 \end{align*}
 we let 
\begin{displaymath}
 t(\mathfrak{K}^\pm) := \mu(\omega_{S_{\pm \lambda}|U}, \pm \lambda, u)
\end{displaymath}
for $u \in Z_{\lambda}^0$ where $\mu$ is the Hilbert-Mumford numerical function. Here,
\begin{displaymath}
 \omega_{S_{\pm \lambda}|U} = \bigwedge\nolimits^{\op{codim} S_{\pm \lambda}} \mathcal N^{\vee}_{S_{ \pm \lambda}|U}
\end{displaymath}
is the relative canonical sheaf of the embedding, $S_{\pm \lambda} \to U$.  

\begin{theorem}[Theorem 3.5.2 of \cite{BFK12}]
\label{thm: BFKVGIT}
 Let $U$ be a smooth, quasi-projective variety equipped with the action of a reductive linear algebraic group, $G$. Let $w \in \op{H}^0(U,\mathcal L)^G$ be a $G$-invariant section of a $G$-line bundle, $\mathcal L$, and assume that $\mu(\mathcal L,\lambda, u) = 0$ for $u \in Z_{\lambda}^0$. 
 
 Assume we have an elementary wall crossing,
 \begin{align*}
  U & = U_+ \sqcup S_{\lambda}, \\
  U & = U_- \sqcup S_{-\lambda},
 \end{align*}
and that $S^0_{\lambda}$ admits a $G$ invariant affine open cover. Fix $d \in \Z$.  For the following functors, abuse notation to also let them represent their essential image. Then the following hold
 
 \begin{enumerate}
  \item If $t(\mathfrak{K}^+) < t(\mathfrak{K}^-)$, then there are fully-faithful functors,
  \begin{displaymath}
   \Phi^+_d: \dabsfact{U_-,G,w|_{U_-}} \to \dabsfact{U_+,G,w|_{U_+}},
  \end{displaymath}
  and, for $-t(\mathfrak{K}^-) + d \leq j \leq -t(\mathfrak{K}^+) + d -1$,
  \begin{displaymath}
   \Upsilon_j^+: \dabsfact{w|_{Z_{\lambda}^0},C(\lambda),w|_{Z_{\lambda}^0}}_jT \to \dabsfact{U_+,G, w|_{U_+}},
  \end{displaymath}
  and a semi-orthogonal decomposition,
  \begin{displaymath}
   \dabsfact{U_+,G,w|_{U_+}} = \langle \Upsilon^+_{-t(\mathfrak{K}^-)+d}, \ldots, \Upsilon^+_{-t(\mathfrak{K}^+)+d-1}, \Phi^+_d \rangle.
  \end{displaymath}
  \item If $t(\mathfrak{K}^+) = t(\mathfrak{K}^-)$, then there is an exact equivalence,
  \begin{displaymath}
   \Phi^+_d: \dabsfact{U_-, G,w|_{U_-}} \to \dabsfact{U_+, G,w|_{U_+}}.
  \end{displaymath}
  \item If $t(\mathfrak{K}^+) > t(\mathfrak{K}^-)$, then there are fully-faithful functors,
  \begin{displaymath}
   \Phi^-_d: \dabsfact{U_+, G, w|_{U_+}} \to \dabsfact{U_-, G,w|_{U_-}},
  \end{displaymath}
  and, for $-t(\mathfrak{K}^+) + d \leq j \leq -t(\mathfrak{K}^-) + d -1$,
  \begin{displaymath}
   \Upsilon_j^-: \dabsfact{Z_{\lambda}^0, C(\lambda),w|_{Z_{\lambda}^0}}_j \to \dabsfact{U_-, G, w|_{U_-}},
  \end{displaymath}
  and a semi-orthogonal decomposition,
  \begin{displaymath}
   \dabsfact{U_-, G, w|_{U_-}} = \langle \Upsilon^-_{-t(\mathfrak{K}^+)+d}, \ldots, \Upsilon^-_{-t(\mathfrak{K}^-)+d-1}, \Phi^-_d \rangle.
  \end{displaymath}
 \end{enumerate}
\end{theorem}

\begin{lemma}
\label{lem: restricts}
Suppose $V \to U$ is a $G$-equivariant open immersion which is compatible with an elementary wall-crossing $V = V_{\pm} \sqcup S_{\pm}$.
Then, the fully-faithful functor
\[
\Phi^{\pm}_d :  \dabsfact{U_{\mp}, G , w}  \to \dabsfact{U_{\pm}, G , w}
\]
restricts to a functor
\[
\Phi^{\pm}_d  :  \dabsfact{V_{\mp}, G , w}_{\op{rel }U_\mp}  \to \dabsfact{V_\pm, G , w}_{\op{rel } U_\pm}.
\]
\end{lemma}

\begin{proof}
Without loss of generality, we consider just the case
\[
\Phi^{+}_d :  \dabsfact{U_{-}, G , w}  \to \dabsfact{U_{+}, G , w}.
\]

Let $\F$ be an object of  $\dabsfact{U_-, G , w}$ whose support does not intersect $U_- \backslash V_-$.  The functor $\Phi^{+}_d $ is constructed in the proof of Theorem 3.5.2 of \cite{BFK12}.  By definition, it is constructed precisely so that there is an object $\tilde{F} \in \dabsfact{U, G , w}$ whose restriction to $U_-$ is $\F$ and whose restriction to $U_+$ is $\Phi^{+}_d (\F)$.  This means that the support of $\tilde{F}$ is contained in $\op{Supp } \F \cup S_\lambda$.  Now, 
\[
(\op{Supp } \F \cup S_\lambda) \cap U \backslash V = \emptyset
\]
by the assumption that the wall-crossing is compatible.  
Hence,
\begin{align*}
\op{Supp } \Phi^{+}_d (\F) \cap U_+ \backslash V_+ & \subseteq (\op{Supp } \F \cup S_\lambda \cap U_+) \cap U \backslash V \\
& =  \emptyset,
\end{align*}
as desired.

\end{proof}

Now, suppose $V \to U$ is a $G$-equivariant open immersion which is compatible with an elementary wall-crossing $V = V_{\pm} \sqcup S_{\pm}$ such that $[V_+ / G]$ is isomorphic to $\op{tot }\mathcal E^\vee$ over $[Y/G]$.
Denote by  $\dabsfact{V_-, G, w}^{\op{perf}}$ the full subcategory of $\dabsfact{V_-, G, w}$ whose image under $\Omega^{-1} \circ \Phi^{+}_d$ lies in $\op{Perf}[Z/G]$. 

Assume further that the zero section of $[V_+/G]$ does not intersect $[U_+ \backslash V_+ / G]$.  Then,
 by Lemmas~\ref{lem: zerosec} and \ref{lem: restricts}, $\dabsfact{V_-, G, w}^{\op{perf}}$ is a full subcategory of $\dabsfact{V_-, G, w}_{ \op{rel } U_-}$.

Finally, recall that we have a pair of functors
\[
i_*: \dabsfact{V_-, G, w}_{\op{rel }U_-} \to \dabsfact{U_-, G, w}
\]
\[
i^*: \dabsfact{U_-, G, w} \to \dabsfact{V_-, G, w}.
\]
These restrict to a pair of functors
\[
i_*: \dabsfact{V_-, G, w}^{\op{perf}} \to \dabsfact{U_-, G, w}
\]
\[
i^*: \dabsfact{U_-, G, w} \to \dabsfact{V_-, G, w}.
\]

\begin{theorem}\label{theorem: CCRGIT}
Suppose $V \to U$ is a $G$-equivariant open immersion which is compatible with an elementary wall-crossing $V = V_{\pm} \sqcup S_{\pm}$ such that:
\begin{itemize}
\item $[V_+ / G]$ is isomorphic to $\op{tot }\mathcal E^\vee$ over $[Y/G]$,
\item the zero section of $[V_+/G]$ does not intersect $[U_+ \backslash V_+ / G]$,
\item $t(\mathfrak{K}^+) \leq  t(\mathfrak{K}^-)$, that $w$ extends to $U$ as a section of $\O(\chi)$, 
\item $[U_- /G]$ has finite diagonal, that $[\partial w / G] \subseteq [U_- /G]$ is proper over $\op{Spec }\kappa$, and 
\item $\partial w  \subseteq Z(w) $ in $V_- $. 
\end{itemize}
Then, the functors
\[
i_* :    \dabsfact{V_-, G, w}^{\op{perf}}  \to \dabsfact{U_-, G, w}
\]
\[
i^* :  \dabsfact{U_-, G, w} \to \dabsfact{V_-, G, w} 
\]
form a crepant categorical resolution.
\end{theorem}

\begin{proof}
The assumptions assure that  $\dabsfact{U_-, G, w}$ is the homotopy category of a homologically smooth and proper dg-category by Lemmas~\ref{lem: homologically smooth} and \ref{lem: proper} and Proposition 5.11 of \cite{BFK11}.  Since $i$ is an open immersion, the functors are both left and right adjoint and $i^* \circ i_*$ is the identity.
\end{proof}

\begin{remark}
Theorem~\ref{theorem: CCRGIT} has a natural context for resolving factorization categories associated to Landau-Ginzburg models that correspond non-smooth toric complete intersections. This is seen explicitly in an example in Proposition~\ref{CCRsPlanes} below.
\end{remark}
 
 \section{Toric Landau-Ginzburg Models: Their Cones and Phases} \label{sec:3}

 \subsection{Polytopes and Gorenstein Cones}\label{ConesBestiary}
 In this subsection, we will review standard definitions in order to set notation. Good references are \cite{CLS, BN07}. Let $M$ and $N$ be dual lattices of dimension $d$ and $N_{\R} := N \otimes_{\Z} \R$. Let $\sigma$ be a strictly convex cone in $N_{\R}$ of dimension $d$. Recall that the dual cone $\sigma^\vee$ in $M_{\R}$ is defined to be
 $$
 \sigma^\vee := \{ m \in M_{\R} \ | \  \langle m, n\rangle \geq 0 \text{ for all }n \in \sigma\}.
 $$
 
 \begin{definition}
A full-dimensional strictly convex rational polyhedral cone $\sigma \subseteq N_{\R}$ is called 
\begin{enumerate}[(i.)]
\item \newterm{Gorenstein} if there exists an element $\mathfrak{m} \in M$ so that the semigroup $\sigma \cap N$ is generated by finitely many lattice points $n \in N$ that are contained in the the affine hyperplane $\{ n \in N_{\R} \ | \  \langle \mathfrak{m}, n \rangle = 1\}$; 
\item  \newterm{almost Gorenstein} if there exists an element $\mathfrak{m} \in M$ so that the cone is generated over $\Q$ by finitely many lattice points in $\{n \in N  \ | \  \langle \mathfrak{m}, n \rangle =1\}$; and
\item \newterm{$\Q$-Gorenstein} if there exists an element $\mathfrak{m} \in M_{\Q}$ so that the cone is generated over $\Q$ by finitely many lattice points in $\{n \in N_{\R} \ | \  \langle \mathfrak{m}, n \rangle =1\}$.
\end{enumerate}
 \end{definition}
 
 \begin{example}
\begin{enumerate}[(i.)]
\item With respect to the lattice $N = \Z^2$, the cone $\sigma=\op{Cone}( (1,1), (-1,1))$ is Gorenstein with $\mathfrak{m} = (0,1)$. 
\item With respect to the lattice $N= \Z^4$, the cone $\sigma=\op{Cone}((1,0,0,0),(0,1,0,0), (0,0,1,0), \newline(-1,-1,-1,2))$ is almost Gorenstein with $\mathfrak{m}=(1,1,1,2)$, but not Gorenstein. Indeed, $n=(0,0,0,1)$ is a generator of the semigroup $\sigma\cap N$, but $\langle \mathfrak{m},n\rangle = 2$.
\item With respect to the lattice $N=\Z^2$, the cone $\sigma=\op{Cone}((1,2),(-1,2))$ is $\Q$-Gorenstein with $\mathfrak{m} = (0, \frac12)\in M_{\Q}$, but not almost Gorenstein.
\end{enumerate}
 \end{example}
As the cone $\sigma$ is full-dimensional, the lattice element $\mathfrak{m}$ is unique.  Moreover, $\mathfrak{m}$ is in the interior of the dual cone $\sigma^\vee$, since it does not pair to 0 with any nonzero element of the cone $\sigma$. We define the \newterm{$k$-th slice} of the cone $\sigma$ to be the polytope 
 $$
 \sigma_{(k)} := \{ n \in \sigma \ | \  \langle \mathfrak{m}, n \rangle = k \}.
 $$ 

  If, in addition, the dual cone $\sigma^\vee$ is a $\Q$-Gorenstein cone with respect to an element $\mathfrak{n} \in N_{\Q}$, we can define the \newterm{index} $r$ of $\sigma$ to be the pairing $\langle \mathfrak{m}, \mathfrak{n}\rangle$. Since $\mathfrak{m} \in M_{\Q}$  and $\mathfrak{n} \in N_{\Q}$ defined above are unique, the index is well-defined.  Note that the index may be a rational number.
 
Let us now take a minor detour to the realm of polytopes in order to setup the definition of $t$-split $\Q$-Gorenstein cones. Take $M$ to be a lattice.  Consider $t$ lattice polytopes $\Delta_1, \ldots, \Delta_t$ that are positive dimensional in a real vector space $M_{\R}$. We define the \newterm{Cayley polytope} $\Delta_1*\cdots * \Delta_t$ associated to the polytopes $\Delta_1, \ldots, \Delta_t$ to be the polytope in the vector space $M_{\R} \oplus \R^t$ defined by
 $$
 \Delta_1*\cdots * \Delta_t := \op{Conv}( (\Delta_1, e_1), \ldots, (\Delta_t, e_t)),
 $$
 where $e_i$ are the elementary basis vectors for the vector space $\R^t$. 
 \begin{definition}
A polytope $\Delta$ is called a \newterm{Cayley polytope of length $t$} if 
\[
\Delta =  \Delta_1*\cdots * \Delta_t 
\]
for some $\Delta_1, .., \Delta_t$.
 \end{definition}
  
\subsection{Toric Vector Bundles}\label{ToricVectorBundles}

In this section, we will give examples of algebro-geometric manifestations of the cones described in the previous subsection. They show up as supports of fans associated to certain toric vector bundles. 

Recall the following construction of a split toric vector bundle over a toric variety.
 Start with a toric variety $X_{\Sigma}$ associated to a fan $\Sigma \subseteq N_{\R}$. Any torus-invariant Weil divisor $D$ can be written as a linear combination of torus-invariant divisors associated to rays, i.e., 
\[
D = \sum_{\rho \in \Sigma(1)} a_\rho D_\rho
\]
 for some $a_\rho \in \Z$.  Take $r$ such torus-invariant Weil divisors $D_i = \sum_{\rho \in \Sigma} a_{i \rho} D_\rho$ for some $a_{i\rho} \in \Z$. Let $u_\rho$ be the primitive generator of the ray $\rho \in \Sigma(1)$. For all $\sigma \in \Sigma$, define the cone
$$
\sigma_{D_1,\ldots, D_r} := \op{Cone}(\{u_\rho - a_{1\rho} e_1 - \ldots -a_{r\rho}e_r \ | \  \rho \in \sigma(1)\} \cup \{ e_i \ | \  i \in \{1, \ldots, r\})\subset N_{\R} \oplus \R^r.
$$
Take the fan $\Sigma_{D_1, \ldots, D_r}$ to be the fan generated by the cones $\sigma_{D_1, \ldots, D_r}$ and their proper faces. Recall that if $D_i$ are Cartier, then by iterating Proposition 7.3.1 of \cite{CLS}, we can see that the toric variety $X_{\Sigma_{D_1, \ldots, D_r}}$ is the vector bundle $\bigoplus_{i = 1}^r \O(D_i)$ over $X_{\Sigma}$.

Now we describe when the support $|\Sigma_{{D_1}, \ldots, D_r}|$ of the fan $\Sigma_{{D_1}, \ldots, D_r}$ associated to the toric vector bundle is one of the special cones described in Section~\ref{ConesBestiary}.  Assume that $X_\Sigma$ is semiprojective and the divisors $D_i$ are $\Q$-Cartier and anti-nef.

\begin{lemma}[Lemma 5.19 of \cite{FK14}]
Let $\Sigma$ be a fan and suppose $X_{\Sigma}$ is semiprojective. If $-D$ is nef and $\Q$-Cartier, then $X_{\Sigma_{D}}$ is semiprojective.
\end{lemma}
\begin{corollary}\label{totSemiprojective}
Let $\Sigma$ be a fan and suppose $X_{\Sigma}$ is semiprojective. If $-D_1, ..., -D_r$ are nef and $\Q$-Cartier, then $X_{\Sigma_{D_1, \ldots, D_r}}$ is semiprojective.
\end{corollary}
\begin{proof}
This follows immediately by induction on $i$.
\end{proof}
We can describe the dual cone $|\Sigma_{D_1, \ldots, D_r}|^\vee$ explicitly. Such a description was given by Mavlyutov (Lemma 1.6 of \cite{Mav}) for the case when $\sum_i D_i = -K_{X_{\Sigma}}$. In \cite{FK14}, this hypothesis is dropped:
\begin{lemma}[Lemma 5.17 of \cite{FK14}]\label{dualConeIsCayley}
Let $\Sigma$ be a complete fan and 
$$
D_i = \sum_\rho a_{i\rho} D_\rho
$$
be nef and $\Q$-Cartier divisors. The dual cone to $|\Sigma_{-D_1, \ldots, -D_r}|$ is equal to the Cayley cone on the set of polytopes 
$$
\Delta_i := \{ m \in M_{\R} \ | \ \langle m, u_\rho\rangle \geq -a_{i\rho} \text{ for all } \rho \in \Sigma(1)\}
$$
i.e., 
$$
|\Sigma_{-D_1, \ldots, -D_r}|^\vee = \R_{\geq 0} (\Delta_1*\cdots *\Delta_r) = \R_{\geq 0} (\Delta_1+e_1^*) + \ldots + \R_{\geq 0}(\Delta_r + e_r^*).
$$
Moreover, if the divisors $D_i$ are all Cartier, then $\Delta_i$ are lattice polytopes.
\end{lemma}

The cone $|\Sigma_{D_1, \ldots, D_r}|$ can be any of the four types of strictly-convex cones:  Gorenstein, almost Gorenstein, $\Q$-Gorenstein, or not $\Q$-Gorenstein. We give examples of all four.

\begin{example}
Given a toric Fano variety $X_\Sigma$. Take a nef partition $D_1, \ldots, D_r$ of Cartier divisors of its anti-canonical bundle $-K_{X_\Sigma}$, i.e., nef divisors $D_i$ such that
\[
\sum D_i = -K_{X_\Sigma}.
\]
We get a vector bundle $\op{tot}(\bigoplus_{i=1}^r \O(-D_i))$ with anti-canonical determinant.  The corresponding cone $|\Sigma_{D_1, ..., D_r}|$ is a completely split Gorenstein cone. See \cite{BN07} for details.
\end{example}

\begin{example} Let $u_i$ be the standard basis for $\Z^n$ and set $u_0 = -\sum_i u_i$.  Let $\Sigma \subseteq \R^n$ be the complete fan on the rays $\rho_i$ generated by the $u_i$.  The corresponding toric variety is $\P^n$.  The fan $\Sigma_{-2D_{\rho_0}, -2D_{\rho_0}}$ gives the vector bundle $\op{tot}(\O(-2)^{\oplus 2})$.  

The cone $|\Sigma_{-2H, -2H}|$ is generated by $u_1, ..., u_n, e_1, e_2, u_0 + 2e_1 + 2e_2$.  Note that $u_1, ..., u_n, e_1, e_2,$ and $u_0+2e_1+2e_2$ are all extremal generators of the cone $|\Sigma_{-2D_{\rho_0}, -2D_{\rho_0}}|$. If $|\Sigma_{-2D_{\rho_0}, -2D_{\rho_0}}|$ is $\Q$-Gorenstein, then we must have $\mathfrak{n} = u_1^*+ ... + u_n^* + e_1^* + e_2^*$.  But
\[
\langle u_1^*+ ... + u_n^* + e_1^* + e_2^*, u_0 + 2e_1 + 2e_2 \rangle = 4-n.
\]  
Hence,  $|\Sigma_{-2D_{\rho_0}, -2D_{\rho_0}}|$ is almost Gorenstein if and only if $n=3$. If $n>3$, then $|\Sigma_{-2D_{\rho_0}, -2D_{\rho_0}}|$ is not $\Q$-Gorenstein.
\end{example}

\begin{proposition}\label{hypersurfacesExample}
Let $X_\Sigma$ be a projective simplicial toric variety, and let $D$ be a Weil divisor linearly equivalent to $-qK_{X_\Sigma} $ for some positive rational number $q$. If $D$ is nef and $\Q$-Cartier, then the cone $|\Sigma_{-D}|$ is $\Q$-Gorenstein. Moreover, if $q = \frac1r$ for some positive integer $r$ and $D$ is Cartier, then $|\Sigma_{-D}|$ is almost Gorenstein.
\end{proposition}

\begin{proof}
Write $D = \sum_\rho a_\rho D_\rho$ for some $a_\rho\in \Z$. Since $D$ is nef, it suffices to find an element $(m,t) \in (M\times\Z)_{\Q}$ so that 
$$
\langle (m,t), (u_\rho, a_\rho)\rangle = 1
$$
for all $\rho \in \Sigma(1)$.

Consider the projection $\pi: N\times \Z \rightarrow  N$ that induces the projection $\pi: X_{\Sigma_{-D}} \rightarrow X_{\Sigma}$. Let $\rho_b$ be the ray in $\Sigma_{-D}$ given by the one-dimensional cone $\op{Cone}(0,1)$.  Consider the exact sequence
$$
M \times \Z \overset{f_{\Sigma_{-D}(1)}}{\longrightarrow} \Z^{\Sigma_{-D}(1)} \longrightarrow \op{Cl}(X_{\Sigma_{-D}}) \longrightarrow 0.
$$
The first map is defined by $(m,t) \mapsto \sum_{\rho \in \Sigma(1)} \langle (m,t), (u_{\rho}, a_\rho)\rangle e_{\bar \rho} + te_{\rho_b}$, where $\bar \rho:=\op{Cone}(u_{\rho}, a_\rho)$ is the ray in $\Sigma_{-D}(1)$ that corresponds to $\rho \in \Sigma(1)$. The image of $(0,1)$ under the map $f_{\Sigma_{-D}}(1)$ is 
$$
f_{\Sigma_{-D}(1)}(0,1) = \sum_{\rho \in \Sigma(1)} a_\rho e_{\bar{\rho}} + e_{\rho_b}.
$$
Thus, in $\op{Cl}(X_{\Sigma_{-D}})$, we have the equality 
\begin{equation} \label{eq: divisor compare}
- \sum_{\rho \in \Sigma(1)} a_\rho D_{\bar{\rho}}  = D_{\rho_b}.
\end{equation}

By Proposition 4.2.7 of \cite{CLS}, since $X_{\Sigma}$ is simplicial, we know that, for any $\rho_i \in \Sigma(1)$, there is a $d_i \in \N$ so that $d_iD_{\rho_i}$ is Cartier. Note that by Proposition 6.2.7 of \cite{CLS}, we know that the support function for the pullback $\pi^* d_iD_{\rho_i}$ is given by the composition $|\Sigma_{-D}| \overset{\pi}{\rightarrow} |\Sigma| \overset{\varphi_{d_iD_{\rho_i}}}{\rightarrow} \R$ where $\varphi_D(u_{\rho}) = -d_i$ if $\rho = \rho_i$ and $\varphi_D(u_\rho)=i$ otherwise. Moreover, since $\pi(\rho_b) = 0$, the support function of any pullback of any divisor on $X_{\Sigma}$ will map $\rho_b$ to zero. Hence, by the support function description of the pullback, we can see that $\pi^*(d_iD_{\rho_i} )= d_iD_{\bar\rho_i}$ in $\op{Cl}(X_{\Sigma_{-D}})$.  Let $d := \prod d_i$. Plugging into \eqref{eq: divisor compare}, we obtain
\begin{equation} \label{eq: divisor compare 2}
d \pi^* D = d D_{\rho_b}
\end{equation}
in $\op{Cl}(X_{\Sigma_{-D}})$.

Now, note that $D = qK_{X_\Sigma}$ hence $-q\sum_{\rho \in \Sigma(1)} D_{\bar \rho} = q\pi^*K_{X_{\Sigma}} = \pi^* D = D_{\rho_b}$. Thus we have the equality 
\begin{equation} \label{eq: divisor compare 3}
\sum_{\rho \in \Sigma(1)}D_{\bar \rho} + \frac{1}{q} D_{\rho_b} = 0
\end{equation}
in $\op{Cl}(X_{\Sigma_{-D}})\otimes \Q$. By applying $-\otimes \Q$ to the exact sequence we started with, we have
$$
(M \oplus \Z)_{\Q} \overset{f_{\Sigma_{-D}(1)}}{\longrightarrow} \Q^{\Sigma_{-D}(1)} \longrightarrow \op{Cl}(X_{\Sigma_{-D}})\otimes \Q \longrightarrow 0.
$$
Since $\sum_{\rho \in \Sigma(1)}e_{\bar \rho} + \frac{1}{q} e_{\rho_b} $ is in the kernel of the second map, there exists an element $(m,t) \in (M \oplus \Z)_{\Q}$ such that $f_{\Sigma_{-D}(1)}(m,t) = \sum_{\rho \in \Sigma(1)}e_{\bar \rho} + \frac{1}{q} e_{\rho_b} $. By the definition of $f_{\Sigma_{-D}(1)}$, we then have that
$$
\langle (m,t), (u_\rho, a_\rho) \rangle = 1
$$
for all $\rho\in \Sigma(1)$ and
\begin{equation}\label{PairingForHypersurfaces}
\langle (m,t), (0,1) \rangle = \frac{1}{q}.
\end{equation}

In the case that $D$ is Cartier, we obtain \eqref{eq: divisor compare 2} in the Picard group.  Since the Picard group has no torsion by Proposition 4.2.5 of \cite{CLS}, this yields an equality
\[
 \pi^*D = D_{\rho_b}
 \]
in $\op{Pic}(X_{\Sigma_{-D}})$.  Furthermore, if $q = \frac{1}{r}$ for some positive integer $r$, then \eqref{eq: divisor compare 3} holds in $\op{Pic}(X_{\Sigma_{-D}})$.
Thus by the same logic above there exists an $(m,t) \in M\oplus \Z$ so that 
$$
\langle (m,t), (u_\rho, a_\rho) \rangle = 1
$$
for all $\rho\in \Sigma(1)$ and $\langle (m,t), (0,1)\rangle = r$.
\end{proof}

\begin{remark}
We do not know the appropriate generalization for complete intersections except when $q=1$.  In this case, $|\Sigma_{-D_1, ..., -D_r}|$ and $|\Sigma_{-D_1, ..., -D_r}|^\vee$ are Gorenstein of index $r$ if and only if $\sum D_i = - K$, see Proposition 3.6 of \cite{BB97}.  
\end{remark}

    \subsection{Toric Stacks Associated to Fans}
  
We now define a quotient stack $\mathcal{X}_{\Sigma}$ that is associated to the fan $\Sigma$ that is the quotient of an open subset of affine space by an abelian group. This quotient stack will be isomorphic to the toric variety $X_{\Sigma}$ when the toric variety is smooth. Let $n$ be the number of rays in the fan $\Sigma$.
We can associate a new fan $\op{Cox}(\Sigma) \subseteq \R^{\Sigma(1)}$ to $\Sigma$ that is defined to be
\begin{equation}\label{CoxFan}
\op{Cox}(\Sigma) := \{ \op{Cone}(e_\rho | \ \rho \in \sigma) | \ \sigma \in \Sigma \}.
\end{equation}
Enumerating the rays, this fan is a subfan of the standard fan for $\A^n$:
\[
\Sigma_n := \{ \op{Cone} (e_i | \ i \in I) | \ I \subseteq \{1, ..., n \} \}.
\]
 Hence the toric variety $U_{\Sigma} := X_{\op{Cox}(\Sigma)}$ is an open subset of $\A^{n}$. We now define the group $S_{\Sigma(1)}$ which acts on $U_{\Sigma}$.
 
We now describe a quotient associated to a set of lattice elements $\nu = (v_1, ..., v_n) \subseteq N$ where $N$ is a lattice of dimension $d$. We will focus on the case where $\nu = \{u_\rho \ | \ \rho \in \Sigma(1)\}\subseteq N,$ where $u_\rho$ is the primitive lattice generator of the ray $\rho$. Let $M$ be the dual lattice to $N$.  We get a right exact sequence
\begin{align}
\label{eq: f}
M & \overset{f_\nu}{\longrightarrow} \Z^n \overset{\pi}{\longrightarrow} \op{coker}(f_\nu) \to 0 \\
m & \mapsto \sum_{i=1}^n \langle v_i, m \rangle e_i.  \notag
\end{align}
Applying $\op{Hom}(-, \gm)$ we get a left exact sequence
\begin{equation}\label{eq: hatpi}
0 \longrightarrow \op{Hom}(\op{coker}(f_\nu), \gm) \overset{\widehat{\pi}}{\longrightarrow} \gm^{n} \overset{\widehat{f_{\nu}}}{\longrightarrow} \gm^m
\end{equation}
We set 
\begin{equation}\label{defn:Snu}
S_\nu := \op{Hom}(\op{coker}(f_\nu), \gm)
\end{equation}
We write $S_{\Sigma(1)}$ for $S_\nu$ when $\nu = \{ u_\rho \ | \  \rho \in \Sigma(1)\}$.
\begin{definition}
We call $U_{\Sigma}$ the \newterm{Cox open set} associated to $\Sigma$.  We define the \newterm{Cox stack} associated to $\Sigma$ to be
\[
\mathcal{X}_{\Sigma}:= [U_{\Sigma}/S_{\Sigma(1)} ].
\]

\end{definition}

The Cox stack is called the canonical toric stack in the previous literature \cite{FMN}. 

\begin{theorem}
If $\Sigma$ is simplicial, then $\mathcal{X}_{\Sigma}$ is a smooth Deligne-Mumford stack with coarse moduli space $X_\Sigma$.  When $\Sigma$ is smooth (or equivalently $X_\Sigma$ is smooth) $\mathcal{X}_{\Sigma} \cong X_\Sigma$.
\label{thm: stack realization}
\end{theorem}
\begin{proof}
The first statement is Theorem 4.11 of \cite{FMN}. It also follows from a combination of Proposition 5.1.9 and Theorem 5.1.11 in \cite{CLS}, which also gives the second statement.
\end{proof}

Note that $S_{\Sigma(1)} \subseteq \gm^{|\Sigma(1)|}$.  Note that any element in $\chi \in \op{coker}(f_\nu)$ gives a map $\chi: S_{\Sigma(1)} \rightarrow \gm$. Consequently, each ray $\rho \in \Sigma(1)$ gives a character $\chi_\rho$ of $S_{\Sigma(1)}$ given by the element $\pi(e_\rho) \in \op{coker}(f_\nu)$. Hence, given a divisor $D = \sum a_{\rho}D_\rho$ on $X_\Sigma$, we associate a character $\chi_D := \prod_\rho \chi_\rho^{a_\rho}$ of $S_{\Sigma(1)}$, defined by the element $\pi(\sum_\rho a_\rho e_\rho)$. The total space $\op{tot}(\O_{\mathcal{X}_{\Sigma}}(\chi_D))$ is a quotient stack given by $S_{\Sigma(1)}$ acting on $U_{\Sigma}\times \C$ induced by the standard action on $U_{\Sigma}$ and the character on $\C$. This can be done iteratively for a split vector bundle.
 
We can use this dictionary to move between split vector bundles over toric varieties and quotient stacks.  Namely, we have the following proposition.
\begin{proposition}[Proposition 5.16 of \cite{FK14}]
Let $D_1, \ldots, D_r$ be divisors on $X_\Sigma$. There is an isomorphism of stacks
\begin{equation}
\mathcal{X}_{\Sigma_{D_1, \ldots, D_r}} \cong \op{tot}(\bigoplus_{i=1}^r \O_{\mathcal{X}_{\Sigma}}(\chi_{D_i})).
\label{eq: vector bundle stack iso}
\end{equation}
\end{proposition}

We can break the above proposition into the following two lemmas.  These are already implicit in the proof of Propoisition 5.16 of \cite{FK14} but we include them here for completeness.

\begin{lemma}\label{EqualityOfSnus}
There is a group isomorphism 
\begin{equation}
S_{\Sigma(1)} \cong S_{\Sigma_{D_1, \ldots, D_r}(1)}.
\label{eq: vector bundle group isomorphism}
\end{equation}
\end{lemma}
\begin{proof}
First note that  $\Z^{\Sigma_{D_1,\ldots, D_t}(1)}=\Z^{\Sigma(1)} \times \Z^t$. We write the generators of the direct summands of this decomposition as $e_\rho$ and $e_i$, respectively. Using Equation~\eqref{eq: f}, construct a commutative diagram 
\begin{equation}
\begin{CD}
0 @>>> M\oplus\Z^t @>f_{\Sigma_{D_1, \ldots, D_t}(1)} >> \Z^{\Sigma(1)+t} @>\pi>> \op{coker}(f_{\Sigma_{D_1, \ldots, D_t}(1)}) @>>> 0\\
@. @V\op{proj}_M VV @VgVV @VVV @.\\
0 @>>> M @>f_{\Sigma(1)} >> \Z^{\Sigma(1)} @>\pi>> \op{coker}(f_{\Sigma(1)})  @>>> 0,
\end{CD}
\label{eq: projection diagram}
\end{equation}
where $\op{proj}_M: M\oplus\Z^t \rightarrow M$ is the standard projection and $g$ is defined by
\begin{equation}\begin{aligned}
g: \Z^{\Sigma(1)+t} &\rightarrow \Z^{\Sigma(1)}, \op{ where} \\ 
e_\rho &\mapsto e_\rho, \\ 
e_i &\mapsto \sum_\rho a_{i\rho} e_\rho.
\end{aligned}\end{equation}
The final vertical map is induced by the first two. The kernel of $\op{proj}_M$ is $\Z^t$ and the cokernel of $\op{proj}_M$ is trivial, and the kernel of $g$ is $\Z^t$ and the cokernel of $g$ is trivial. This gives an equality of cokernels. As $S_{\Sigma_{D_1, \ldots, D_r}(1) } := \op{Hom}(\op{coker}(f_{\Sigma_{D_1, \ldots, D_t}(1)}), \gm)$ and $S_{\Sigma(1)} : =  \op{Hom}(\op{coker}(f_{\Sigma(1)}), \gm)$, this commutative diagram induces an equality of groups.
\end{proof}

\begin{lemma}\label{DecompCoxStack}
We have an isomorphism of quasi-affine varieties
$$
U_{\Sigma_{D_1, \ldots, D_t}} = U_{\Sigma} \times \A^t
$$
which induces the isomorphism of stacks \eqref{eq: vector bundle stack iso}.
\end{lemma}
\begin{proof}
Consider the fan $\Sigma$ above. Let $\Sigma_t := \{ \op{Cone}(e_i) \ | \ i \in I) \ | \ I \subset \{1, \ldots, t\}\}$.  Note that 
\begin{equation}\begin{aligned}
\op{Cox}(\Sigma_{D_1, \ldots, D_t}) &= \{\op{Cone}(e_\rho \ | \ \rho \in \sigma) \ | \ \sigma \in \Sigma_{D_1, \ldots, D_t} \} \\ 
	&= \{ \op{Cone}(e_\rho \ | \ \rho \in \sigma) \ | \ \sigma \in \Sigma\} \times \Sigma_t\\
	&= \op{Cox}(\Sigma) \times \Sigma_t.
\end{aligned}\end{equation}
The first line is by definition, the second line comes from all maximal cones in $\Sigma_{D_1, \ldots, D_t}$ containing the rays generated by $e_i$ for all $i$, and the third line is by definition.

Now, the action of $S_{\Sigma(1)}=S_{\Sigma_{D_1,\ldots, D_t}(1)}$ on $U_{\Sigma_{D_1, \ldots, D_t}} = U_{\Sigma} \times \A^t$ is described by \eqref{eq: projection diagram}.  This shows that $S_{\Sigma_{D_1,\ldots, D_t}(1)}$ acts on $U_{\Sigma}$ via the isomorphism with $S_{\Sigma(1)}$. Moreover, on the $i^{\op{th}}$-coordinate $u_i$ of $\A^t$, it acts via the character $\chi_{D_i}$.  This gives the isomorphism of stacks \eqref{eq: vector bundle stack iso}.
\end{proof}

  \subsection{VGIT on Affine Space}

Take an affine space 
\[
X:= \mathbb{A}^{n+r} = \op{Spec }\kappa[x_1, ..., x_n, u_1, ..., u_r].
\]
Consider the open dense torus $\gm^{n+r}$ with the standard embedding and action on $X$. Take a subgroup $S\subseteq \gm^{n+r}$.  For $1 \leq i \leq n+r$ we get a character $\chi_i$ from the composition of the inclusion and the projection onto the $i^{\op{th}}$ summand.

\begin{definition}
Let $\times : \gm^{n+r} \rightarrow \gm$ be the multiplication map. We say that $S$ satisfies the \newterm{quasi-Calabi-Yau condition} if $\times|_S$ is torsion.
\end{definition}

\begin{remark}
The quasi-Calabi-Yau condition is equivalent to the sum $\sum_{i=1}^{n+r} \chi_i$ being torsion.
\end{remark}

The reason for the distinction between the variables $x_i, u_i$ is that we equip $\A^{n+r}$ with an additional $\gm$-action.  Namely, for $\lambda \in \gm$ we define 
\begin{equation}\begin{aligned}\label{RchargeAction}
\lambda \cdot x_i & := x_i \\
\lambda \cdot u_i & := \lambda u_i 
\end{aligned}\end{equation}
and call this $\gm$-action \newterm{$R$-charge}.

This gives an action $S \times \gm$ on $\A^{n+r}$ and defines a distinguished character $\chi_R$ coming from projection onto the second factor.  A \newterm{superpotential} is a semi-invariant function $w$ with respect to $\chi_R$. The data $(\A^{n+r}, S\times \gm, w, \chi_R)$ is a \newterm{gauged Landau-Ginzburg model}.  We can restrict this data to any invariant open subset of $\A^{n+r}$ to get various new gauged Landau-Ginzburg models.  

In this paper, we choose such open sets using geometric invariant theory (GIT) for the action of $S$ on $\A^{n+r}$.  Let us now review this story, which is called variation of geometric invariant theory quotients (VGIT). The possible GIT quotients of $\A^{n+r}$ by $S$ correspond to a choice of a (rational) character 
\[
\chi \in \hat{S}_{\Q} := \op{Hom}(S, \gm) \otimes_{\Z} \Q. 
\]
That is, given a $\chi \in \hat S_{\Q}$ rationalize the denominator to get $G$-equivariant line bundle $\O(d\chi)$ for some $d > 0$.  Geometric invariant theory determines an open subset $U_{\chi}$ of $\A^{n+r}$ called the semi-stable locus. 

Partition $ \hat S_{\Q}$ into the subsets 
\[
\sigma_\chi := \{ \tau \in  \hat S_{\Q} \ | \ U_{\tau} = U_{\chi} \}.
\]
It turns out that each $\sigma_\chi$ is a cone and the set of all such cones form a fan $\Sigma_{GKZ}$ in $\hat S_{\Q}$ called the \newterm{GKZ-fan} (or \newterm{secondary fan}).  The maximal cones of this fan are called \newterm{chambers} and the codimension one cones are called \newterm{walls}. There are only a finite number of chambers in the fan $\Sigma_{GKZ}$. For any character $\chi_p$ in the interior of a chamber $\sigma_p$, we denote by $U_p$ the open set of $\A^{n+r}$ that consists of the semi-stable points with respect to the associated line bundle to $\chi_p$.  

\begin{theorem}
For any two chambers $\sigma_p$ and $\sigma_q$, if $S$ satisfies the quasi-Calabi-Yau condition, then there is an equivalence of categories:
$$
\op{D}^{\op{abs}}[U_p,   S \times \gm, w] \cong \op{D}^{\op{abs}}[U_q,   S \times \gm, w].
$$
\end{theorem}
\begin{proof}
This is a consequence of Theorem~\ref{thm: BFKVGIT}.
As for why this works for all chambers, one can use Theorem 14.4.7 of \cite{CLS}.  Namely, one can get from any chamber to any other chamber by a sequence of elementary wall crossings as the GKZ-fan has convex support.

 As we are in the toric setting, the result, in fact, goes back to Theorem 3 of \cite{HW12}.  Another version of this result can be found in \cite{HL15} Corollary 4.8 and Proposition 5.5.

\end{proof}

 \subsection{Variation of Geometric Invariant Theory for Toric Stacks and the Secondary Fan}\label{CohTri}
  
Here, we consider the geometric invariant theory associated to  the quotient of the affine space  $X := \A^{n}$ by the abelian group $S: = S_\nu$, as defined in Equation~\ref{defn:Snu}.   As it turns out, the different GIT quotients of $X$ by $S$ have an interpretation both in terms of the secondary fan and in terms of fans whose rays have primitive generators in the point collection $\nu$. 

Following \S 15.2 of \cite{CLS}, take $\nu= (v_1, \ldots, v_n)$ to be a collection of distinct, nonzero points in $N$. 

\begin{definition}
We say that $\nu$ is \newterm{geometric}  if each $v_i \in \nu$ is nonzero and generates a distinct ray in $N_{\Q}$.
\end{definition}

\begin{proposition}[Exercise 15.1.8 in \cite{CLS}]\label{prop:regTriangFans}
Suppose $\nu$ is geometric.  Then there is a bijective correspondence between chambers of the secondary fan and simplicial fans $\Sigma$ such that $\Sigma(1) \subseteq \{ \op{Cone}(v_i) \ | \  1 \leq i \leq n \}$, $|\Sigma| = \op{Cone}(\nu)$, and $X_\Sigma$ is semiprojective. 
\label{prop: GKZ chamber bijection}
\end{proposition}

Allow us to describe the cones in the secondary fan in a bit more detail. 
Tensoring the short exact sequence in \eqref{eq: f} with $\Q$, we get the sequence
\[
M_{\Q} \stackrel{f_\nu}{\longrightarrow} \Q^n \stackrel{\pi}{\longrightarrow} \op{coker}(f_\nu)\otimes \Q \to 0.
\]
Here, the vector space $(\hat{S}_\nu)_{\Q} =  \op{coker}(f_\nu)\otimes \Q$ is the space in which the secondary fan $\Sigma_{\text{GKZ}}$ lives. 

Note that if we take the standard basis vectors $e_i$ for $\Q^n$, then the support of the secondary fan $|\Sigma_{\text{GKZ}}|$ is the cone $\op{Cone}(\pi(e_i))$. We now will give the structure of the fan that comes from a decomposition, but first let us give a definition in order to give a general setup. 

\begin{definition}[Definition 6.2.2 of \cite{CLS}]
A generalized fan $\Sigma$ in $N_{\R}$ is a finite collection of cones $\sigma \subseteq N_{\R}$ so that
\begin{enumerate}
\item every $\sigma \in \Sigma$ is a rational polyhedral cone,
\item for any $\sigma \in \Sigma$, each face of $\sigma$ is also in $\Sigma$, and
\item for any $\sigma_1, \sigma_2 \in \Sigma$, the intersection $\sigma_1\cap\sigma_2$ is a face of each.
\end{enumerate}
\end{definition}

Start with a generalized fan $\Sigma$ in $N_{\R}$ and a set $I_{\emptyset} \subset \nu$ so that the support $|\Sigma|$ is $\op{Cone}(\nu)$, the toric variety $X_{\Sigma}$ is semiprojective, and we can write any cone $\sigma \in \Sigma$ as a cone $\op{Cone}(v_i \ | \  v_i \in \sigma, v_i \notin I_{\emptyset})$. Given such a pair $(\Sigma, I_{\emptyset})$, we define a cone in $\Q^n$ 

\begin{equation*}
\tilde \Gamma_{\Sigma, I_{\emptyset}} := \left\{ 
 (a_i) \in \Q^n  \ \middle | \ 
   \begin{aligned}
      & \text{there exists a convex support function $\varphi$ }  \\
	&  \text{so that $\varphi(v_i) = -a_i$ if $v_i \notin I_{\emptyset}, \varphi(v_i) \geq -a_i$ if $v_i \in I_{\emptyset} $} 
  \end{aligned}
\right\}.
\end{equation*}

Take the cone $\Gamma_{\Sigma, I_{\emptyset}}$ to be the image of $\tilde \Gamma_{\Sigma, I_{\emptyset}}$ under the image of $\pi$. The set of all such $\Gamma_{\Sigma, I_{\emptyset}}$ gives the fan $\Sigma_{\text{GKZ}}$. By Proposition 15.2.1 of \cite{CLS}, we have a tractable description of the GKZ cone as an intersection of cones:
\begin{equation}\label{eq:altGKZconeDef}
\Gamma_{\Sigma, I_{\emptyset}} = \bigcap_{\sigma \in \Sigma_{\text{max}}} \op{Cone}(\pi(e_\rho) \ | \ u_{\rho} \in I_{\emptyset} \text{ or } \rho \notin \sigma).
\end{equation}

Fix a fan $\Sigma \in N_{\R}$ with $\Sigma(1) \subseteq \nu$.  A priori, we get two different stacks.  One is associated to $\Gamma_{\Sigma, \nu \backslash \Sigma(1)}$.  This comes from the GIT problem belonging to the $S_\nu$-action on $\A^\nu$.  The other is associated to $\Gamma_{\Sigma, \emptyset}$, which comes from the GIT problem associated to $S_{\Sigma(1)}$-action on $\A^{\Sigma(1)}$.  Fortunately, the two stacks are isomorphic.

\begin{lemma}\label{lem: stack iso}
 Suppose we have an exact sequence of algebraic groups,
\[
0 \to H \xrightarrow{i} G \xrightarrow{\pi} Q \to 0.
\]
Let $G$ act on $X$ and hence on $X \times Q$ via $\pi$.  Then we have an isomorphism of stacks
\begin{equation}
[X \times Q / G ] \cong [X / H].
\label{eq: natural stacks}
\end{equation}
\end{lemma}

\begin{proof} $[X \times Q / G ]$ is the functor that assigns to a scheme $Y$ the groupoid  $Y \leftarrow \mathcal G \rightarrow X \times Q$ of $G$-torsors over $Y$ with $G$-equivariant maps to $X \times Q$.  Similarly  $[X / H]$ assigns to a scheme $Y$ the groupoid $Y \leftarrow \mathcal H\rightarrow X$ of $H$-torsors over $Y$ with $H$-equivariant maps to $X$.
Finally, one straightforwardly yet tediously checks that there is an equivalence of groupoids, 
\begin{align*}
\{ Y \leftarrow \mathcal G \rightarrow X \times Q \}& \Longleftrightarrow \{ Y \leftarrow \mathcal H\rightarrow X \} \\
[ Y \leftarrow \mathcal G \rightarrow X \times Q ]& \Rightarrow  [Y \leftarrow \mathcal G \times_{X \times Q} X \rightarrow X ] \\
[Y \leftarrow \mathcal H \overset{H}{\times} G \rightarrow X \times Q ]& \Leftarrow [Y \leftarrow \mathcal H \rightarrow X ].
\end{align*}
\end{proof}

\begin{corollary}
There is a natural isomorphism of stacks,
\[
[U_{\Sigma}  \times \gm^{I_{\emptyset}} / S_\nu] \cong [U_{\Sigma} /  S_{\nu \backslash I_\emptyset}].
\]
\label{prop: stack iso appl}
\end{corollary}
\begin{proof}

This is a special case of Lemma~\ref{lem: stack iso}.  Since $\nu, \nu \backslash I_\emptyset$ span $N_{\R}$, the maps $f_\nu, f_{\nu \backslash I_\emptyset}$ are injective.  Starting with the second and third row, the snake lemma gives the top isomorphism in the following diagram:
\[
\begin{CD}
0 @>>> 0 @>>> \Z^{I_\emptyset} @= \Z^{I_{\emptyset}}  @>>> 0\\
@. @VVV @VVV @VVV @.\\
0 @>>> M @>f_\nu>> \Z^{\nu} @>>> \widehat{S_\nu} @>>> 0\\
@. @| @VVV @VVV @.\\
0 @>>> M @>f_{\nu \backslash I_\emptyset}>> \Z^{\nu \setminus I_\emptyset} @>>> \widehat{S_{\nu \backslash I_\emptyset}}  @>>> 0.
\end{CD}
\]

Applying $\op{Hom}(-,\gm)$ to the right vertical exact sequence, gives
\[
0 \to S_{\nu \backslash I_\emptyset} \to S_\nu \to \gm^{I_\emptyset} \to 0.
\]
Hence, we can specialize Equation~\eqref{eq: natural stacks} to $H = S_{\nu \backslash I_\emptyset}, G = S_\nu,  Q = \gm^{I_\emptyset}, X = U_{\Sigma}$.

\end{proof}

Now, let $\sigma \subseteq N_{\R}$ be a $\Q$-Gorenstein cone.  Let $\nu \subseteq \sigma \cap N$ be a finite set which contains the ray generators of $\sigma$.  Define
\begin{equation}\begin{aligned}
\nu_{=1} &:= \{  v \in \nu \ | \ \langle \mathfrak{m}, v \rangle = 1\}; \\
\nu_{\neq1} &:= \{ a \in \nu \ | \ \langle \mathfrak{m}, a \rangle \neq 1\}.
\end{aligned}\end{equation}
Let $\chi_K$ be the character $-\sum_{v\in \nu} \chi_{D_{\nu}}$.

\begin{lemma} Take $\nu = \{v_1, \ldots, v_n\}$ to be a geometric collection of nonzero lattice points in $N$. Let $\Sigma$ be a simplicial fan in $N_{\R}$ and $I_{\emptyset} \subset \nu$ so that the support $|\Sigma|$ is the cone $\op{Cone}(\nu)$, the toric variety $X_{\Sigma}$ is semiprojective, and $\Sigma(1) \coprod I_{\emptyset} = \nu$. If the cone $\sigma$ is $\Q$-Gorenstein, then we have the following:
\begin{enumerate}
\item If $\langle \mathfrak{m}, u_\rho\rangle>1$ for all $\rho \in \nu_{\neq1}$ and $\nu_{\neq1} \subseteq I_{\emptyset}$, then $\chi_{K} \in \Gamma_{\Sigma, I_{\emptyset}}$.
\item If $\langle \mathfrak{m}, u_\rho\rangle<1$ for all $\rho \in \nu_{\neq1}$ and $\nu_{\neq1} \subseteq I_{\emptyset}$, then $-\chi_{K} \in \Gamma_{\Sigma, I_{\emptyset}}$.
\item If $\nu_{\neq1} = \emptyset$, then $\chi_{K}$ is zero in $(\hat{S}_\nu)_{\Q}$ hence in every chamber of the secondary fan.
\end{enumerate}
\label{lemma: Orlov setup}
\end{lemma}
\begin{proof}

Using the description of $\Gamma_{\Sigma, I_{\emptyset}}$ given in Equation~\eqref{eq:altGKZconeDef}, we know that
\[
\op{Cone}(\pi(e_\rho)\ | \ u_{\rho} \in \nu_{\neq1}) \subset \Gamma_{\Sigma, I_{\emptyset}}.
\]

Now, in $\widehat{S_\nu} \otimes_{\Z} \Q$ we have
\begin{align}
0 & = (\pi\circ f_\nu)( \mathfrak{m}) \notag \\
& = \sum_{v \in \nu} \langle  \mathfrak{m}, v \rangle \chi_{D_i} \notag \\
& = \sum_{v \in \nu_{=1}}  \chi_{D_v} + \sum_{v \in \nu_{\neq1}} \langle \mathfrak{m},v \rangle \chi_{D_a}
\label{eq: class relation}
\end{align}
as $\langle \mathfrak{m}, v \rangle = 1$ for all $v \in \nu_{=1}$.
This implies
\begin{align*}
\chi_{K} & = - \sum_{v\in \nu} \chi_{D_v} \\
& = \sum_{u_\rho \in \nu_{\neq1}} (\langle \mathfrak{m}, u_\rho\rangle-1) \chi_{D_\rho}.
\end{align*}

Hence if $\langle \mathfrak{m}, u_\rho\rangle>1$ for all $u_\rho \in \nu_{\neq1}$,
\[
\chi_{K} \in \op{Cone}(\pi(e_\rho) \ | \ u_\rho \in \nu_{\neq1}) \subseteq \Gamma_{\Sigma, I_{\emptyset}}.
\]
Similarly if $\langle \mathfrak{m}, u_\rho\rangle<1$ for all $u_\rho \in \nu_{\neq1}$,
\[
-\chi_{K} \in \op{Cone}(\pi(e_\rho) \ | \ u_\rho \in \nu_{\neq1}) \subseteq \Gamma_{\Sigma, I_{\emptyset}}.
\]
Finally, if $\nu_{\neq1} = \emptyset$, then $\chi_{K} =  \sum_{\rho \in \Sigma(1)} \chi_{D_\rho} = (\pi\circ f_\nu)( \mathfrak{m}) = 0$.
\end{proof}

\section{Derived Categories of Toric LG Models and Fractional CY Categories}\label{GorCones}

\subsection{Serre Functors of Factorization Categories}

As before, let $\sigma \subseteq N_{\R}$ be a $\Q$-Gorenstein cone, $\nu \subseteq \sigma \cap N$ be a finite, geometric collection of lattice points which contains the ray generators of $\sigma$. Partition the set $\nu$ into two subsets
\begin{equation}\begin{aligned}
\nu_{=1} &= \{  v \in \nu \ | \ \langle \mathfrak{m}, v \rangle = 1\} \text{ and} \\
\nu_{\neq1} &= \{ v \in \nu \ | \ \langle \mathfrak{m}, v \rangle \neq 1\}.
\end{aligned}\end{equation}
Since $\sigma$ is $\Q$-Gorenstein, the ray generators of $\sigma$ are contained in $\nu_{=1}$. Choose any subset
\[
R \subseteq \nu.
\]
The set $R$ gives an action of $\gm$ on $\A^\nu$ by
\begin{equation}
\lambda_R \cdot x_i := 
\begin{cases}
\lambda x_i & \tif v_i \in R \\
x_i & v_i \notin R \\
\end{cases}
\end{equation}
All together, this gives an action of $S_\nu \times \gm$ on $\A^\nu$.

Let $\Sigma \subseteq N_{\R}$ be a simplicial fan such that $\Sigma(1) \subseteq  \nu_{=1}$, $X_\Sigma$ is semiprojective, and $\op{Cone}(\Sigma(1)) = \sigma$. This gives an open subset
\[
U_{\Sigma} \times \gm^{\nu \backslash \Sigma(1)}  \subseteq \A^\nu
\]
and we restrict the action of  $S_\nu \times \gm$ to this open subset.  

Let 
\[
\chi: S_\nu \times \gm \to \gm
\]
be the projection character onto the $\gm$ factor.
Finally, let $w$ be a function on $\A^\nu$ which is semi-invariant with respect to $\chi$, i.e., $w \in \Gamma(U_{\Sigma}  \times \gm^{\nu \backslash \Sigma(1)} , \O(\chi))^{S_\nu \times \gm}$.

\begin{remark}
We have restricted our additional $\gm$-action and choice of $w$ with geometric applications in mind, namely, so we can apply Theorem~\ref{Hirano} to obtain Proposition~\ref{prop: ToricStackHirano}.  These restrictions force $w$ to be of the form
\[
w = \sum x_i f_i
\] 
where $v_i \in R$ and $f_i$ is a function of the variables not in $R$.  Therefore,
\[
\partial w \subseteq Z(w).
\]
We implicitly use this fact in applying Theorem~\ref{SerreFunctorDescription} below.  In fact, this condition holds for any quasi-homogeneous function of non-zero degree by Euler's homogeneous function theorem.    
\end{remark}

Recall that, for each $v_i \in \nu$, we also have characters $\chi_{D_i}$ defined by the composition
\[
S_\nu  \hookrightarrow \gm^{\nu} \xrightarrow{\pi_i} \gm
\]
where $\pi_i$ is the projection onto the factor corresponding to $v_i$.

Denote by $\chi_K$ the inverse of the character corresponding to the composition
\begin{align*}
S_\nu \times \gm & \stackrel{\hat\pi \times \lambda_R}{\hookrightarrow} \gm^{\nu} \xrightarrow{\times} \gm, \\
(g, \lambda) & \longmapsto \hat{\pi}(g) \cdot \lambda_R,
\end{align*}
where $\hat{\pi}$ is from Equation~\eqref{eq: hatpi} and
\[
(\lambda_R)_i := \begin{cases}
\lambda & \text{ if } v_i \in R \\
1 &  \text{ if } v_i \notin R.
\end{cases}
\]
In other words, if we identify the character group $\widehat{S_\nu \times \gm} = \widehat{S_\nu} \oplus \Z$
then
\begin{align*}
\chi_K & := - \sum_{v_i \in R} (\chi_{D_i}, 1) - \sum_{v_i \notin R} (\chi_{D_i}, 0) \\
& = - \sum_{v_i \in \nu}(\chi_{D_i}, 0) - (0, |R|).
\end{align*}

Notice that for any $v_j \in {\nu \backslash \Sigma(1)} $, $\chi_{D_j} \neq 0$.  Indeed, if $\chi_{D_j} =0$, then then there exists a lattice element $m \in M$ such that
\[
\langle m, v_i \rangle = \delta_{ij}.
\]
This, in turn, means that $v_j$ is a ray generator of $\op{Cone}(\nu)$ which is ruled out by the assumption that $v_j \in {\nu \backslash \Sigma(1)}$ as $\op{Cone}(\nu) = \op{Cone}(\Sigma(1))$.
Hence, as each $\chi_{D_j}$ is non-trivial we have a surjective map
\begin{align}\label{definemapforH}
 S_\nu \times \gm &  \overset{p}{\twoheadrightarrow} \gm^{\nu \backslash \Sigma(1)} \\
 (g, \lambda) & \mapsto \prod_{v_j \in {\nu \backslash \Sigma(1)} } \chi_{D_j}(g) \cdot \lambda_{R \cap (\nu \backslash \Sigma(1))}.
\end{align}

Let $H_{\Sigma, R}$ be the kernel of this map so that there is an exact sequence
\begin{equation}
\label{eq: exact sequence for H}
0 \to H_{\Sigma,R} \to  S_\nu \times \gm \to \gm^{\nu \backslash \Sigma(1)}  \to 0.
\end{equation}
We will write $H$ when $\Sigma$ and $R$ are understood in the immediate context.

\begin{lemma}
Viewing $\chi_{K}$ as a character of $H$, the equivariant canonical bundle $\omega_{[U_{\Sigma} /  H]}$ is isomorphic to $\O(\chi_{K})$.  
\end{lemma}
\begin{proof}
Recall, from Lemma~\ref{lem: stack iso} that there is a natural isomorphism of stacks,
\begin{equation}
[U_{\Sigma} \times \gm^{\nu \backslash \Sigma(1)}  / S_\nu \times \gm] \cong [U_{\Sigma} /  H].
\label{eq: another stack iso}
\end{equation}
Hence, the statement is equivalent to showing that $\O(\chi_K)$ is isomorphic to the equivariant canonical bundle of $[U_{\Sigma} \times \gm^{\nu \backslash \Sigma(1)}  / S_\nu \times \gm]$, when we view $\chi_K$ as a character of $S_\nu \times \gm$.  

The canonical bundle on $[U_{\Sigma} \times \gm^{\nu \backslash \Sigma(1)}  / S_\nu \times \gm]$ is the restriction of the canonical bundle on $[\A^{\nu} / S_\nu \times \gm]$, so we can reduce to checking the statement on affine space.  Now, we have the standard fact that the cotangent bundle on affine space is just the dual vector space (with its natural equivariant structure which in this case is just the dual grading on the dual vector space) i.e.
\[
\Omega_{[\A^{\nu} / S_\nu \times \gm]} = \bigoplus_{v_i \in R} \O( -D_i, 1) \oplus \bigoplus_{v_i \notin R} \O( -D_i, 0).
\]  
Therefore,
\begin{equation}\begin{aligned}
\omega_{[U_{\Sigma} /  H]} &= \Omega^{|\nu|}_{[U_\Sigma / S_\nu \times \gm]} \\
	&= \O( - \sum_{v_i \in R} (D_i , 1) - \sum_{v_i \notin R} (D_i , 0)) \\
	&= \O(\chi_K).
\end{aligned}\end{equation}
 \end{proof}
 
\begin{convention}
Identifying $\widehat{H}$ as a quotient of $\widehat{S_\nu} \oplus \Z$, we write elements of $\widehat{H}$ in the form $(a,b)$ with $a \in \widehat{S_\nu}, b \in \Z$ and view them as equivalence classes.

\end{convention}

\begin{lemma}
We have the following equality in $\widehat{H} \otimes_{\Z} {\Q}$,
\[
\chi_K =  (0, -\sum_{v_i \in \nu_{\neq1}}\langle \mathfrak{m}, v_i \rangle + |\nu_{\neq1}| - |R|).
\]
\label{lem: class relation}
If $\sigma$ is almost Gorenstein, then the equality holds in $\widehat{H}$.
\end{lemma}

\begin{proof}

As $\op{Hom}(-,\gm)$ is exact, we may apply it to Equation~\eqref{eq: exact sequence for H}, 
to obtain an exact sequence
$$
0 \longrightarrow \Z^{\nu \backslash \Sigma(1)} \stackrel{\hat{p}}{\longrightarrow} \widehat{S_\nu} \oplus \Z  \longrightarrow\widehat{H} \longrightarrow 0.
$$ 
Notice
\[
\hat{p}(e_i) = 
\begin{cases}
(\chi_{D_i}, 1) & \text{ if } v_i \in R \\
(\chi_{D_i}, 0) & \text{ if } v_i \notin R.
\end{cases}
\]
Hence in $\widehat{H}$,
\begin{equation}\begin{aligned}
(\chi_{D_i}, 1) = 0 & \text{ if } v_i \in R \\
(\chi_{D_i}, 0) = 0 & \text{ if } v_i \notin R.
\end{aligned}
\label{eq: A relation}
\end{equation}

Thus,
\begin{equation}\begin{aligned}
\chi_K & = - \sum_{v_i \in \nu}(\chi_{D_i}, 0) - (0, |R|) \\
&= (-\sum_{v_i \in \nu_{=1}} \chi_{D_i} - \sum_{v_i \in \nu_{\neq1}} \chi_{D_i}, |R| ) \\
&= (-\sum_{v_i \in \nu_{\neq1}}{(\langle \mathfrak{m}, v_i \rangle - 1)} \chi_{D_i}, |R| ) \\
&= (0,  -\sum_{v_i \in \nu_{\neq1}} \langle \mathfrak{m}, v_i \rangle + |\nu_{\neq1}| - |R|).
\end{aligned}\end{equation}
The first line is by definition.  The second line is because $\nu$ is a disjoint union of $\nu_{=1}$ and $\nu_{\neq1}$.  The third line follows from \eqref{eq: class relation} (notice that this holds over $\Z$ when $\mathfrak m \in M$ and over $\Q$ when $\mathfrak m \in M_{\Q}$).  The fourth line is \eqref{eq: A relation}.

\end{proof}

We can restrict $w$ by defining a function $\bar w$ on $\mathbb{A}^{\Sigma(1)}$  by setting all variables associated to points in the set $\nu \setminus \Sigma(1)$ to one. When we restrict $\bar w$ to $U_{\Sigma}$, we have $\bar w \in \Gamma(U_{\Sigma}, \O_{U_{\Sigma}}((0,1)))^H$.

\begin{remark}
Under the isomorphism of stacks~\eqref{eq: another stack iso}, $w$ corresponds to $\bar w$.
\end{remark}

We can now state the following special case of Theorem~\ref{SerreFunctorDescription}.
\begin{corollary}
Assume that $[\partial \bar w / S_{\Sigma(1)}]$ is proper.
The category 
$$
\op{D}^{\op{abs}}[U_{\Sigma},  H, \bar w] 
$$
is fractional Calabi-Yau of dimension 
$$
 - 2\sum_{v_i \in \nu_{\neq1}}\langle \mathfrak{m}, v_i \rangle +2|\nu_{\neq1}| - 2|R| + \dim N_{\R}. 
$$
If $\sigma$ is almost Gorenstein, then it is Calabi-Yau.
\label{thm: fractional CY}
\end{corollary}

\begin{proof}
Let $d$ be the smallest natural number such that $d \cdot \mathfrak{m} \in M$ (notice that $d=1$ when $\sigma$ is almost Gorenstein). We have,
\begin{align*}
S^d & = ( - \otimes \omega^d_{[U_{\Sigma} / H]})[d(\op{dim }U_{\Sigma} -\op{dim }H - 1)] \\
& = (- \otimes \O(d\chi_K)) [d(\op{dim }U_{\Sigma} -\op{dim }H - 1)] \\
& = (- \otimes \O( d(0, - \sum_{v_i \in \nu_{\neq1}} \langle \mathfrak{m}, v_i \rangle  +|\nu_{\neq1}| - |R|)) [d(\op{dim }U_{\Sigma} -\op{dim }H - 1)] \\
& = [d(-2\sum_{v_i \in \nu_{\neq1}}\langle \mathfrak{m}, v_i \rangle +2|\nu_{\neq1}| - 2|R| + \op{dim }U_{\Sigma} - \op{dim }H - 1)]\\ 
& = [d(-2\sum_{v_i \in \nu_{\neq1}}\langle \mathfrak{m}, v_i \rangle +2|\nu_{\neq1}| - 2|R| + \dim N_{\R})].
\end{align*}
The first line is Theorem~\ref{SerreFunctorDescription}.  The second line follows from the fact that $\omega_{[U_{\Sigma} / H]}$ is the restriction of the canonical bundle on $[\A^{\Sigma(1)} / H]$.  The third line is Lemma~\ref{lem: class relation}.  The fourth line follows from the isomorphism of functors 
\[
[2] = (-\otimes \O(0,1) )
\]
in $\op{D}^{\op{abs}}[U_{\Sigma},  H, \bar w]$ since $\bar w$ is a section of the equivariant line bundle $\O(0,1)$. The last line follows from the definition of $H$.

\end{proof}

We now provide a toric specialization of Theorem~\ref{thm: BFKVGIT}, found in the unabridged version \cite{BFK12v2} of the paper \cite{BFK12}. First let us provide some notation and the assumptions for the setup of this theorem. Let:
\begin{itemize}
\item $\Sigma_+$ and $\Sigma_-$ be two adjacent chambers in the secondary fan sharing a wall $\tau$, 
\item $\lambda: \gm \rightarrow S_\nu$ be the primitive one-parameter subgroup defining hyperplane containing the wall $\tau$ separating $\Sigma_+$ and $\Sigma_-$
\item $\chi_+ \in \Sigma_+$, $\chi_- \in \Sigma_-$, and $\chi_0 \in \tau$  be two characters in the relative interior of the adjacent chambers and wall,
\item $U_+$, $U_-$ and be the open semistable loci of $\A^\nu$ corresponding to their respective characters $\chi_+, \chi_-$, and $U_0$ be the intersection of the fixed locus of $\lambda$ with the open semistable locus of $\chi_0$,
\item $S_0$ be the quotient group $S_\nu / \lambda(\gm)$
\item $G$ be a group whose action commutes with $S_\nu$ (in our application, this will be R-charge)
\item $w\in \kappa[x_i \ | \ v_i \in \nu]$ be an $S_\nu$-invariant and $G$-semi-invariant function with respect to a character $\eta$ of $G$,
\item $w_+, w_-, w_0$ be induced sections of the line bundles determined by $\eta$ on $[U_+ / S_\nu \times G], [U_- / S_\nu \times G],$ and $[U_0 / S_0 \times G]$ respectively.
\end{itemize}

\begin{theorem}[Theorem 5.1.2 of \cite{BFK12v2}] \label{thm: BFK toric version}
Let $\mu =  - \sum_{v_i \in \nu} \langle \lambda, v_i\rangle,$ and  $d\in \Z$. The following statements hold:
\begin{enumerate}
 \item If $\mu > 0$, there exists fully-faithful functors,
\begin{displaymath}
 \Phi_d : \op{D}^{\op{abs}}[U_-, S_\nu \times G, w_- ] \to  \op{D}^{\op{abs}}[U_+, S_\nu \times G, w_+ ]
\end{displaymath}
 and
\begin{displaymath}
 \Upsilon_-: \op{D}^{\op{abs}}[U_{0}, S_0 \times G, w_0] \to \op{D}^{\op{abs}}[U_{+}, S_\nu \times G, w_+] ,
\end{displaymath}
 and a semi-orthogonal decomposition, with respect to $\Phi_d, \Upsilon_-$,
\[
\op{D}^{\op{abs}}[U_{-}, S_\nu \times G, w_- ]=  \left\langle \mathcal{D}_-(-\mu - d + 1), \ldots,\mathcal{D}_-(-d), \op{D}^{\op{abs}}[U_{-}, S_\nu \times G, w_- ]\right\rangle,
\]
where $\mathcal{D}_-(\ell)$ is the image of $\Upsilon_-$ twisted by a character of $\lambda$-weight $\ell$.
\item  If $\mu = 0$, there is an equivalence
\begin{displaymath}
 \Phi_d : \op{D}^{\op{abs}}[U_{-}, S_\nu \times G, w_- ] \to \op{D}^{\op{abs}}[U_{+}, S_\nu \times G, w_+].
\end{displaymath}

 \item If $\mu < 0$, there exists fully-faithful functors,
\begin{displaymath}
 \Phi_d : \op{D}^{\op{abs}}[U_{+}, S_\nu \times G, w_+] \to \op{D}^{\op{abs}}[U_{-}, S_\nu \times G, w_- ]
\end{displaymath}
 and
\begin{displaymath}
 \Upsilon_+: \op{D}^{\op{abs}}[U_{0}, S_0 \times G, w_0] \to  \op{D}^{\op{abs}}[U_{-}, S_\nu \times G, w_- ],
\end{displaymath}
 and a semi-orthogonal decomposition, with respect to $\Phi_d, \Upsilon_+$,
\[
 \op{D}^{\op{abs}}[U_{-}, S_\nu \times G, w_- ] = 
  \left\langle \mathcal{D}_+(-d), \ldots, \mathcal{D}_+(\mu-d+1), \op{D}^{\op{abs}}[U_{+}, S_\nu \times G, w_+] \right\rangle,
\]
 where $\mathcal{D}_+(\ell)$ is the image of $\Upsilon_+$ twisted by a character of $\lambda$-weight $\ell$. 

\end{enumerate}
\end{theorem}

We now use this theorem iteratively to provide a comparison theorem to a GIT chamber whose absolute derived category corresponds to a (fractional) Calabi-Yau category.

\begin{theorem}\label{FFLGmodels}
Let $\tilde{\Sigma}$ be any simplicial fan such that $\tilde{\Sigma}(1) = \nu$ and $X_{\tilde{\Sigma}}$ is semiprojective.  Similarly, let $\Sigma$ be any simplicial fan such that $\Sigma(1) \subseteq \nu_{=1}$, $X_{\Sigma}$ is semiprojective, and $\op{Cone}(\Sigma(1)) = \sigma$. We have the following:
\begin{enumerate}
\item If $\langle \mathfrak{m}, a \rangle > 1$ for all $a \in \nu_{\neq1}$, then there is a fully-faithful functor,
\[
\op{D}^{\op{abs}}[U_{\Sigma},   H, \bar w] \longrightarrow \op{D}^{\op{abs}}[U_{\tilde{\Sigma}},   S_\nu \times \gm, w].
\]
\item If $\langle \mathfrak{m}, a \rangle < 1$ for all $a \in \nu_{\neq1}$, then there is a fully-faithful functor,
\[
 \op{D}^{\op{abs}}[U_{\tilde{\Sigma}},   S_\nu \times \gm, w]  \longrightarrow \op{D}^{\op{abs}}[U_{\Sigma},   H, \bar w].
 \]
\item If $A = \emptyset$, then there is an equivalence,
\[
\op{D}^{\op{abs}}[U_{\Sigma},   H, \bar w] \cong \op{D}^{\op{abs}}[U_{\tilde{\Sigma}},   S_\nu \times \gm, w].
\]
\end{enumerate}
Furthermore, if $[\partial \bar w / S_{\Sigma(1)}]$ is proper, then $\op{D}^{\op{abs}}[U_{\Sigma},   H, \bar w]$ is fractional Calabi-Yau.  If, in addition, $\sigma$ is almost Gorenstein, then $\op{D}^{\op{abs}}[U_{\Sigma},   H, \bar w]$ is Calabi-Yau.
\end{theorem}

\begin{proof}

We prove (1) and later state the necessary adjustments for (2) and (3).
Proposition~\ref{prop: GKZ chamber bijection} says that $\Gamma_{\tilde{\Sigma}, \emptyset}, \Gamma_{\Sigma, \nu \backslash \Sigma(1)}$ are both chambers of the GKZ fan of $\nu$.
Suppose that $\langle \mathfrak{m}, a \rangle > 1$ for all $a \in \nu_{\neq1}$.  Then, by Lemma~\ref{lemma: Orlov setup}, $-\chi_K \in \Gamma_{\Sigma, \nu \backslash \Sigma(1)}$. 

Choose a straight line path, $\gamma_1: [0,1] \to (\widehat{S_\nu})_{\mathbb{R}}$, such that,
\begin{itemize}
 \item $\gamma_1(0)$ lies in the interior of $\Gamma_{\tilde{\Sigma}, \emptyset}$,
 \item $\gamma_1(1) = -\chi_K$,
 \item For any $\epsilon > 0$ so that $\gamma_1((0,1- \epsilon))$ does not intersect any cone of codimension $2$.
 \end{itemize}
 The existence of such a path is easily justified.  Namely, a generic choice will do, see, e.g., the proof of Theorem 5.2.3 of \cite{BFK12}.

Since there are finitely many chambers, the union of the chambers not containing $-\chi_K$ is closed.
Hence, we may choose $\epsilon$ sufficiently small so that $B_{\epsilon}(-\chi_K)$, the ball of radius $\epsilon$ centered at $-\chi_K$, only intersects chambers containing $-\chi_K$.  Then, choose a second straight line path, such that, 
\begin{itemize}
 \item $\gamma_2(0) = \gamma_1(1-\epsilon)$, 
 \item $\gamma_2(1)$ lies in $\text{rel int}(\Gamma_{\Sigma, \nu \backslash \Sigma(1)}) \cap B_{\epsilon}(-\chi_K)$,
 \item $\gamma_2([0,1])$ does not intersect any cone of codimension $2$.
 \end{itemize}
 Similarly, the existence of such a path follows from the convexity of 	$B_{\epsilon}(-\chi_K) \cap |\Sigma_{\op{GKZ}}|$ and, as before, the fact that you can generically avoid codimension $2$ cones.

The concatenation of $\gamma_1, \gamma_2$ defines a sequence of wall crossings in the GKZ fan of $\nu$ which begins in $\Gamma_{\tilde{\Sigma}, \emptyset}$ and ends in $\Gamma_{\Sigma, \nu \backslash \Sigma(1)}$.
The fans corresponding to $\Gamma_{\tilde{\Sigma}, \emptyset}, \Gamma_{\Sigma, \nu \backslash \Sigma(1)}$ are, by definition, $\tilde{\Sigma}, \Sigma$ respectively.

Notice that for each wall $\tau$ which intersects $\gamma_1([0,1-\epsilon])$, $-\chi_K$ either lies on $\tau$ or is in the direction of $\gamma$.  Furthermore, each $\tau$ which intersects $\gamma_2([0,1])$, must also intersect $B_{\epsilon}(-\chi_K)$ and, hence, $-\chi_K$ lies in $\tau$.

  Hence, by Theorem~\ref{thm: BFK toric version} each wall-crossing induces a fully-faithful functor or equivalence between categories of singularities corresponding to successive chambers.  
Concatenating gives the desired fully-faithful functor,
\[
\op{D}^{\op{abs}}[U_{\Sigma},   H, \bar w] \longrightarrow \op{D}^{\op{abs}}[U_{\tilde{\Sigma}},   S_\nu \times \gm, w].
\]
This finishes the proof of (1). To prove (2), replace $-\chi_K$ by $\chi_K$ which switches the direction of the fully-faithful functor.  To prove (3), only use the path $\gamma_2$. The last part of the statement of the theorem is just a repetition of Theorem~\ref{thm: fractional CY}.

\end{proof}

\begin{remark}
A choice of $\Sigma$ as in Theorem~\ref{FFLGmodels} always exists, for example, one can apply Proposition 15.1.6 of \cite{CLS} to the set $\nu_{=1}$.
\end{remark}

\begin{remark}
The fully-faithful functors appearing in Theorem~\ref{FFLGmodels} actually give rise to a semi-orthogonal decomposition.  This can be described explicitly in terms of the wall-crossings which occur in the path $\gamma$ by iteratively using the semi-orthogonal decompositions described in Theorem~\ref{thm: BFK toric version}. 
The description of the orthogonal can be rather complicated and cumbersome.  We describe the orthogonal in the example appearing in Subsection~\ref{SOD Geom FCY} to illustrate how the orthogonal is always computable in practice.
\end{remark}

\subsection{Application to Toric Complete Intersections}\label{subsec: toric CI}

We now unpack the geometric consequences of this theorem. Let $\Psi \subseteq N_{\R}$ be a complete fan such that $X_\Psi$ is projective and $D_1, ..., D_t$ be nef-divisors. Write these nef divisors as linear combinations 
$$
D_i = \sum_{\rho\in\Psi(1)} a_{i\rho} D_\rho
$$
 Assume that $|\Psi_{-D_1, ..., -D_t}| \subseteq (N \times \Z^t)_{\R}$ is $\Q$-Gorenstein. We refer the reader back to Proposition~\ref{hypersurfacesExample} for a class of examples of such cones. Let $e_i$ be the standard basis of the sublattice $\Z^t$  in $(N \times \Z^t)$, and set
 \[
 \mathfrak n := \sum_{i = 1}^t e_i.
 \]
In the case where the $D_i$ are all Cartier, this definition is aligned with the definition of $\mathfrak n$ in Section~\ref{ConesBestiary} by Lemma~\ref{dualConeIsCayley}, but here we do not assume that the cone $(\op{Cone}(\nu))^\vee$ is $\Q$-Gorenstein.
 
We restrict to the case where 
\[
\op{Cone}(\nu) = |\Psi_{-D_1, ..., -D_t}|
\]
and
\[
 \{ u_{\rho} \ | \ \rho \in \Psi_{-D_1, ..., -D_t}(1)\} \subseteq \nu.
\]
The condition $\op{Cone}(\nu) = |\Psi_{-D_1, ..., -D_t} |$ amounts to $\op{Cone}(\nu )^\vee$ being a Cayley cone of length $t$. Set $R$ to be the subset $\{ e_i \ | \ 1 \leq i \leq t\}$.  Suppose $\Sigma$ is any fan such that
\begin{itemize}
\item $X_\Sigma$ is semi-projective,
\item $|\Sigma| = |\Psi_{-D_1, ..., -D_t}|$, and
\item For any $\delta \in \Sigma(1)$, we have $u_{\delta} \in \nu$ and $\langle \mathfrak m, u_\delta \rangle = 1$. 
\end{itemize}

Consider the set of lattice points
$$
\Xi : = \{ m \in |\Psi_{-D_1, ..., -D_t}|^\vee \cap (M \times \Z^t) \ | \ \langle m, \mathfrak{n} \rangle = 1\}
$$
We also specialize to the case where the superpotential $w$ is of the form
$$
w = \sum_{m \in \Xi} c_m x^m,
$$
for some $c_m \in \kappa$.  Let $|\nu| = n$.  Enumerate the rays $\rho_1, \ldots, \rho_{n-t}$ corresponding to the ray generators in $\nu \backslash R$ and introduce the variable $x_k$ for the ray $\rho_k$ for $1 \leq k \leq n-t$.   Similarly introduce the variables $u_j$  for $1 \leq j \leq t$ corresponding to the ray generators $e_i$ in $R$. 

If $m \in \Xi$ then there exists a unique $j_0$ so that 
\[
\langle m, e_j \rangle = 
\begin{cases} 
1 & \text{ if } j = j_0 \\
0 & \text{ if } j \neq j_0.
\end{cases}
\]
 Hence, we can partition the set $\Xi$ into subsets  
\[
\Xi_j: = \{ m \in \Xi \ | \ \langle m, e_j \rangle = 1\}.
\]
Note that $\Xi = (\Delta_1*\ldots*\Delta_t)\cap (M\times \Z^t)$ and $\Xi_j = (\Delta_j,e_j^*)\cap(M\times\Z^t)$, using the polytopes defined in  Lemma~\ref{dualConeIsCayley}.

We can decompose $w$ as:
$$
w =  \sum_{m \in \Xi} c_m x^m = \sum_{i=1}^t u_if_i, \text{ where } f_i := \sum_{m\in \Xi_j} c_m \prod_{k = 1}^{n-t}  x_k^{\langle m, u_{\rho_k}\rangle}.
$$
If $m \in \Xi_j$, then the corresponding function $\prod_{k = 1}^{n-t} x_k^{\langle m, u_{\rho_k}\rangle}$ is a global section of the nef divisor $D_k$. Hence, the function $f_j$ is a global section of $D_j$. The common zero locus of all $f_j$ is a global quotient substack,
$$
\mathcal{Z} : = Z(f_1, \ldots, f_t) \subseteq \mathcal{X}_{\Psi},
$$
of $\mathcal{X}_{\Psi}$. When $f_1, ..., f_t$ define a complete intersection, we can relate $\dbcoh{\mathcal Z}$ to the factorization category associated to the fan $\Sigma$.

\begin{proposition}\label{prop: ToricStackHirano}
Assume that $f_1, \ldots, f_t$ defines a complete intersection.  Then, there is an equivalence of categories,
$$
\dbcoh{\mathcal{Z}} \cong \op{D}^{\op{abs}}[U_{\Psi},   S_\nu \times \gm, w].
$$ 
\end{proposition}
\begin{proof}
This is really a corollary of Theorem~\ref{Hirano} due to Isik, Shipman, and Hirano.  We describe the specifics of our setup below.

First, by Proposition~\ref{prop: stack iso appl}, we can reduce to the case where 
\[
\nu = \{ u_\rho \ | \ \rho \in \Psi_{-D_1, ..., -D_t}(1) \}.
\]
Then, by Lemma~\ref{DecompCoxStack}, the map
\[
\pi: U_{\Psi_{-D_1, \ldots, -D_t}} \rightarrow U_{\Psi} \times \A^t
\]
induces the isomorphism of stacks,
\[
\mathcal X_{\Psi_{-D_1, \ldots, -D_t}} \cong \op{tot} (\bigoplus_{i=1}^t \O_{\mathcal X_{\Psi}}(\chi_{-D_i})).
\]

Since we chose $R = \{ e_i \ | \ 1 \leq i \leq t\}$, the subgroup $\gm$ in $S_\nu \times \gm$ acts by scaling the coordinates $u_i$ of $\A^t$.  This is precisely fiberwise dilation of the vector bundle $\op{tot}(\oplus_{i=1}^r \O_{\mathcal{X}_{\Psi}}(\chi_{-D_i}))$.  Hence, we may apply Theorem~\ref{Hirano} to get the result.

\end{proof}

\begin{corollary}\label{FFLGCI}
Assume that $f_1, \ldots, f_t$ defines a complete intersection.  Let $\Sigma$ be any simplicial fan such that $\Sigma(1) \subseteq \nu_{=1}$, $X_{\Sigma}$ is semiprojective, and $\op{Cone}(\Sigma(1)) = \sigma$. 
We have the following:
\begin{enumerate}
\item If $
\langle u_\rho + \sum_{i=1}^t a_{i\rho}e_i, \mathfrak m \rangle \geq 1 
$
for all $i$,  then there is a fully-faithful functor,
\[
\op{D}^{\op{abs}}[U_{\Sigma},   H, \bar w] \longrightarrow \dbcoh{\mathcal{Z}}.
\]
\item If 
$
\langle u_\rho + \sum_{i=1}^t a_{i\rho}e_i, \mathfrak m \rangle \leq 1 
$
for all $i$,  then there is a fully-faithful functor,
\[
\dbcoh{\mathcal{Z}}  \longrightarrow \op{D}^{\op{abs}}[U_{\Sigma},   H, \bar w].
 \]
\item If 
$
\langle u_\rho + \sum_{i=1}^t a_{i\rho}e_i, \mathfrak m \rangle = 1 
$
for all $i$, then there is an equivalence,
\[
\op{D}^{\op{abs}}[U_{\Sigma},   H, \bar w] \cong \dbcoh{\mathcal{Z}}.
\]
\end{enumerate}
Furthermore, if $\mathcal Z$ is smooth, then $\op{D}^{\op{abs}}[U_{\Sigma},   H, \bar w]$ is fractional Calabi-Yau.  If, in addition, $\sigma$ is almost Gorenstein, then $\op{D}^{\op{abs}}[U_{\Sigma},   H, \bar w]$ is Calabi-Yau.
\end{corollary}
\begin{proof}
This is a direct corollary of combining Theorem~\ref{FFLGmodels} and Proposition~\ref{prop: ToricStackHirano}.  Note that since $X_\Psi$ is projective, Corollary~\ref{totSemiprojective} implies that $X_{\Psi_{-D_1, \ldots, -D_t}}$ is semiprojective. We then satisfy the hypotheses in Theorem~\ref{FFLGmodels}. 
\end{proof}

This corollary is quite general. In particular we can relate this to the Examples in Subsection~\ref{ToricVectorBundles}.
\begin{example}
Let $X_{\Sigma}$ be a projective toric variety in $N_{\R}$ and let $D=-qK_{X_{\Sigma}}$ for some positive rational number $q$. Suppose $D$ is nef and there exists a global section $f \in \Gamma(X_{\Sigma}, D)$. We consider the hypersurface $\mathcal{Z} = Z(f) \subset \mathcal{X}_{\Sigma}$. By Proposition~\ref{hypersurfacesExample}, $X_{\Sigma_{-D}}$ is $\Q$-Gorenstein. Let $\mathfrak{m} \in M_{\Q}$ be the element so that the cone $|\Sigma_{-D}|$ is generated by $\{n \in N_{\R} \ | \ \langle \mathfrak{m}, n\rangle = 1\}$. In the proof of~\ref{hypersurfacesExample}, we show that $\langle \mathfrak{m}, (0,1)\rangle = \frac{1}{q}$, hence we have:
\begin{enumerate}
\item If $q <1$, there is a fully faithful functor
$$\op{D}^{\op{abs}}[U_{\Sigma},   H, \bar w] \longrightarrow \dbcoh{\mathcal{Z}}.$$
\item If $q>1$, there is a fully faithful functor 
\[
\dbcoh{\mathcal{Z}}  \longrightarrow \op{D}^{\op{abs}}[U_{\Sigma},   H, \bar w].
 \]
 \item If $q=1$, there is an equivalence
 $$
 \op{D}^{\op{abs}}[U_{\Sigma},   H, \bar w] \cong \dbcoh{\mathcal{Z}}.
 $$
 \end{enumerate}
If $\mathcal{Z}$ is smooth, then $ \op{D}^{\op{abs}}[U_{\Sigma},   H, \bar w] $ is fractional Calabi-Yau. If $q = \frac{1}{r}$, then  $\op{D}^{\op{abs}}[U_{\Sigma},   H, \bar w] $ is Calabi-Yau.  In Subsection~\ref{subsec: Orlov}, we go through this example in detail in the case where $X_{\Sigma} = \P^n$.
\end{example}

We can specialize even further to the case where both categories are geometric. Suppose that there exists elements $e_i'\in N\times \Z^t$ for $1 \leq i \leq s$ such that
\[
\sum_{i=1}^s e_i' = \sum_{i =1}^t e_i = \mathfrak n
\]
and there exists a $\Z$-basis for $N\times \Z^t$ which contains the set $\{e_i'\}$. Assume also that under the projection $p: N \times \Z^t \rightarrow N \times \Z^t / \langle e_i'\rangle$ the lattice points $p(u_\rho)$ for all $\rho \in \Psi_{-D_1, ..., -D_t}(1)$ are primitive, so that the cones over each $p(u_\rho)$ can become the rays of a new fan $\Upsilon$ we construct below. We assume that $\langle \mathfrak m, e_i'\rangle = 1$ for all $i$, i.e., that $\{e_i'\} \subset \nu_{=1}$. This is automatically implied if either $\mathfrak m \in M \times \Z^t$ or $\langle \mathfrak m, e_i'\rangle \in \Z$ for all $i$.

In this case, set 
\[
\nu = \{ u_\rho  \ | \ \rho \in \Psi_{-D_1, ..., -D_t}(1) \} \cup \{e_1', ..., e_s' \}.
\]

The $e_i'$ define a new collection of polytopes
\[
\Delta'_i := \{ a \in M \times \Z^t \ | \  \langle a, e_j' \rangle = \delta_{ij} \}
\]
such that
\[
\Delta'_1 * ... * \Delta'_s = |\Psi_{-D_1, ..., -D_t}|^\vee.
\]

Let
 \[
L := (N \times \Z^t) / \Z^r
\]
and $p$ be the projection with dual projection $p^*: (N\times\Z^t)^* \rightarrow L^*$.
Consider the Minkowski sum.
\[
\Delta' := \sum_{i=1}^s p^*(
\Delta_i') \subseteq L_{\R}
\]

Then, we can let $\Upsilon \subseteq L_{\R}$ be a simplicial refinement of the normal fan to $\Delta'$.  Each Minkowski summand $\Delta_i'$ defines a nef divisor $E_i$ on $X_\Upsilon$.  Furthermore, 
\begin{align*}
 |\Psi_{-D_1, ..., -D_t}|^\vee & = \R_{\geq0}(\Delta'_1 * ... * \Delta'_s)  \\
 & = |\Upsilon_{-E_1, ..., -E_s}|^\vee
\end{align*}
where the second line is Proposition~\ref{dualConeIsCayley}.
Hence
\[
|\Psi_{-D_1, ..., -D_t}| = |\Upsilon_{-E_1, ..., -E_s}|.
\]
 
Now, we can add two additional $\gm$ actions to the $S_\nu$-action on $\A^{\nu}$.  The first action $(\gm)_1$ is determined by $R_1 = \{e_1, ..., e_t\}$ and the second action $(\gm)_2$ is determined by $R_2 = \{ e_1', ..., e_s' \}$.

\begin{lemma}
There is an isomorphism of stacks,
\[ 
[\A^\nu / S_\nu \times (\gm)_1] \cong [\A^\nu / S_\nu \times (\gm)_2]. \]
\label{lem: Rcharge fixer upper}
\end{lemma}

\begin{proof}
Consider the following one-parameter subgroup
\[
\beta: \gm \hookrightarrow \gm^{\nu}
\]
 acting on $\A^{\nu}$ so that for $v \in \nu$, $s \in \gm$ acts by
\begin{equation}
\beta(s) \cdot x_v := 
\begin{cases}
\frac{1}{s}x_v & \tif v \in R_1 \backslash R_2 \\
s x_v & \tif v = \tif v \in R_2 \backslash R_1 \\
x_v & \text{otherwise}.\\
\end{cases}
\label{eq: fix Rcharge2}
\end{equation}
We claim that $\beta(\gm) \subseteq S_\nu$.
Indeed, by definition, $S_\nu$ lies in an exact sequence
\[
0 \longrightarrow S_\nu \overset{\widehat{\pi}}{\longrightarrow} \gm^{\nu} \overset{\widehat{f_{\nu}}}{\longrightarrow} \gm^{\op{dim }N + t}, 
\]
hence, to show that $\beta(\gm) \subseteq S_\nu$, we can simply check that $\widehat{f_\nu} \circ \beta  = 0$.  Since the functor $\widehat{(-)} := \op{Hom}( - ,  \gm)$ is exact, this is equivalent to $\widehat{\beta} \circ f_\nu =0$.  The latter is a morphism between free $\Z$-modules, hence vanishes if and only if the dual morphism vanishes. The vanishing of the dual morphism goes as follows:
\begin{align*}
f_\nu^\vee (\widehat{\beta}^\vee (n)) & = f_\nu^\vee ( n (\sum_{\rho \in R_2 \backslash R_1} e_{\rho}  - \sum_{\rho \in R_1 \backslash R_2} e_{\rho})  ) \\
&  = n (\sum_{i=1}^s e_i' - \sum_{i=1}^t e_i) = 0.
\end{align*}
Notice that $\beta$ splits as
\[
\chi_{D_i} \circ \beta = \op{Id}
\]
for $i \in R_2 \backslash R_1$ (without loss of generality, we may assume $R_2 \backslash R_1 \neq \emptyset$, as otherwise $R_1 = R_2$ and  the statement of the lemma is empty).

Now add the $(\gm)_1$-action.  This is a $\gm$ action on $\A^{|\nu|}$ which is given explicitly as
\begin{equation}
s \cdot x_v := 
\begin{cases}
s x_v & \tif v \in \{e_1, ..., e_t\}\\
x_v & \tif v \notin \{e_1, ..., e_t\}.\\
\end{cases}
\label{eq: Rcharge2}
\end{equation}
Let $\overline{S_\nu}$ be the subgroup induced by the splitting so that $S_\nu = \overline{S_\nu} \times \beta(\gm) \subset \gm^\nu$.   There is an automorphism \begin{align*}
F : (\overline{S_\nu} \times \gm) \times \gm & \to (\overline{S_\nu} \times \gm) \times \gm \\
(\bar{s}, s_1, s_2) & \mapsto (\bar{s}, s_2 s_1, s_2). 
\end{align*}

Now, the global quotient stack $[\A^\nu /S_{\nu} \times \gm]$ can be considered with the action of $S_{\nu} \times \gm$ given by precomposition with $F$.  Under $F$, the action of $S_\nu \times 1$ on $\A^{|\nu|}$ is the same as $F(S_\nu \times 1 ) = S_\nu \times 1$.  However, the projection action of $1 \times \gm$, becomes the action of the element $F(1,1,s) = (1,s,s)$.

To determine the action of the element $(1,s,s)$, notice that the action of $(1,s,1)$ is given by Equation~\eqref{eq: fix Rcharge2} and the action of $(1,1,s)$ is given by Equation~\eqref{eq: Rcharge2}.  Combining these two equations we get:
\begin{equation}
(1,s,s) \cdot x_v := 
\begin{cases}
s x_v & \tif v \in \{e_1', ..., e_s' \} \\
 x_v & \tif v \neq \{e_1', ..., e_s' \} \\
\end{cases}
\end{equation}

This is the $\gm$ action determined by $R_2$.  Hence we have,
\begin{align}
[\A^\nu /S_{\nu} \times (\gm)_1 ] & \cong [\A^\nu / F^{-1}(S_{\nu} \times \gm)] \notag\\
 & = [\A^\nu/ S_\nu \times (\gm)_2]
\label{eq: realization'2}
\end{align}
as desired.

\end{proof}

Our new decomposition of $\mathfrak{n}$ gives a new decomposition of $\Xi$.  Namely, if $m \in \Xi$ then there exists a unique $j_0$ so that 
\[
\langle m, e'_j \rangle = 
\begin{cases} 
1 & \text{ if } j = j_0 \\
0 & \text{ if } j \neq j_0.
\end{cases}
\]
This gives a new partition of $\Xi$ into subsets 
\[
\Xi'_j: = \{ m \in \Xi \ | \ \langle m, e'_j \rangle = 1\}.
\]
If we again enumerate the rays as $\rho_1, \ldots, \rho_{n-s}$ corresponding to the ray generators in $\nu \backslash R_2$ and introduce the variable $x_k$ for the ray $\rho_k$ for $1 \leq k \leq n-s$.   Similarly introduce the variables $u_j$  for $1 \leq j \leq s$ corresponding to the ray generators $e_i'$ in $R_2$.   We get a decomposition of $w$ as:
$$
w =  \sum_{m \in \Xi} c_m x^m = \sum_{i=1}^t u'_j g_j, \text{ where } g_i := \sum_{m \in \Xi'_j} c_m \prod_{k \in 1}^{n-t}  x_k^{\langle m, u_{\rho_k}\rangle}.
$$

As above, the functions $g_i$ can be interpreted as global sections of $\O(E_i)$ on $X_\Upsilon$ and a closed substack
$$
\mathcal{Z}' : = Z(g_1, \ldots, g_s) \subseteq \mathcal{X}_{\Upsilon},
$$
of $\mathcal{X}_{\Upsilon}$.

\begin{corollary}\label{thm:comparegeometric}
Assume that $f_1, \ldots, f_t$ and $g_1, ..., g_s$ define complete intersections.  Assume further that $s = \langle \mathfrak m, \mathfrak n \rangle$.
We have the following:
\begin{enumerate}
\item If 
$
\langle u_\rho + \sum a_{i\rho}e_\rho, \mathfrak m \rangle \geq 1 
$
for all $i$,  then there is a fully-faithful functor,
\[
\dbcoh{\mathcal Z'} \longrightarrow \dbcoh{\mathcal Z}.
\]
\item If 
$
\langle u_\rho + \sum a_{i\rho}e_\rho, \mathfrak m \rangle \leq 1 
$
for all $i$,  then there is a fully-faithful functor,
\[
\dbcoh{\mathcal Z}  \longrightarrow \dbcoh{\mathcal Z'}.
 \]
\item If 
$
\langle u_\rho + \sum a_{i\rho}e_\rho, \mathfrak m \rangle = 1 
$
for all $i$, then there is an equivalence,
\[
\dbcoh{\mathcal Z'} \cong \dbcoh{\mathcal Z}.
\]
\end{enumerate}
Furthermore, if $\mathcal Z'$ is smooth, then it has torsion canonical bundle.  If, in addition, $\sigma$ is almost Gorenstein, then $\mathcal Z'$ is Calabi-Yau.
\end{corollary}
\begin{proof}
We apply Theorem~\ref{FFLGmodels} to the case 
\[
\nu = \{ u_\rho  \ | \ \rho \in \Psi_{-D_1, ..., -D_t}(1) \} \cup \{e_1', ..., e_s' \}
\]
and $R = R_1$.  Proposition~\ref{prop: ToricStackHirano} gives us the equivalence 
\[
\dabsfact{U_{\Psi_{-D_1, ..., -D_t}}, S_{\Psi(1)} \times (\gm)_1, w} \cong\dbcoh{\mathcal Z}.
\]
 Similarly,  Lemma~\ref{lem: Rcharge fixer upper} and Proposition~\ref{prop: ToricStackHirano} gives us the equivalences
 \begin{align*}
 \dabsfact{U_{\Upsilon_{-E_1, ..., -E_t}}, S_{\Upsilon(1)} \times (\gm)_1, w} &\cong  \dabsfact{U_{\Upsilon_{-E_1, ..., -E_t}}, S_{\Upsilon(1)} \times (\gm)_2, w}\\
  &\cong  \dbcoh{\mathcal Z'}.
 \end{align*}\end{proof}

\begin{remark}
Case (3) of Corollary~\ref{thm:comparegeometric} proves the Batyrev-Nill conjecture, as in \cite{FK14}.
\end{remark}

\begin{remark}
Case (1) of Corollary~\ref{thm:comparegeometric} when $t=1$ relates a Calabi-Yau complete intersection to a hypersurface in a projective bundle. It proves the fully faithfulness of the semi-orthogonal decomposition in Proposition 2.10 of \cite{Orl06} for the case of Calabi-Yau complete intersections in toric varieties.
\end{remark}

A new case where Corollary~\ref{thm:comparegeometric} applies is the following. 

\begin{example}
\label{UnderConstruction} 
Let $N = \Z^6$ and $M$ its dual lattice, where $\{e_i\}$ is the standard elementary basis for $N$. Consider the point collection $\nu = \{v_1, \ldots, v_7, a_2, a_2\}$ in $N$ where
\begin{equation}\begin{aligned}
v_1 &= (1,0,0,0,0,0), & v_4 &=  (0,0,0,0,0,1),& v_7 &= (0,0,0,0,1,1), \\
v_2 &= (0,1,2,0,0,0),  & v_5 &= (1,1,0,0,0,0), & a_1 &= (1,1,1,0,0,0), \\
v_3 &= (0,0,0,2,1,0), & v_6 &= (0,0,1,1,0,0), & a_2 &= (0,0,0,1,1,1).
\end{aligned}\end{equation}
Here, $\mathfrak{m}=(1,0,\frac{1}{2}, \frac{1}{2}, 0,1)$, $\nu_{=1} = \{v_i\}$ and $\nu_{\neq1} = \{a_i\}$. Note that $\langle \mathfrak{m}, a_i\rangle >1$ for both $a_i$. In this example, there are multiple vector bundle structures. Here note that
$$
a_1 + a_2 = v_5 + v_6 + v_7,
$$
which correspond to the sets $\{e_i\}$ and $\{e_i'\}$, respectively, in the notation above and  $\mathfrak{n} = (1,1,1,1,1,1)$. Note that $\langle \mathfrak{m}, \mathfrak{n}\rangle = 3.$ 

The first vector bundle structure can be seen via looking at the projection $\pi: N \rightarrow N/ \langle a_1, a_2\rangle$. Here, we can see a fan in $N/ \langle a_1, a_2\rangle\cong \Z^4$ (we use the isomorphism given by changing to the basis $e_1, e_2, e_4, e_5, e_1+e_2+e_3, e_4+e_5+e_6$ and then projecting to the first four dimensions). Set $\bar v_i$ to be $\pi(v_i)$. There is a fan $\Sigma$ where $X_{\Sigma}$ is semiprojective and $\Sigma(1)$ is  generated by
\begin{equation}\begin{aligned}
\bar v_1 &= (1,0,0,0), & \bar v_4 &=  (0,0,-1,-1),& \bar v_7 &= (0,0,-1,0), \\
\bar v_2 &= (-2,-1,0,0),  & \bar v_5 &= (1,1,0,0), &  \\
\bar v_3 &= (0,0,2,1), & \bar v_6 &= (-1,-1,1,0). & 
\end{aligned}\end{equation}
Let $D_{\bar \rho_i}$ correspond to the divisor associated to the ray $\bar \rho_i = \op{Cone}(\bar v_i)$. Here, we identify two divisors $D_1 = 2D_{\bar \rho_2} + D_{\bar\rho_6}$ and $D_2 = D_{\bar \rho_4} + D_{\bar \rho_7}$ so that $\Sigma_{-D_1, -D_2}$ is a semiprojective fan. Moreover, we can see that $\Sigma_{-D_1, -D_2}(1) = \nu$.

The second vector bundle structure can be seen by looking at the projection $\pi: N \rightarrow N / \langle v_5, v_6, v_7\rangle$. Here, we have a fan in $N/ \langle v_5, v_6, v_7\rangle \cong \Z^3$ (using the isomorphism given by changing to the basis $e_1, e_3, e_5, e_1+e_2, e_3+e_4, e_5+e_6$ and then projecting to the first three dimensions). Set $\bar v_i '$ to be $\pi'(v_i)$. There is a fan $\Upsilon$ where $X_{\Upsilon}$ is semiprojective and $\Upsilon(1)$ is generated by the cones over each of these lattice points:
\begin{equation}\begin{aligned}
\bar v_1' &= (1,0,0), & \bar v_4' &=  (0,0,-1).&  \\
\bar v_2' &= (-1,2,0),  & &  \\
\bar v_3' &= (0,-2,1). & & 
\end{aligned}\end{equation}
Let $D_{\bar \rho_i'}$ be the divisor associated to $\op{Cone}(v_i') \in \Sigma'(1)$. We define three divisors: 
\begin{equation}
E_1 = D_{\bar \rho_2'}, \quad E_2 = 2D_{\bar \rho_3'}, \quad E_3 = D_{\bar \rho_4'}.
\end{equation}
Here, $\Upsilon_{-E_1, -E_2, -E_3}(1) \subset \nu_{=1}$.

Define a global function on the affine space $\A^\nu$ by taking the finite set 
$$
\Xi = \{ m \in M \ | \ m \in \op{Cone}(\nu), \langle m, v_i \rangle = 1 \text{ for $i =5, 6,7$}, \langle m, a_i\rangle = 1 \text{ for $i=1,2$}\}.
$$
Take a generic potential
$$
W = \sum_{m \in \Xi} c_m x_i^{\langle m, v_i\rangle}.
$$
which expands as 
$$
W = c_1x_1x_5x_8 + c_2 x_2x_5x_8 + c_3x_1^2x_6x_8 + c_4x_2^2x_6x_8 + c_5 x_3^2 x_6x_9 + c_6x_4^2x_6x_9 + c_7 x_3x_7x_9 + c_8 x_4x_7x_9.
$$
The global sections associated to each of the divisors above are
\begin{equation*}\begin{aligned}
f_1 &=  c_1x_1x_5 + c_2 x_2x_5 + c_3x_1^2x_6 + c_4x_2^2x_6, \quad f_2= c_5 x_3^2 x_6 + c_6x_4^2x_6 + c_7 x_3x_7 + c_8 x_4x_7, \\
g_1 &= c_1x_1 + c_2 x_2, \quad g_2= c_3x_1^2 + c_4x_2^2 + c_5 x_3^2  + c_6x_4^2 \quad g_3 = c_7 x_3 + c_8 x_4.
\end{aligned}\end{equation*}

Now, $\mathcal{Z'} = Z(g_1,g_2,g_3)\subset \mathcal{X}_{\Upsilon}$ and is a 0-dimensional stack with 2-torsion canonical bundle. On the other hand, $\mathcal{Z} = Z(f_1, f_2)\subset \mathcal{X}_{\Sigma}$ is a 2-dimensional stack. Since $\langle \mathfrak{m}, a_i\rangle >1$ for both $a_i$, then by Corollary~\ref{thm:comparegeometric},  we have a fully-faithful functor 
$$
\dbcoh{\mathcal Z'} \longrightarrow \dbcoh{\mathcal Z}.
$$
\end{example}

\begin{remark}
In Example~\ref{UnderConstruction}, both of the decompositions of $\mathfrak{n}$ into sums of elements in $\op{Cone}(\nu) \cap N$ are maximal in that there does not exist a set of elements $I \subset \op{Cone}(\nu) \cap N$ so that $\sum_{n \in I} n = v_i$ or $\sum_{n \in I} n = a_i$. This is implied by the fact that $\nu$ is a Hilbert basis for the semigroup $\op{Cone}(\nu) \cap N$. This differentiates our results from those in \cite{Kuz15} as, in this example, the two vector bundle structures are, at least, not related by a toric projective bundle construction.  This was the most basic example we found. There are higher dimensional examples as well.
\end{remark}

\section{Examples}\label{sec: examples}

\subsection{Smooth Degree $d$ Hypersurfaces in Projective Space}\label{subsec: Orlov}

Let $N = \Z^{n+1}$, with elementary basis vectors $e_i$. Let $M$ be the dual lattice to $N$. Take the geometric point collection $\nu = \{v_1, \ldots, v_n, v_{n+1}, a\}$ where 
\begin{equation}\begin{aligned}
v_i &= e_i \text{ for $1\leq i \leq n$}, \\
v_{n+1} &= -e_1 -\ldots -e_n + de_{n+1} \\
a &= e_{n+1}
\end{aligned}\end{equation}
The cone $\sigma := \op{Cone}(\nu)$ is $\Q$-Gorenstein and $\mathfrak{m} := (1,\ldots, 1, \frac{n+1}{d})$. Note that $\mathfrak{m}\in M$ if and only if $d$ divides $n+1$. We have that $\langle \mathfrak{m}, a \rangle = \frac{n+1}{d}$, so
\begin{enumerate}
\item $\langle \mathfrak{m}, a \rangle > 1 $ if $d < n+1$,
\item  $\langle \mathfrak{m}, a \rangle < 1 $ if $d > n+1$, and 
\item $\langle \mathfrak{m}, a \rangle = 1 $ if $d = n+1$.
\end{enumerate}

Now, one easily computes that $S_\nu = \gm$ acting on $X:= \A^{n+2}$ with weights $1, ..., 1, -d$.
We denote the coordinates of $\A^{n+2}$ by  $x_1, \ldots, x_{n+1}$ for the lattice points $v_1, \ldots, v_{n+1}$ and the final coordinate by $u$ for the lattice point $a$.

The secondary fan/GIT fan for  this action of $S_\nu$ is one-dimensional and pictured in Figure~\ref{fig: Orlov GIT fan}. The irrelevant ideal and corresponding GIT quotients are also included in the figure.

\renewcommand{\thefigure}{\thesection.\arabic{figure}}
\begin{center}
\begin{tikzpicture}
\draw[latex-] (-4,0) -- (4,0) ;
\draw[-latex] (-4,0) -- (4,0) ;
\foreach \x in  {0}
\node[fill=black,draw=black,circle,inner sep=2pt, label=below:{0}] at (0,0) {};
\node[label=above:{$\langle x_1, \ldots, x_{n+1}\rangle$}] at (2.12,0) {};
\node[label=below:{$\op{tot}\O_{\P^n}(-d)$}] at (2.12,0) {};
\node[label=above:{$\langle u\rangle$}] at (-2.12,0) {};
\node[label=below:{$\left[\kappa^{n+1} / \Z_d\right]$}] at (-2.12,0) {};
\end{tikzpicture}
\captionof{figure}{GIT Fan for $\gm$ action}\label{fig: Orlov GIT fan}
\end{center}

Consider the set of lattice points
$$
\Xi := \{ m \in M\cap \sigma^\vee  \ | \ \langle m, \mathfrak{n}\rangle = 1\}.
$$
Note that $\op{Conv}(\Xi)$ is a regular simplex with side lengths $d$. Also, we then have a superpotential 
$$
w = \sum_{m \in \Xi} c_m ux^m
$$
for some $c_m \in \kappa$. The sum $f = \sum_{m \in \Xi} c_m x^m$ is a homogeneous degree $d$ polynomial in the variables $x_i$. Choose the coefficients $c_m$ so that $Z(f)$ is a smooth hypersurface in $\op{Proj}(\kappa[x_1, \ldots, x_{n+1}])$. 

We have two fans $\Sigma$ and $\tilde \Sigma$ that correspond to the two chambers of the secondary fan. The fan $\Sigma$ corresponding to the negative direction is the collection of cones consisting of $\sigma$ and its proper faces. Note that $\Sigma$ is simplicial, $X_{\Sigma} = \A^{n+1} / \Z_d$ is semiprojective, $\Sigma(1) \subseteq \nu_{=1}$, and $\op{Cone}(\Sigma(1)) = \sigma$. The corresponding potential on $U_{\Sigma}$ is $\bar w = f$. 

The fan $\tilde \Sigma$ corresponding to the positive direction is the star subdivision of $\Sigma$ along $e_{n+1}$. Hence, $\tilde \Sigma$ is a simplicial fan with $\tilde \Sigma(1) = \nu$ and $X_{\tilde \Sigma} = \op{tot}\O_{\P^n}(-d)$ is semiprojective. We therefore are in the context of Proposition~\ref{prop: ToricStackHirano} and can study the derived category of the hypersurface $\mathcal{Z} := Z(f) \subset \P^n$.  Another way to say this is that $\tilde{\Sigma} = \Psi_{-dD_1}$ where $\Psi$ is the fan for $\P^n$ and $D_1$ is the coordinate hyperplane defined by $x_1$.

By Corollary~\ref{FFLGCI}, 
\begin{enumerate}
\item if $d < n+1$, then we have a fully faithful functor,
$$
\op{D}^{\op{abs}}[U_{\Sigma},   \gm, f] \longrightarrow \dbcoh{\mathcal{Z}},
$$
\item if $d > n+1$, then we have a fully faithful functor,
$$
\dbcoh{\mathcal{Z}} \longrightarrow \op{D}^{\op{abs}}[U_{\Sigma},   \gm, f],
$$
\item if $d = n+1$, then we have an equivalence
$$
\op{D}^{\op{abs}}[U_{\Sigma},   \gm, f] \cong \dbcoh{\mathcal{Z}}.
$$
\end{enumerate} 
Moreover, since $f$ cuts out a smooth hypersurface, by Theorem~\ref{thm: fractional CY}, one has that the category $\op{D}^{\op{abs}}[U_{\Sigma},   \gm, f]$ is fractional Calabi-Yau of dimension 
$$
\frac{(n+1)(d-2)}{d}.
$$
If $d$ divides $n+1$, then $\mathfrak{m} \in M$ and the category $\op{D}^{\op{abs}}[U_{\Sigma},   \gm, f]$ is Calabi-Yau of the given dimension. 

The path $\gamma$ crosses a single wall determined by the ``identity'' one parameter subgroup.  The fixed locus for the action of $\gm$ is just the origin of $\A^\nu$.  Using the description of the right orthogonal in Theorem 5.2.1 of \cite{BFK12}, we get Theorem 3.11 of \cite{Orl09} (with the possible addition of a finite group action).  Without the finite group action, details of the explicit comparison were already provided in Section 7 of \cite{BFK12}.  The statement can also be derived from a minor generalization of Orlov's proof.

\subsection{A Semi-Orthogonal Decomposition with a Geometric FCY Category}\label{SOD Geom FCY}

We start by defining a set $\nu \subset N:= \Z^6$, consisting of eight lattice points:
\begin{equation*}\begin{aligned}
v_1 &= (1,0,0,0,0,0), & v_2 &= (0,1,0,0,0,0), & v_3 &= (0,0,1,0,0,0), & v_4 &= (0,1,0,2,1,2), \\
v_5 &= (-1,-2,-1,-2,0,0), & v_6 &= (0,0,0,0,1,0), & v_7 &= (0,0,0,0,-1,1) & v_8 &= (0,0,0,0,0,1)
\end{aligned}\end{equation*}
Here, we can see that the cone $\op{Cone}(\nu)$ is $\Q$-Gorenstein. Here, $\mathfrak{m}=(1,1,1,-\frac52,1,2)$ and $\nu_{=1} = \{v_1,\ldots, v_7\}$.  In this example, we have that $\mathfrak{n} = (0,0,0,0,0,1)$, so $\langle \mathfrak{m}, \mathfrak{n}\rangle = 2$. We define a superpotential on $\A^\nu = \A^8$ with variables $x_1, \ldots, x_8$:
\begin{equation*}\begin{aligned}
w &= x_8x_6x_1^2x_2 +  x_8x_6x_2^3+  x_8x_6x_3^2x_4 +  x_8x_6x_4^3 +  x_8x_6x_4x_5^2 \\
&\qquad + x_8x_7x_1^2 + x_8x_7x_2^2 + x_8x_7x_3^2 + x_8x_7x_4^2 + x_8x_7x_5^2.
\end{aligned}\end{equation*}

There are two vector bundle structures so that their rays are generated by the elements in $\nu$. First, consider the projection $\pi : N \rightarrow N/ \langle e_6\rangle$ and a complete fan $\Upsilon$ with rays 
\begin{equation*}\begin{aligned}
{\rho}_1 &= (1,0,0,0,0), & {\rho}_2 &= (0,1,0,0,0), & {\rho}_3 &= (0,0,1,0,0), & {\rho}_4 &= (0,1,0,2,1), \\
{\rho}_5 &= (-1,-2,-1,-2,0), & {\rho}_6 &= (0,0,0,0,1), & {\rho}_7 &= (0,0,0,0,-1),
\end{aligned}\end{equation*}
so that $X_{\Upsilon}$ is semiprojective. Consider the line bundle associated to the toric divisor $D = 2D_{\rho_4} + D_{\rho_7}$.  Here, $\Upsilon_{-D}(1) = \{\op{Cone}(v_i) \ | \ v_i \in \nu\}$. Here, $x_8$ is the bundle coordinate. Here note that here
$$
\dabs{[U_{\Upsilon_{-D}}, S_\nu \times \gm, w]} \cong \dbcoh{\mathcal{Z}},
$$
where $\mathcal{Z}$ is the zero set of the global section
$$f_1 = x_6x_1^2x_2 +  x_6x_2^3+  x_6x_3^2x_4 +  x_6x_4^3 +  x_6x_4x_5^2 + x_7x_1^2 + x_7x_2^2 + x_7x_3^2 + x_7x_4^2 + x_7x_5^2$$
of the divisor $D$.
By a routine check, we can see that the zero locus $\mathcal{Z} := Z(f_1) \subset X_{\Upsilon}$ is a smooth stack. 

Alternatively, consider the projection $\pi': N \rightarrow N/ \langle v_6, v_7\rangle \cong\Z^4$. We can define a complete fan $\Upsilon$ with rays
\begin{equation*}
{\rho}_1 = (1,0,0,0), {\rho}_2 = (0,1,0,0), {\rho}_3 = (0,0,1,0), {\rho}_4 = (0,1,0,2),  {\rho}_5 = (-1,-2,-1,-2),
\end{equation*}
so that  $X_{\Upsilon}$ is semiprojective. Namely, $\mathcal{X}_\Upsilon$ is the quotient stack $[\P^4 / \Z_2]$ where the $\Z_2$ acts by
$$
g\cdot (y_0:y_1:y_2:y_3:y_4) = (y_0:-y_1:y_2:-y_3:y_4).
$$
Define two line bundles associated to the toric divisors $E_1= 3D_{\rho_4}, D_2= 2E_{\rho_4}$. Here, we can write the split vector bundle $\Upsilon_{-D_1,-D_2}$ with rays generated by $\nu_{=1}$. We can compute from $w$ that the functions
$$
g_1:= y_1^2y_2 + y_2^3+  y_3^2y_4 +  y_4^3 +  y_4y_5^2, \quad g_2:=y_1^2 + y_2^2 + y_3^2 + y_4^2 + y_5^2
$$
are global sections of $E_1$ and $E_2$, respectively.

Let $H_{\Upsilon,R}$ be the subgroup of $S_\nu$ corresponding to $\Upsilon(1) \subset \nu$ and $R=\{v_6, v_7\}$ and $\bar w$ the potential corresponding to setting $x_8$ to 1. We can see that 
$$
\dabs{[U_{\Upsilon_{-E_1,-E_2}}, H_{\Upsilon,R} \times \gm, \bar w]} \cong \dbcoh{\mathcal{Z}},
$$
where $\mathcal{Z}'$ is the smooth stacky complete intersection
$$
\mathcal{Z}' : = Z(g_1, g_2) \subseteq [\P^4 / \Z_2].
$$
Here, $\mathcal{Z}'$ is a 2-dimensional stack with a 2-torsion canonical bundle.  By Corollary~\ref{thm:comparegeometric}(a), there is a fully faithful functor 
$$
\dbcoh{\mathcal{Z}'} \longrightarrow \dbcoh{\mathcal{Z}}.
$$

To compute the semi-orthogonal decomposition of $ \dbcoh{\mathcal{Z}}$, we first must state the GIT problem associated to $\nu$. We have that $X:= \A^8$ with variables $x_i$ and can compute that $S_\nu = \gm^2 \times \Z_2$. We summarize the weights of each variable with the following table:
\begin{center}\begin{tabular}{c|c}
Coordinates & Weight in $\gm^2\times \Z_2$  \\ \hline
$x_1, x_3, x_5$ & $(1,1,1)$ \\
$x_2, x_4$ & $(1,1,0)$\\
$x_6$ & $(-1,0,0)$ \\
$x_7$ & $(0,1,0)$ \\
$x_8$ & $(-2,-3,0)$
\end{tabular}\end{center}
The secondary fan for this action of $S_\nu$ is two-dimensional and is pictured below:
\renewcommand{\thefigure}{\thesection.\arabic{figure}}
\begin{center}
{ \begin{tikzpicture}
\draw[-latex] (0,0) -- (2,2) ;
\draw[-latex] (0,0) -- (0,2) ;
\draw[-latex] (0,0) -- (-2,0) ;
\draw[-latex] (0,0) -- (-4/3,-2);
\foreach \x in  {0}
\node[fill=black,draw=black,circle,inner sep=2pt] at (0,0) {};
\node[label=below:{$\Sigma_+$}] at (1/3,0) {};
\node[fill=black,draw=black,circle,inner sep=1pt, label=above:{$-K$}] at (4/3,2) {};
\node[label=above:{$\Sigma_-$}] at (2/3,4/3) {};
{\color{red} \draw[-latex] (2,4/3) -- (4/3,2) ; \node[label=below:{$\gamma$}] at (4/3, 2) {};}
\end{tikzpicture}}
\captionof{figure}{GIT Fan for the $\gm^2$ action}\label{fig: Subsec6.2}
\end{center}

In Figure~\ref{fig: Subsec6.2}, the chambers $\Sigma_-$ and $\Sigma_+$ corresponds to the category $\dabs{[U_{\Upsilon_{-D}}, S_\nu \times \gm, w]}$ and $\dabs{[U_{\Upsilon_{-E_1,-E_2}}, H_{\Upsilon,R} \times \gm, \bar w]}$, respectively. The wall corresponds that the chamber shares corresponds to the one-parameter subgroup $\lambda: \gm \rightarrow S_\nu$ corresponding to the element $(1,-1)$. The fixed locus of $\lambda$ is $Z(x_6,x_7,x_8)$ where the semistable locus is the open set $\A^8 \setminus Z(x_1,\ldots, x_5)$, hence $U_0 = Z(x_6,x_7,x_8) \setminus Z(x_1,\ldots, x_5)$. One can compute that $S_0 = S_\nu / \lambda(\gm) \cong \gm \times \Z_2$ acting with weights $(1,1)$ on $x_1, x_3,$ and $x_5$ and $(1,0)$ on $x_2$ and $x_4$. The induced section $w_0 = 0$ as $x_8$ divides $w$.

We can compute that $\mu = - \sum_{v_i \in \nu} \langle (1,-1), v_i\rangle = 1$. By Theorem~\ref{thm: BFK toric version}, we then have that 
$$
\dabs{[U_{\Upsilon_{-D}}, S_\nu \times \gm, w]} = \langle \dabs{[U_0, S_0 \times \gm, w_0]}, \dabs{[U_{\Upsilon_{-E_1,-E_2}}, H_{\Upsilon,R} \times \gm, \bar w]}\rangle.
$$
Using Theorem~\ref{prop: ToricStackHirano}, this simplifies to:
$$
\dbcoh{\mathcal{Z}} = \langle \dabs{[U_0, S_0 \times \gm, 0]}, \dbcoh{\mathcal{Z}'}\rangle.
$$
Since the $\gm$ factor of $S_0 \times \gm$ acts trivially on $U_0$, by Proposition 2.1.6 of \cite{BDFIK17} or Proposition 1.2.2 of \cite{PV}, we have that 
\begin{equation*}\begin{aligned}
\dabs{[U_0, S_0 \times \gm, 0]} = \dbcoh{[\P^4/\Z_2]} & = \langle \O(0,0), \O(1,0), \O(2,0), \O(3,0), \O(4,0), \\
&\qquad \O(0,1), \O(1,1), \O(2,1), \O(3,1), \O(4,1)\rangle.
\end{aligned}\end{equation*}

In conclusion, we can combine the last two lines and use a mutation to say that there is a semi-orthogonal decomposition 
\begin{align*}
\dbcoh{\mathcal{Z}} = \langle  \dbcoh{[\P^4/\Z_2]}, \dbcoh{\mathcal{Z}'} \rangle
& = \langle \dbcoh{\mathcal{Z}'}, E_1, ..., E_{10} \rangle
\end{align*}
where $E_1, ..., E_{10}$ are exceptional objects.

\subsection{Singular Cubic $(3n+1)$-folds}\label{subsec:SingCubics}

In this section, we apply our results to demonstrate Example~\ref{SingCubic} from the introduction. Take $n$ to be a positive integer. Consider the cubic $(3n+1)$-fold $\mathcal{Z}_{\text{sing}} $ given by the equation 
$$
\sum_{i=1}^n x_i f_i(x_{n+1}, \ldots, x_{3n+3}) + x_{3n+4}f_0(x_{n+1}, \ldots, x_{3n+3}).
$$
 In the case $n=1$, the cubic fourfold was singular at a point, namely, at $P = (1, 0,0,0,0,0)\in \P^{5}$. This case was studied by Kuznetsov in \cite{Kuz10}.  In our generalization, the cubic $(3n+1)$-fold is singular in a $(n-1)$-dimensional hyperplane $\{x_{n+1} = \ldots = x_{3n+3} = 0\}$. 

Recall by Orlov's theorem we have a semi-orthogonal decomposition 
\begin{equation}\label{eqn:OrlovSing}
{\dbcoh{\mathcal{Z}_{\text{sing}}}} = \langle \mathcal{A}, \O, \ldots, \O(3n-1)\rangle.
\end{equation}
Here the subcategory $\mathcal{A}$ is not homologically smooth, hence is not a Calabi-Yau category, but has a crepant categorical resolution.

We prove that a crepant categorical resolution of $\mathcal{A}$ is geometric. This is achieved by interpreting $\mathcal A$ as the absolute derived category of a Landau-Ginzburg model. We can also find a Landau-Ginzburg model interpretation of the crepant categorical resolution of $\dbcoh{\mathcal{Z}_{\text{sing}}}$. The details of this will be provided in the exposition and proofs below. For now, we have the following summary:

\begin{proposition}\label{SingCubics} There is a chain of fully faithful functors
$$
\dbcoh{\mathcal{Z}_{CY}} \longrightarrow \widetilde{\dbcoh{\mathcal{Z}_{\text{sing}}}} \longrightarrow \dbcoh{\widetilde{\mathcal{Z}_{\text{sing}}}},
$$
where
\begin{enumerate}
\item $\mathcal{Z}_{CY}$ is the $(n+1)$-dimensional Calabi-Yau complete intersection in $\P^{2n+2}$ given by one generic cubic $f_0$ and $n$ generic quadrics $f_1, \ldots, f_n$ and its derived category $\dbcoh{\mathcal{Z}_{CY}}$ is a crepant categorical resolution of the category $\mathcal{A}$ in Equation~\eqref{eqn:OrlovSing},
\item $\widetilde{\dbcoh{\mathcal{Z}_{\text{sing}}}} $ is a crepant categorical resolution of the derived category of a singular cubic $(3n+1)$-fold $\mathcal{Z}_{\text{sing}} $ given by the equation 
$$
\sum_{i=1}^n x_i f_i(x_{n+1}, \ldots, x_{3n+3}) + x_{3n+4}f_0(x_{n+1}, \ldots, x_{3n+3}),
$$
\item $\widetilde{\mathcal{Z}_{\text{sing}}}$ is the blowup of $\mathcal{Z}_{\text{sing}}$ along the hyperplane $\{x_{n+1}= \ldots = x_{3n+3} = 0\}$. 
\end{enumerate}
\end{proposition}

\begin{remark}
Remark~\ref{rmk: CCR} pointed out a difference between our definition of a categorical resolution of singularities and the one in \cite{Kuz08}.  The final fully-faithful functor in the above Proposition guarantees that $\widetilde{\dbcoh{\mathcal{Z}_{\text{sing}}}}$  is a crepant categorical resolution in the sense of Ibid.\ as well.
\end{remark}

First we will describe the three distinct factorization categories in the same toric GIT problem. Then we will show that they all correspond to the categories in Proposition~\ref{SingCubics} above.

We follow the notation set up in Sections~\ref{sec:3} and~\ref{GorCones} above. Let $N = \Z^{3n+3}$, with elementary basis vectors $e_i$ and $M$ be the dual lattice. Now, consider the geometric point collection $\nu = \{v_1, \ldots, v_{3n+4}, a\}$ in $N$ where
\begin{equation}\begin{aligned}
v_i &= e_i \text{ for $1 \leq i \leq 3n+2$} \\
v_{3n+3} &= - \sum_{i=1}^{3n+2} e_i + 3 e_{3n+3} \\
v_{3n+4} &= - \sum_{i=1}^n e_i + e_{3n+3} \\
a &= e_{3n+3}.
\end{aligned}\end{equation}

The cone $\sigma : = \op{Cone}(\nu)$ is almost Gorenstein with $\mathfrak{m} = (1, \ldots, 1, n+1)$.  Here, $\nu_{=1}= \{ v_i\}$ and $A = \{a\}$. We have that $\langle \mathfrak{m}, a \rangle = n+1 > 1$.  We compute that $S_\nu = \gm^2$ acts on $X:=\A^{3n+5}$ by the weights in the following table: 

\begin{center}\begin{tabular}{c|c}
Coordinates & Weight of $\gm^2$  \\ \hline
$x_1, \ldots, x_n$ & $(1,1)$ \\
$x_{n+1}, \ldots, x_{3n+3}$ & $(1,0)$ \\
$x_{3n+4}$ & $(0,1)$ \\
$u$ & $(-3,-1)$
\end{tabular}\end{center}

\noindent Let $R_1 = \{a\}$ and let $R_2 = \{v_1, \ldots v_n, v_{3n+4}\}$. That is, the $R$-charge $\gm$-action associated to the subset $R_1$ denoted by $(\gm)_{R_1}$ acts with weights 0 on the $x_i$ and with weight 1 on $u$. Analogously, $(\gm)_{R_2}$ acts with weights 0 on $u, x_{n+1}, \ldots, x_{3n+3}$ and 1 on $x_1, \ldots, x_n, x_{3n+4}$. Recall that, by Lemma~\ref{lem: Rcharge fixer upper}, there is a stack isomorphism between different choices of R-charge.

The secondary fan for this action of $S_\nu$ is two-dimensional and is pictured in Figure~\ref{fig: Singular Cubics GIT fan}. 

\renewcommand{\thefigure}{\thesection.\arabic{figure}}
\begin{center}
\begin{tikzpicture}
\draw[-latex] (0,0) -- (3,3) ;
\draw[-latex] (0,0) -- (3,0) ;
\draw[-latex] (0,0) -- (-3,-1) ;
\draw[-latex] (0,0) -- (0,3);
\foreach \x in  {0}
\node[fill=black,draw=black,circle,inner sep=2pt] at (0,0) {};

\node[label=below:{$\Gamma_4$}] at (0,-.5) {};
\node[fill=black,draw=black,circle,inner sep=1pt, label=above:{$-K$}] at (3,1) {};
\node[label=above:{${\Gamma}_1$}] at (-1,1) {};
\node[label=above:{${\Gamma}_2$}] at (.5, 1) {};
\node[label=above:{${\Gamma}_3$}] at (2,1) {};
{\color{red} \draw[-latex] (-.5,1.85)--(.5, 1.85) ;
\node[label=above:{$\gamma_{12}$}] at (.5, 1.77) {};}
\end{tikzpicture}
\captionof{figure}{GIT Fan for the $\gm^2$ action}\label{fig: Singular Cubics GIT fan}
\end{center}
We can compute the relevant irrelevant ideals 
\begin{equation}\begin{aligned}
\mathcal{I}_{\Gamma_1} &= \langle ux_1, \ldots, ux_n, ux_{3n+4}\rangle\\
\mathcal{I}_{\Gamma_2} &= \langle ux_1, \ldots, ux_n, x_{3n+4}x_1, \ldots, x_{3n+4}x_{3n+3}\rangle\\
\mathcal{I}_{\Gamma_3} &= \langle x_{n+1}, \ldots, x_{3n+3}\rangle \langle x_{3n+4}, x_1,\ldots, x_n\rangle\\
\mathcal{I}_{\Gamma_4} &= \langle ux_{n+1}, \ldots, u x_{3n+3}\rangle.
\end{aligned}\end{equation}

A generic superpotential $w$ is of the form
$$
w = u \left( \sum_{i=1}^n x_i f_i(x_{n+1}, \ldots, x_{3n+3}) + x_{3n+4}f_0(x_{n+1}, \ldots, x_{3n+3})\right)
$$
where $f_0$ is a cubic and $f_1, \ldots, f_n$ are quadrics. 

For each chamber, there is an open set $U_i = \A^{3n+5} \setminus Z(\mathcal{I}_{\Gamma_i})$ so that there is a factorization category 
$$
\dabs(U_i, S_\nu \times (\gm)_{R_1}, w)
$$
associated to each chamber $\Gamma_i$. 

\begin{proof}[Proof of Proposition~\ref{SingCubics}] 
By Theorem~\ref{thm: BFKVGIT}, we know that there is a poset structure for which factorization category has a fully faithful functor into another. Namely, we have:
\begin{equation}\begin{aligned}\label{posetSingCubics}
\dabs(U_1, S_\nu \times (\gm)_{R_1}, w) &\cong \dabs(U_4, S_\nu \times (\gm)_{R_2}, w) \\
&\downarrow \\ 
 \dabs(U_2, &S_\nu \times (\gm)_{R_1}, w) \\ &\downarrow \\ \dabs(U_3, &S_\nu \times (\gm)_{R_1}, w).
\end{aligned}\end{equation}

By providing equivalences to the geometric categories specified in the proposition, part (1) is proven by combining Propositions~\ref{prop:Gamma4Sing} and~\ref{SingGamma12}, part (2) by Proposition~\ref{SingGamma12} and part (3) by Proposition~\ref{Gamma3Sing} below.
\end{proof}

We will go through each chamber systematically, explaining their geometric content. Note that the R-charge for the $\Gamma_4$ chamber changed as the bundle coordinates will change in our geometric interpretation. Here, we can show that
$$
\dabs(U_1, S_\nu \times (\gm), w)  \cong \dbcoh{\mathcal{Z}_{CY} }
$$
via the chamber $\Gamma_4$.

\begin{proposition}\label{prop:Gamma4Sing}
The category $\dabs(U_1, S_\nu \times (\gm), w)$ is equivalent to the derived category $\dbcoh{\mathcal{Z}_{CY} }$ where $\mathcal{Z}_{CY} $ is the $(n+1)$-dimensional Calabi-Yau complete intersection $Z(f_0,\ldots, f_n)$ in $\P^{2n+2}$ defined by one cubic $f_0$ and $n$ quadrics $f_1, \ldots, f_n$.
\end{proposition}

\begin{proof}
First recall that
$$
\dabs(U_1, S_\nu \times (\gm), w)  \cong \dabs(U_4, S_\nu \times (\gm), w).
$$
In the chamber $\Gamma_4$, note that the fan $\Sigma_4$ corresponding to this chamber has the rays generated by $v_1, \ldots, v_n, v_{3n+4}$ as generators for all maximal cones. We then can take the projection $\pi: \Z^{3n+3} \rightarrow \Z^{3n+3} / \langle e_1,\ldots, e_n, e_{3n+4} - e_1-\ldots - e_n\rangle \cong \Z^{2n+2}$. Denote by $\Psi_4$ the fan generated by the image under $\pi$ of the cones in $\Sigma_4$. Then $\Psi_4$ is the standard fan for $\P^{2n+2}$. One can check that 
$$X_{\Sigma_4} = \op{tot}(\O_{\P^{2n+2}}(-3) \oplus \O_{\P^{2n+2}}(-2)^{\oplus n}).$$
Let $\mathcal{Z}_{CY}$ denote the complete intersection
$$
\mathcal{Z}_{CY} = Z(f_0, f_1, \ldots, f_n) \subseteq \P^{2n+2}.
$$
Since the $f_i$ are generic, we have that $\mathcal{Z}_{CY}$ is a smooth stack. We have the equivalence 
\begin{equation}\label{DEquivSingCubicCY}
\dabs(U_4, S_\nu \times (\gm)_{R_2}, w) \cong \dbcoh{\mathcal{Z}_{CY}}.
\end{equation}
Moreover, by Corollary~\ref{thm: fractional CY}, since $\Sigma_4(1) = \nu_{=1}$, we have that $\dbcoh{\mathcal{Z}_{CY}}$ is a Calabi-Yau category of dimension
$$
 - 2\sum_{a \in \nu_{\neq 1}}\langle \mathfrak{m}, a \rangle +2|\nu_{\neq 1}| - 2|R_2| + \dim N_{\R} = -2(0) + 0 - 2(n+1) + (3n+3) = n+1.
$$
\end{proof}

\begin{proposition}\label{SingGamma12}
The following hold:
\begin{enumerate}
\item The category $\dabs(U_1, S_\nu \times (\gm)_{R_1}, w)$ is a crepant categorical resolution of the Calabi-Yau category $\mathcal{A}$ in Equation~\ref{eqn:OrlovSing}. 
\item The category $ \dabs(U_2, S_\nu \times (\gm)_{R_1}, w)$ is a crepant categorical resolution of the category $\dbcoh{\mathcal{Z}_{\text{sing}}}$ in Equation~\ref{eqn:OrlovSing}.
\end{enumerate}
\end{proposition}

\begin{proof}
This follows from Lemmas~\ref{lem:SingCubicsCCR} and~\ref{lem:geomInterpGam12} below.
\end{proof}

The idea of the proof of this proposition is to show that $U_1$ and $U_2$ correspond to partial compactifications of gauged Landau-Ginzburg models corresponding to the Orlov theorem described above in Subsection~\ref{subsec: Orlov}. We first will recall the necessary data from that subsection and will then use the machinery created in Section~\ref{sec: CCR} to prove the lemma. 

We define subsets of $U_1$ and $U_2$. Consider the subideals
\begin{equation}\begin{aligned}
\mathcal{J}_{\Gamma_1} &:= \langle ux_{3n+4}\rangle \subset \mathcal{I}_{\Gamma_1}; \\
\mathcal{J}_{\Gamma_2} &:= \langle x_{3n+4}x_1, \ldots, x_{3n+4}x_n\rangle \subset \mathcal{I}_{\Gamma_2}.
\end{aligned}\end{equation}
Now we have two new open subsets $V_i := \A^{3n+5} \setminus Z(\mathcal{J}_{\Gamma_i})$.

In $X' = \A^{3n+4}$ with variables $x_1, \ldots, x_{3n+3}, u$, consider the ideals $\mathcal{I}'_1 = \langle u\rangle$ and $\mathcal{I}'_2 = \langle x_1, \ldots, x_{3n+3}\rangle$ and the open sets $U_i' = X' \setminus Z(\mathcal{I}'_i)$. Let $\gm$ act with weight $1$ on $x_i$ and weight $-3$ on $u$.  By Lemma~\ref{lem: stack iso}, we have a stack isomorphism
$$
\left[ \bigslant{V_i}{\gm^2\times (\gm)_{R_1}}\right] = \left[\bigslant{U_i'}{\gm \times (\gm)_{R_1}}\right].
$$
Define the superpotential
$$
\bar w  = u \left( \sum_{i=1}^n x_i f_i(x_{n+1}, \ldots, x_{3n+3}) + f_0(x_{n+1}, \ldots, x_{3n+3})\right)
$$
where $f_0$ is a cubic and $f_1, \ldots, f_n$ are quadrics. This is a specialization of $w$ where $x_{3n+4}$ is set to one. 

\begin{lemma}\label{lem:SingCubicsCCR}
The category $\dabs(U_i, S_\nu \times (\gm)_{R_1}, w)$ is a crepant categorical resolution of the  category $\dabs(U_i', \gm \times (\gm)_{R_1}, \bar w)$ for $i = 1,2$.
\end{lemma}

\begin{proof}

Consider the open immersion $V \hookrightarrow U$, where 
\begin{equation}\begin{aligned}
V &= X \setminus Z(x_{3n+4}),\\
U &= X \setminus Z(ux_1, \ldots, ux_n, x_{3n+4}).
\end{aligned}\end{equation} 
A direct computation shows that the ideal $\langle ux_1, \ldots, ux_n, x_{3n+4}\rangle$ is the irrelevant ideal associated to the cone in the GIT fan that is the common face between the chambers $\Gamma_1$ and $\Gamma_2$. The path $\gamma_{12}$ in Figure~\ref{fig: Singular Cubics GIT fan} gives the following stratifications associated to its elementary wall crossing:
\begin{equation}\begin{aligned}
U &= U_1 \sqcup S_-, \quad V= V_1 \sqcup S_-, \quad S_- := Z(u)\cap U, \\
U &= U_2 \sqcup S_+, \quad V= V_2 \sqcup S_+, \quad S_+ := Z(x_1, \ldots, x_{3n+3})\cap U.
\end{aligned}\end{equation}
Note that 
\begin{equation}\begin{aligned}
S_- \cap (U_1 \setminus V_1) &= \varnothing \text{ and}\\
S_+ \cap (U_2 \setminus V_2) &= \varnothing,
\end{aligned}\end{equation}
hence the immersions are compatible with the elementary wall crossing. By Theorem~\ref{theorem: CCRTotal}, we have that $\dabs(U_i, S_\nu \times (\gm), w)$ is a crepant categorical resolution of $\dabs(U_i', \gm \times (\gm)_{R_1}, \bar w)$.
\end{proof}

\begin{lemma}\label{lem:geomInterpGam12}
We have the following derived equivalences:
\begin{equation}\begin{aligned}
\dabs(U_1', \gm \times (\gm)_{R_1}, \bar w) &\cong \mathcal{A}; \\
\dabs(U_2', \gm \times (\gm)_{R_1}, \bar w) &\cong \dbcoh{\mathcal{Z}_{\text{sing}}};
\end{aligned}\end{equation}
where $\mathcal{A}$ and $ \dbcoh{\mathcal{Z}_{\text{sing}}}$ are as defined in Equation~\ref{eqn:OrlovSing}.
\end{lemma}

\begin{proof} The GIT problem for $X'$ with the $\gm$-action defined above is the same as that in Subsection~\ref{subsec: Orlov}:
\renewcommand{\thefigure}{\thesection.\arabic{figure}}
\begin{center}
\begin{tikzpicture}
\draw[latex-] (-4,0) -- (4,0) ;
\draw[-latex] (-4,0) -- (4,0) ;
\foreach \x in  {0}
\node[fill=black,draw=black,circle,inner sep=2pt, label=below:{0}] at (0,0) {};
\node[label=above:{$\langle x_1, \ldots, x_{3n+3}\rangle$}] at (2.12,0) {};
\node[label=below:{$\op{tot}\O_{\P^{3n+2}}(-3)$}] at (2.12,0) {};
\node[label=above:{$\langle u\rangle$}] at (-2.12,0) {};
\node[label=below:{$\left[\kappa^{3n+3} / \Z_{3}\right]$}] at (-2.12,0) {};
\end{tikzpicture}
\captionof{figure}{GIT Fan for $\gm$ action}
\end{center}
Recall from Subsection~\ref{subsec: Orlov} that we have a fully faithful functor
$$
\dabs(U_1', \gm \times \gm, \bar w) \longrightarrow \dabs(U_2', \gm \times \gm, \bar w),
$$
and by Theorem~\ref{Hirano}, 
\begin{equation}\label{DEquivZSing}
\dabs(U_2', \gm \times \gm, \bar w) \cong \dbcoh{\mathcal{Z}_{\text{sing}}},
\end{equation}
where 
$$
\mathcal{Z}_{\text{sing}} := Z( \sum_{i=1}^n x_i f_i(x_{n+1}, \ldots, x_{3n+3}) + f_0(x_{n+1}, \ldots, x_{3n+3})) \subset \P^{3n+2}
$$
is the singular cubic $(3n+1)$-fold. The category $\mathcal{A}$ is $\dabs(U_1', \gm \times \gm, \bar w)$.
\end{proof}

\noindent We finish with chamber $\Gamma_3$.

\begin{proposition}\label{Gamma3Sing}
Let $\op{Bl}_Y(\P^{3n+2})$ be the blowup of $\P^{3n+2}$ along the hyperplane $Y$ given by $\{x_{n+1} =\ldots = x_{3n+3} = 0\}$. Denote by $\mathcal{Z}$ the hypersurface stack
$$
\widetilde{\mathcal{Z}_{\op{sing}}} = Z( \sum_{i=1}^n x_i f_i(x_{n+1}, \ldots, x_{3n+2}) + x_{3n+4}f_0(x_{n+1}, \ldots, x_{3n+3})) \subseteq \op{Bl}_Y(\P^{3n+2}).
$$ 
Then we have the equivalence
$$
\dabs(U_3, S_\nu \times (\gm)_{R_1}, w) \cong \dbcoh{\widetilde{\mathcal{Z}_{\op{sing}}} }.
$$
\end{proposition}
\begin{proof}
In the chamber $\Gamma_3$, note that the fan $\Sigma_3$ corresponding to this chamber has the ray generated by $a$ in all maximal cones. We then can consider the projection map $\pi: \Z^{3n+3} \rightarrow \Z^{3n+2}= \Z^{3n+3} / \langle e_{3n+3}\rangle$ which induces a fan $\Psi_3$ that is the image under $\pi$ of all the faces in $\Sigma_3$. Then $X_{\Psi_3} = \op{Bl}_Y(\P^{3n+2})$ where $Y$ is the hyperplane given by $\{x_{n+1} =\ldots = x_{3n+3} = 0\}$. Call the exceptional divisor $E$. Then 
$$X_{\Sigma_3} = \op{tot}(\O_{\op{Bl}_Y(\P^{3n+2})} ( -3H-E)).$$

The equivalence
$$
\dabs(U_3, S_\nu \times (\gm)_{R_1}, w) \cong \dbcoh{\widetilde{\mathcal{Z}_{\op{sing}}} }.
$$
is then immediately obtained by Theorem~\ref{Hirano}.
\end{proof}

\subsection{Degree $d$ $(2d-2)$-folds Containing Two Planes} \label{subsec:TwoPlanes}
 Fix $d\geq 3$. Consider the two planes $P_1 = \{x_{2d-3}=x_{2d-2} =x_{2d-1} = 0\}$ and $P_2 = \{x_1= \ldots = x_{2d-4} = x_{2d} = 0\}$ in $\P^{2d-1}$.  Let $\mathcal{Z}_{\op{sing}}$ be a generic cubic that contains both $P_1$ and $P_2$.  When $d = 3$, the cubic is smooth and this example's rationality was studied by Hassett \cite{Ha00}. When $d>3$, the cubic is singular. 
 
  Recall by Orlov's theorem we have a semi-orthogonal decomposition:
\begin{equation}\label{TwoPlanesSOD}
 \dbcoh{\mathcal Z_{\op{sing}}} = \langle \mathcal{A}, \O, \ldots, \O(d-1)\rangle.
\end{equation}
In the case where $d>3$, $\mathcal Z_{\op{sing}}$ is not smooth so $\mathcal{A}$ is not Calabi-Yau, but has a crepant categorical resolution that is. 

 We will prove that a categorical resolution of $\mathcal{A}$ is geometric. As in the previous subsection, this is achieved by interpreting $\mathcal A$ as the absolute derived category of a Landau-Ginzburg model. We can also find a Landau-Ginzburg model interpretation of the crepant categorical resolution of $\dbcoh{\mathcal{Z}_{\text{sing}}}$. The details of this will be provided in the exposition and proofs below. For now, we summarize our findings in the following way:

\begin{proposition}\label{prop:TwoPlanesCubic}
There is a chain of fully faithful functors
 \[
\begin{tikzcd}
{} & {} & \widetilde{\dbcoh{\mathcal{Z}_2}} \arrow{rd}{} & {} \\
\dbcoh{\mathcal{Z}_{CY}}\arrow{r}{} & \widetilde{\dbcoh{\mathcal{Z}_{\op{sing}}}} \arrow{rd}{}\arrow{ru}{}& {} & \dbcoh{\widetilde{\mathcal Z_{\op{sing}}}}, \\
{} & {} & \widetilde{\dbcoh{\mathcal{Z}_3}} \arrow{ru}{}  & {}
 \end{tikzcd}
 \]
 where:
 \begin{enumerate}
 \item $\mathcal{Z}_{CY}$ is a $(2d-4)$-dimensional Calabi-Yau complete intersection of two polynomials of bidegree $(d-2,2)$ and $(d-1,1)$ in $\P^{2d-4} \times \P^2$ and its derived category $\dbcoh{\mathcal{Z}_{CY}}$ is a crepant categorical resolution of $\mathcal{A}$ in Equation~\eqref{TwoPlanesSOD},
 \item $\widetilde{\dbcoh{\mathcal{Z}_{\op{sing}}}}$ is a crepant categorical resolution of the derived category of the degree $d$ hypersurface $\mathcal{Z}_{\op{sing}}$ in $\P^{2d-1}$ containing the two planes $P_1$ and $P_2$,
 \item $\widetilde{\dbcoh{\mathcal{Z}_2}}$ is a crepant categorical resolution of the derived category of the degree $d$ hypersurface $\mathcal{Z}_{\op{sing}}$ blown up at the plane $P_2$, 
 \item $\widetilde{\dbcoh{\mathcal{Z}_3}}$ is a crepant categorical resolution of the derived category of the degree $d$ hypersurface $\mathcal{Z}_{\op{sing}}$ blown up at the plane $P_1$, and 
 \item $\widetilde{\mathcal Z_{\op{sing}}}$ is the degree $d$ hypersurface $\mathcal{Z}_{\op{sing}}$ blown up at both planes $P_1$ and $P_2$.
 \end{enumerate}
\end{proposition}

 \begin{remark}
As in the previous example, our proposition guarantees that $\widetilde{\dbcoh{\mathcal{Z}_{\op{sing}}}}$, $\widetilde{\dbcoh{\mathcal{Z}_2}}$, and $\widetilde{\dbcoh{\mathcal{Z}_3}}$ are crepant categorical resolutions in the sense of Kuznetsov \cite{Kuz08} as well as our own. 
 \end{remark}
 
 As alluded to previously, the fully-faithful functors in Proposition~\ref{prop:TwoPlanesCubic}  are obtained using comparisons between toric Landau-Ginzburg models. The precise toric setup is as follows. Fix the lattice $N = \Z^{2d}$ with elementary basis vectors $e_i$ and its dual lattice $M$.  Take the geometric point collection $\nu = \{ v_1, \ldots, v_{2d+2}, a\}$ where 
\begin{equation}\begin{aligned}
v_i &= e_i \text { for $1 \leq i \leq 2d-1$} \\
v_{2d} &= -\sum_{i=1}^{2d-1} e_i + d e_{2d} \\
v_{2d+1} &= e_{2d-3} + e_{2d-2} + e_{2d-1} - e_{2d} \\
v_{2d+2} &= -e_{2d-3} - e_{2d-2} - e_{2d-1} + 2 e_{2d} \\
a &= e_{2d}.
\end{aligned}\end{equation}
The cone $\sigma : = \op{Cone}(\nu)$ is almost Gorenstein and $\mathfrak{m} = (1,\ldots, 1, 2)$. Note that the elements in the set $\nu_{=1}: = \{v_i\}$ all pair to one with $\mathfrak{m}$, and $\langle \mathfrak{m}, a \rangle = 2$. Let $R_1 = \{v_{2d+1}, v_{2d+2}\}$ and $R_2 = \{a\}$.

Let  $X:= \A^{2d+3}$ and compute that $S_{\nu} = \gm^3$. We denote the coordinates of $\A^{2d+3}$ by $x_1, \ldots, x_{2d}$ for the lattice points $v_1, \ldots, v_{2d}$ and the coordinates $u_1, u_2, u_3$ for the points $v_{2d+1}, v_{2d+2}, a$. The weights for $S_\nu$ act according to the table below:
\begin{center}\begin{tabular}{c|c}
Coordinates & Weight of $\gm^3$  \\ \hline
$x_1, \ldots, x_{2d-4}, x_{2d}$ & $(1,0,0)$ \\
$x_{2d-3}, x_{2d-2}, x_{2d-1}$ & $(1,0,1)$ \\
$u_1$ & $(0,1,0)$ \\
$u_2$ & $(0,1,1)$ \\
$u_3$ & $(-d,-1,-2)$
\end{tabular}\end{center}

The GIT fan has eight chambers. We describe them explicitly. Let 
\begin{align*}
p_0 & := (\nicefrac{1}{2}, \nicefrac{1}{2},\nicefrac{1}{2}) \\
p_1 & := (1,0,0) \\
p_2 & := (1,0,1) \\
 p_3 & :=(0,1,0) \\
  p_4 & := (0,1,1) \\
   p_5 & := (-d,-1,-2)
\end{align*}
be points in $\R^3$ and
\begin{align*}
\mathcal K_1& :=\langle x_1, \ldots, x_{2d-4}, x_{2d}\rangle \\
 \mathcal K_2& :=\langle x_{2d-3}, x_{2d-2}, x_{2d-1}\rangle
 \end{align*}
  be ideals in $\kappa[x_1,\ldots, x_{2d}, u_1,u_2,u_3]$. 
  
  The following table describes the eight chambers of the GIT fan and the irrelevant ideals corresponding to the unstable locus for each chamber.
\begin{center}\begin{tabular}{c|c|c}
Chamber $\Gamma_i$ & Cone in $\hat G_{\R}$ & Irrelevant ideal $\mathcal{I}_i$ \\ \hline
$\Gamma_1$ & $\op{Cone}(p_5, p_1,p_2)$ & $\langle u_3\rangle \mathcal K_1\mathcal K_2$\\
$\Gamma_2$ & $\op{Cone}(p_0, p_3,p_4)$ & $\langle u_1u_2\rangle \mathcal K_1 + \langle u_1u_2\rangle \mathcal K_2 + \langle u_2u_3\rangle \mathcal K_1 + \langle u_1u_3\rangle \mathcal K_2$ \\
$\Gamma_3$ & $\op{Cone}(p_0, p_1,p_3)$ & $\langle u_1u_2\rangle \mathcal K_1 + \langle u_1\rangle \mathcal K_1\mathcal K_2+ \langle u_2u_3\rangle \mathcal K_1$\\
$\Gamma_4$ & $\op{Cone}(p_0, p_2,p_4)$ & $\langle u_1u_3\rangle \mathcal K_2 + \langle u_1u_2\rangle \mathcal K_2 + \langle u_2\rangle \mathcal K_1\mathcal K_2$\\
$\Gamma_5$ & $\op{Cone}(p_0, p_1,p_2)$ & $\langle u_1,u_2\rangle \mathcal K_1\mathcal K_2$ \\
$\Gamma_6$ & $\op{Cone}(p_5, p_1,p_3)$ & $\langle u_1u_3, u_2u_3\rangle \mathcal K_2$\\
$\Gamma_7$ & $\op{Cone}(p_5, p_2,p_4)$ & $\langle u_1u_3, u_2u_3\rangle \mathcal K_1$\\
$\Gamma_8$ & $\op{Cone}(p_5, p_3,p_4)$ & $\langle u_1u_3\rangle \mathcal K_2 + \langle u_2u_3\rangle \mathcal K_1 + \langle u_1u_2u_3\rangle$ \\
\end{tabular}\end{center}

For $1 \leq i \leq 8$, let $U_i := \A^{2d+3} \setminus Z(\mathcal{I}_i)$ be the semi-stable locus corresponding to each chamber.  Finally, consider a function $w =  \sum_{m \in \Xi} c_mx^m$ for generic choices of constants $c_m \in \kappa$. We can rewrite $w$ in the form 
\begin{equation}\label{TwoPlanesSuperpotential}
w = u_1u_3 f_1(x_1, \ldots, x_{2d}) + u_2u_3f_2(x_1, \ldots, x_{2d}).
\end{equation}
for some polynomials $f_1, f_2$ which is smooth on all of the $U_i$.  Let $\mathcal{D}_i := \dabs( U_i, \gm^3 \times \gm, w)$ be the factorization category associated to the GIT chamber $\Gamma_i$.

 \begin{proof}[Proof of Proposition~\ref{prop:TwoPlanesCubic}]
Using the fact that $\chi_{-K}$ corresponds to the point $(d,1,2)$ in $\hat G_{\R}$, we apply Theorem~\ref{thm: BFKVGIT} to obtain a poset structure for the categories $\mathcal{D}_i$ given by the following diagram of fully-faithful functors:

 \begin{equation}\label{chainoffactorizationsPlanes}
\begin{tikzcd}
{} & {} & \mathcal{D}_3 \arrow{rd}{} & {} \\
\mathcal{D}_1 \cong \mathcal{D}_6 \cong \mathcal{D}_7\cong \mathcal{D}_8 \arrow{r}{} & \mathcal{D}_2  \arrow{rd}{}\arrow{ru}{}& {} & \mathcal{D}_5 \\
{} & {} & \mathcal{D}_4 \arrow{ru}{}  & {}
 \end{tikzcd}
 \end{equation}
 
 The claim is now proven by giving geometric interpretations to the five distinct categories. This is done in Propositions~\ref{TwoPlanesGamma1}, ~\ref{TwoPlanesGamma5}, and~\ref{CCRsPlanes} below.
 \end{proof}
 
The categories $\mathcal{D}_1$ and $\mathcal{D}_5$ are derived categories of algebraic varieties, while $\mathcal{D}_2, \mathcal{D}_3$, and $\mathcal{D}_4$ are crepant categorical resolutions of derived categories of singular varieties. We will describe the categories in order.

\begin{proposition}\label{TwoPlanesGamma5}
Let $\mathcal{Z}_{CY}$ be the zero locus $Z(f_1, f_2) \subseteq \P^{2d-4} \times \P^2$, which is a $(2d-4)$-dimensional Calabi-Yau complete intersection. Then there is an equivalence
 $$\mathcal{D}_1 :=\dabs(U_1, \gm^3 \times (\gm)_{R_1}, w) \cong \dbcoh{\mathcal{Z}_{CY}}.$$
\end{proposition}

\begin{proof}
 Consider the fan $\Sigma_1$ associated to the GIT chamber $\Gamma_1$. It is constructed by taking the cone generated by $v_1, \ldots, v_{2d}$ and then star subdividing along $v_{2d+1}$ and $v_{2d+2}$.  Consider the product of projective spaces $\P^{2d-4} \times \P^2$. Let $H_1$ and $H_2$ be the hyperplane divisors associated to $\P^{2d-4}$ and $\P^2$ respectively. One can compute that 
$$X_{\Sigma_1} \cong \op{tot}(\O_{\P^{2d-4} \times \P^2}(-(d-2)H_1 - 2H_2) \oplus \O_{\P^{2d-4} \times \P^2}(-(d-1)H_1 - H_2).
$$  
Let $\mathcal{Z}_{CY}$ be the zero locus $Z(f_1, f_2) \subseteq \P^{2d-4} \times \P^2$, which is a $(2d-4)$ Calabi-Yau complete intersection. By Theorem~\ref{Hirano}, we have the equivalence
$$
\mathcal{D}_1 \cong \dabs(U_1, \gm^3 \times (\gm)_{R_1}, w) \cong \dbcoh{\mathcal{Z}_{CY}}.
$$ 
\end{proof}

We now move to the crepant categorical resolutions. 

\begin{proposition}\label{CCRsPlanes}
The following hold:
\begin{enumerate}[(a)]
\item The absolute derived category $\dabs(U_1, S_\nu \times (\gm), w)$ is a crepant categorical resolution of $\mathcal A$ given in Equation~\eqref{TwoPlanesSOD}.
\item The absolute derived category $\dabs(U_2, S_\nu \times (\gm), w)$ is a crepant categorical resolution of $\dbcoh{\mathcal{Z}_{\text{sing}}}$. 
\item The absolute derived category $\dabs(U_3, S_\nu \times (\gm), w)$ is a crepant categorical resolution of $\dbcoh{\mathcal{Z}_2}$, where $\mathcal{Z}_2$ is the strict transform of $\mathcal{Z}_{\op{sing}}$ in $\op{Bl}_{P_2}(\P^{2d-1})$.
\item The absolute derived category $\dabs(U_4, S_\nu \times (\gm), w)$ is a crepant categorical resolution of $\dbcoh{\mathcal{Z}_1}$, where $\mathcal{Z}_1$ is the strict transform of $\mathcal{Z}_{\op{sing}}$ in $\op{Bl}_{P_1}(\P^{2d-1})$.
\end{enumerate}
\end{proposition}

\begin{proof}
Consider the open immersion $V \hookrightarrow U$, where 
\begin{equation}\begin{aligned}
V &= X \setminus Z(u_1u_2),\\
U &= X \setminus Z(u_1u_2, u_2u_3\mathcal{K}_1, u_1u_3\mathcal{K}_2).
\end{aligned}\end{equation}

A direct computation shows that the ideal $\langle u_1u_2, u_2u_3\mathcal{K}_1, u_1u_3\mathcal{K}_2\rangle$ is the irrelevant ideal associated to the cone in the GIT fan that is the common face between the chambers $\Gamma_2$ and $\Gamma_8$. The path between these two chambers yields the following stratifications associated to its elementary wall crossing:
\begin{equation}\begin{aligned}
U &= U_8 \sqcup S_-, \quad V= V_8 \sqcup S_-, \quad S_- := Z(u_3)\cap U, \\
U &= U_2 \sqcup S_+, \quad V= V_2 \sqcup S_+, \quad S_+ := Z(x_1, \ldots, x_{2d})\cap U.
\end{aligned}\end{equation}
Note that 
\begin{equation}\begin{aligned}
S_- \cap (U_8 \setminus V_8) &= \varnothing \text{ and}\\
S_+ \cap (U_2 \setminus V_4) &= \varnothing,
\end{aligned}\end{equation}
hence the immersions are compatible with the elementary wall crossing. Consider the gauged Landau-Ginzburg model $(V_2, \gm^3 \times (\gm)_{R_2}, w)$. Consider the affine space $X_{u_1, u_2} = \A^{2d+1}$ found by taking $\op{Spec}(\kappa[x_1, \ldots, x_{2n}, u_3])$. There is a stack isomorphism
$$
\left[ \bigslant{V_2}{\gm^3\times (\gm)_{R_1}}\right] = \left[\bigslant{(X_{u_1,u_2} \setminus Z(x_1,\ldots, x_{2d}))}{\gm \times (\gm)_{R_1}}\right] = \op{tot} \O_{\P^{2d-1}}(-dH),
$$
thus $U_2$ is a partial compactification of  $\op{tot} \O_{\P^{2d-1}}(-dH)$. The superpotential $w$ is an extension of the section $f_1+f_2$ on $ \op{tot} \O_{\P^{2d-1}}(-dH)$. Note that by Hirano's theorem
$$
\dabs(X_{u_1,u_2} \setminus Z(x_1,\ldots, x_{2d}), \gm \times (\gm), u_3(f_1+f_2)) = \dbcoh{\mathcal{Z}_{\op{sing}}},
$$
where $\mathcal{Z}_{\op{sing}}$ is the zero locus of $Z(f_1+f_2) \subset \P^{2d-1}$. Note that since the section $u_1f_1+u_2f_2$ also defines a section of $-dH + E_1 + E_2$ in $\widetilde{\P}$, we know that $\mathcal{Z}_{\op{sing}}$ contains the two planes $x_{2d-3}=x_{2d-2} =x_{2d-1} = 0$ and  $x_1= \ldots = x_{2d-4} = x_{2d} = 0$. 

By Theorem~\ref{theorem: CCRTotal}, we have that $\mathcal{D}_2$ is a crepant  categorical resolution of $\dbcoh{\mathcal{Z}_{\op{sing}}}$. We have two semi-orthogonal decompositions: 
$$
\dbcoh{\mathcal{Z}_{\op{sing}}} = \langle \mathcal{A}, \O, \ldots, \O(d-1)\rangle,
$$
where $\mathcal{A} = \dabs(V_8, \gm^3 \times (\gm)_{R_2}, w)$, and 
$$
\mathcal{D}_2 = \langle \mathcal{D}_8, \O, \ldots, \O(d-1)\rangle.
$$
By Theorem~\ref{theorem: CCRGIT}, we have that 
\begin{equation}
\mathcal{D}_1 \cong \mathcal{D}_8 :=\dabs(U_8, S_\nu \times (\gm), w)
\end{equation}
is a crepant categorical resolution of $\mathcal{A}$. 

Let $\mathcal{Z}_1$ and $\mathcal{Z}_2$ be the resultant varieties from taking $\mathcal{Z}_{\op{sing}}$ blowing up the planes $x_{2d-3}=x_{2d-2} =x_{2d-1} = 0$ and $x_1= \ldots = x_{2d-4} = x_{2d} = 0$ , respectively.  By doing the analogous comparisons between $\Gamma_3$ and $\Gamma_7$ and between $\Gamma_4$ and $\Gamma_6$, one can see that $\mathcal{D}_3$ is a crepant categorical resolution of $\dbcoh{\mathcal{Z}_1}$ and $\mathcal{D}_4$ is a crepant categorical resolution of $\dbcoh{\mathcal{Z}_2}$. 
\end{proof}

We finish with the geometric interpretation of the category $\mathcal{D}_5$.

\begin{proposition}\label{TwoPlanesGamma1}
Consider the two planes $P_1 = \{x_{2d-3}=x_{2d-2} =x_{2d-1} = 0\}$ and $P_2 = \{x_1= \ldots = x_{2d-4} = x_{2d} = 0\}$ in $\P^{2d-1}$.  Let $\mathcal{Z}_{\op{sing}}$ be the cubic $Z(f_1+f_2)$ in $\P^{2d-1}$ where $f_1$ and $f_2$ are the cubics defined in Equation~\eqref{TwoPlanesSuperpotential}. Consider the blow up $\widetilde{\mathcal{Z}_{\op{sing}}}$ of $\mathcal{Z}_{\op{sing}}$ along $P_1$ and $P_2$. Then $\mathcal{Z}_{\op{sing}}$ contains both $P_1$ and $P_2$ and we have the equivalence
$$
\mathcal{D}_5:=\dabs( U_5, \gm^3 \times \gm, w) \cong \dbcoh{\widetilde{\mathcal{Z}_{\op{sing}}}}.
$$
\end{proposition}

\begin{proof} Start with the standard fan for $\P^{2d-1}$, then blow up the hyperplanes $x_{2d-3}=x_{2d-2} =x_{2d-1} = 0$ and  $x_1= \ldots = x_{2d-4} = x_{2d} = 0$ to obtain the variety $\widetilde{\P}$. Note that $\op{Cl}(\widetilde{\P}) = \Z^3$, and it is generated by the hyperplane section $H$ and the exceptional divisors $E_1$ and $E_2$ given by the respective blowups described above. Consider the divisor $D = - dH + E_1 + E_2$. One can check that the fan $\Sigma_5$ is the total space of the line bundle $\O_{\widetilde{\P}}(-D)$. A generic global section of $\O_{\widetilde{\P}}(-D)$ is given by taking $w$ and setting $u_1$ and $u_2$ to one. Let $\widetilde{\mathcal Z_{\op{sing}}}$ be the zero locus $Z(f_1 + f_2) \subseteq \widetilde{\P}$. By Theorem~\ref{Hirano}, there is an equivalence 
$$
 \dabs(U_5, \gm^3 \times (\gm)_{R_2}, w) \cong  \dbcoh{\widetilde{\mathcal Z_{\op{sing}}}}.
$$ 
The fact that $\mathcal{Z}_{\op{sing}}$ contains $P_1$ and $P_2$ is clear from the definition of the divisor $D$.
\end{proof}

 \end{document}